\documentclass[11pt]{article}
\usepackage[utf8]{inputenc}
\usepackage{amsmath,amssymb,amsthm,graphicx}
\usepackage[all]{xy}
\usepackage{xcolor}
\usepackage[top=3cm, bottom=3cm, left=4cm, right=4cm]{geometry}
\title{Toledo invariants of Topological Quantum Field Theories}

\newcommand{\Q}{\mathbb{Q}}
\newcommand{\R}{\mathbb{R}}
\newcommand{\C}{\mathbb{C}}
\newcommand{\Z}{\mathbb{Z}}
\newcommand{\N}{\mathbb{N}}
\renewcommand{\H}{\mathbb{H}}
\renewcommand{\P}{\mathbb{P}}
\newcommand{\A}{\mathcal{A}}
\newcommand{\M}{\mathcal{M}}
\newcommand{\V}{\mathcal{V}}

\renewcommand{\S}{\mathcal{S}}
\newcommand{\Mb}{\overline{\mathcal{M}}}
\newcommand{\T}{\mathcal{T}}
\newcommand{\AT}{\overline{\mathcal T}}
\newcommand{\Tb}{\overline{\mathcal{T}}}
\newcommand{\U}{\mathrm{U}}
\newcommand{\PGL}{\mathrm{PGL}}
\newcommand{\PU}{\mathrm{PU}}
\newcommand{\SU}{\mathrm{SU}}
\newcommand{\SO}{\mathrm{SO}}

\newcommand{\BPU}{\mathrm{BPU}}
\newcommand{\ab}{\mathrm{ab}}
\newcommand{\irr}{\mathrm{irr}}

\newcommand{\Mh}{^h\!\!\M}
\newcommand{\Mbh}{^h\!\!\overline{\M}}
\newcommand{\SL}{\mathrm{SL}}

\newcommand{\E}{\mathcal E}
\renewcommand{\epsilon}{\varepsilon}
\DeclareMathOperator{\id}{Id}
\DeclareMathOperator{\Stab}{Stab}

\DeclareMathOperator{\Mod}{Mod}
\DeclareMathOperator{\Aut}{Aut}
\DeclareMathOperator{\ch}{ch}
\DeclareMathOperator{\sch}{sch}
\DeclareMathOperator{\sign}{sign}
\DeclareMathOperator{\Hom}{Hom}
\DeclareMathOperator{\End}{End}

\DeclareMathOperator{\nodes}{Sing}

\DeclareMathOperator{\tr}{Tr}

\DeclareMathOperator{\ord}{ord}
\DeclareMathOperator{\str}{sTr}
\renewcommand{\phi}{\varphi}
\renewcommand{\tilde}{\widetilde}

\newtheorem{Theorem}{Theorem}
\newtheorem*{Theorem*}{Theorem}
\newtheorem{Proposition}{Proposition}
\newtheorem{Definition}{Definition}
\newtheorem{Lemma}{Lemma}
\newtheorem{Corollary}{Corollary}
\newtheorem*{Corollary*}{Corollary}
\newtheorem{Remark}{Remark}

\begin{document}
\date{}


\author{Bertrand Deroin and
Julien Marché
}

\newcommand{\Addresses}{{
  \bigskip
  \footnotesize

  B.~Deroin, \textsc{CNRS-Laboratoire AGM - Université de Cergy-Pontoise}\par\nopagebreak
  \textit{E-mail address},  \texttt{bertrand.deroin@univ-cyu.fr}

  \medskip

  J.~Marché, \textsc{IMJ-PRG - Sorbonne Universit\'e}\par\nopagebreak
  \textit{E-mail address},  \texttt{julien.marche@imj-prg.fr}

}}

\maketitle


\abstract{We prove that the Fibonacci quantum representations $\rho_{g,n}:\Mod_{g,n}\to \PU(p,q)$ for $(g,n)\in\{(0,4),(0,5),(1,2),(1,3),(2,1)\}$ are holonomy representations of complex hyperbolic structures on some compactifications of the corresponding moduli spaces $\M_{g,n}$. As a corollary, the forgetful map between the corresponding compactifications of \(\mathcal M_{1,3}\) and \(\mathcal M_{1,2}\) is a surjective holomorphic map between compact complex hyperbolic orbifolds of different dimensions higher than one, giving an answer to a problem raised by Siu. 

The proof consists in computing their Toledo invariants: we put this computation in a broader context, replacing the Fibonacci representations with any Hermitian modular functor and extending the Toledo invariant to a full series of cohomological invariants beginning with the signature $p-q$. 

We prove that these invariants satisfy the axioms of a Cohomological Field Theory and compute the $R$-matrix at first order (hence the usual Toledo invariants) in the case of the $\SU_2/\SO_3$-quantum representations at any level.}

\section{Introduction}

\subsection{Motivation} 
The moduli spaces \(\mathcal M_{g,n} \)  of genus \(g\) curves with \(n\) marked points, do not seem to have geometric structures in general, nor their partial compactifications. However, some very interesting curiosities happen in particular cases; leading examples of this kind are the compact type partial compactification of \(\mathcal M_2\) and \(\mathcal M_3\) that carry structures locally modelled on Siegel spaces (via the Jacobian of the curve), or the examples of complex hyperbolic structures on certain partial compactifications of \(\mathcal M_{0,n}\) with \(n\leq 8\), using hypergeometric integrals (see \cite{DM, Thurston} and \cite{McMullen} for further developments and a nice historical treatment to this topic). In all these examples, a key role is played by the holonomy of the geometric structure: a linear representation of the corresponding mapping class group \(\text{Mod} _{g,n} \).    

The original motivation of this work, that emerged while the second author lectured on TQFT in Bordeaux and Paris, see \cite{Bordeaux}, is to investigate whether quantum representations  provide interesting geometric structures on moduli spaces and/or  their partial compactifications: those are representations of the mapping class groups with values in the projective linear group of vector spaces called spaces of conformal blocks. They are associated to the data of a simple compact Lie group \(G\), a level \(\ell\) (a positive integer), and some finite set $\Lambda$ of irreducible representations of $G$ (depending on $\ell$). This theory is extremely rich and have various aspects, one is analytical, based on quantization of character varieties of surface groups, the other one is combinatorial/topological, based on a modular category (constructed from the representation theory of quantum groups or from the Kauffman bracket), we refer to \cite{Bakalov-Kirillov} for a general overview. While the two points of view are equivalent, we will follow the combinatorial/topological road here.

A crucial property, which has been made explicit using the topological viewpoint as in \cite{BHMV} is that the space of conformal blocks is defined over a cyclotomic field $k$ of order \(\ell\) and the image of a quantum representation takes values in the group of projective transformations that preserves a pseudo-hermitian form on the space of conformal blocks defined over \(k\) (they also preserve an integral structure, hence taking values in an arithmetic group, as was proved by Gilmer and Masbaum in \cite{GM}). An interesting consequence is that whence we fix an embedding \(i: k\rightarrow \C\), the representation  gives rise to a representation 
\begin{equation} \label{eq: QR} \rho_{g,n}^i(\lambda_1, \ldots, \lambda_n) : \text{Mod}_{g,n}  \rightarrow \text{PU} (p ,q)\end{equation}
where \(\lambda_1,\ldots,\lambda_n\in \Lambda\) and \(p, q\) are the integers (depending highly on all data) so that \(d=p+q\) is the dimension of the space of conformal blocks and \(\sigma=p-q\) is the signature of the hermitian form.

Denote by $\H^{p,q}$ the Hermitian symmetric space associated to $\PU(p,q)$. The question that motivates this work is: does \(\mathcal M_{g,n} \), or a partial compactification of it, carries a \(\mathbb H^{p,q}\)-structure whose holonomy is given by the quantum representation \eqref{eq: QR}? The existence of geometric structures modelled on other homogeneous spaces of \(\text{PU}(p,q)\) is not addressed here but is certainly very interesting.

It turns out that in general the dimension of $\H^{p,q}$, which is equal to the product \(pq\), is much larger than the dimension of the moduli space \(\mathcal M_{g,n} \), which is \(3g-3+n\). To our knowledge, the only quantum representations where the coincidence of dimension 
\[ pq=3g-3+ n>1\]
holds are Fibonacci representations: the quantum representations associated to the compact Lie group \(\text{SO} (3)\), with level  \(\ell =5\). Notice that we do not use the specific definition of level from conformal field theory: for us it will simply denote the order of the root of unity necessary to carry over the construction. 

Later in the introduction, we provide an elementary construction of Fibonacci representations for the reader which is not familiar with TQFT, but before doing so, we present the main result of the article, namely the computation of Toledo invariants of quantum representations, which can be performed for any Hermitian modular functor. The computation of these invariants in the perspective of geometrization of quantum representations is our fundamental tool, and in some special cases permits, thanks to Siu's rigidity theory, to overcome the lack of naturally defined period maps associated to quantum representations.

\subsection{Toledo invariants of Hermitian modular functors}

A fundamental property of quantum representations of level \(\ell\) is that they map any Dehn twist to an element of order \(\ell\) in the group \(\text{PU} (p,q)\). So they can be thought of as representations defined on the quotient $\Mod_{g,n}^\ell$ of the mapping class group by the group generated by \(\ell\)-th powers of Dehn twists. This group is the orbifold fundamental group of the compact orbifold \(\Mb_{g,n}^{\ell}\) obtained from Deligne-Mumford compactification \(\Mb_{g,n}\) of moduli space by twisting the complex structure along the boundary divisors of \(\Mb_{g,n}\), namely the map \(\Mb _{g,n}^{\ell}\rightarrow \Mb_{g,n}\) is a set-theoretic bijection but in orbifold charts, it ramifies at order \(\ell\) on the boundary divisors. In terms of stacks, this is the \(\ell\)-root stack ramifying along the boundary divisors : it has already been considered in the context of TQFT  \cite{Eyssidieux-Funar} and  in the context of $r$-spin structures, see \cite{Chiodo} Section 2.1. From all this, we can see quantum representations as representations defined on the fundamental group \(\pi_1(\Mb_{g,n}^{\ell})\). 

The main topological information that detects whether a representation \(\rho : \pi_1 (X) \rightarrow \text{PU} (p,q)\) defined on the fundamental group of a compact complex manifold is or isn't the holonomy of a \(\mathbb H^{p,q}\)-structure, is contained in a characteristic class belonging to the second rational cohomology of the manifold, called the Toledo invariant.  This class is the degree two part of a higher cohomology class that is defined in the following way: suppose \(\rho\) lifts to a representation with values in \( \text{U}(p,q)\), and take a decomposition of $\E$, the flat \( \C^{p,q}\)-bundle over \(X\) with monodromy \(\rho\) as an orthogonal sum \(\E=\E^+ \oplus \E^-\) of a positive rank \(p\) subbundle \(\E^+\) and a negative rank \(q\) subbundle \(\E^-\). Then the higher cohomology class is defined by 
\[ \sch(\rho) := \ch (\E^+ ) - \ch (\E^-) \in H^* (X, \mathbb Q)\]
where \( ch\) is the Chern character and $sch$ stands for {\em super/signed Chern character}. These invariants have been introduced in the context of Hermitian K-theory with applications to algebraic topology, see \cite{Mishchenko}. If \(X\) is an orbifold, we can still define a higher Toledo class in the cohomology of the underlying topological space of \( X\) with rational coefficients. In the case of a quantum representation \(\rho : \pi _1 (\Mb _{g,n} ^{\ell}) \rightarrow \text{PU} (p,q)\), we thus have a class \(\sch (\rho) \in H^* (\Mb_{g,n}, \mathbb Q)\simeq H^* (\Mb_{g,n}^{\ell}, \mathbb Q) \). 

Recall that the representation $\rho$ depends on $g,n$ of course, but also on the particular embedding $i:k\to \C$ and the colors $\lambda_1,\ldots,\lambda_n$ attached to the marked points. Setting $V=\Q[\Lambda]$, we define a multilinear map $\omega_{g,n}:V^n\to H^*(\Mb_{g,n},\Q)$ by putting 
$$\omega_{g,n}(\lambda_1,\ldots,\lambda_n)=\sch(\rho_{g,n}^i(\lambda_1,\ldots,\lambda_n)).$$

\begin{Theorem*}
For any hermitian modular functor, the family $\omega_{g,n}$ defined above satisfies the axioms of a Cohomological Field Theory (CohFT).
\end{Theorem*}

A particular interesting instance of this is that the degree \(0\) part of this CohFT, namely the signature \(\sigma\), defines on $V$ a structure of Frobenius algebra which have not been studied before as far as we know. We prove that for the $\SU_2/\SO_3$-TQFTs at any level, these $\Q$-algebras are semi-simple. 

To illustrate this theorem, consider the example of the Fibonacci representation where $k=\Q(q)$ and $q^5=1$. 
In this case, $\Lambda$ has two elements called the trivial color and the non-trivial color and $V=\Q[\Lambda]$ is a quadratic number field. 

When $i(q)=e^{4i\pi/5}$, the representation $\rho^i_{g,n}$ becomes unitary, hence the higher cohomological invariants vanish: we have $\sch(\rho^i_{g,n})=d_{g,n}$, the dimension of the representation. Then $V=\Q(\phi)$ as an algebra where $\phi=-q-q^{-1}$ is the golden ratio and 
$$ d_{g,n}=\tr_{\Q(\phi)/\Q}(\phi^n(2+\phi)^{g-1}).$$

When $i(q)=e^{2i\pi/5}$, the representation $\rho^i_{g,n}$ is no longer unitary. We have in this case $V=\Q(j)$ where $j^2+j+1=0$ and the signature of the representation is 
$$\sigma_{g,n}=\tr_{\Q(j)/\Q}(j^n(2+j)^{g-1}).$$

We collect in the following table the complete signature $p|q$ of $(V,h)$:

\vspace{0,5cm}

\begin{tabular}{|c|c|c|c|c|c|c|c|}
\hline
&$n=0$& $n=1$ & $n=2$ & $n=3$ & $n=4$ & $n=5$ & $n=6$\\
\hline
$g=0$ & $1|0$ & $0|0$ & $0|1$ & $1|0$ & $1|1$ & $1|2$ & $3|2$ \\ 
\hline
$g=1$ & $2|0$ & $0|1$ & $1|2$ & $3|1$ & $3|4$ & $5|6$ & $10|8$ \\
\hline
$g=2$ & $4|1$ & $1|4$ & $5|5$ & $9|6$ & $11|14$ & $20|20$ & $34|31$ \\
\hline
$g=3$ & $9|6$ & $7|13$ & $19|16$ & $29|26$ & $42|48$ & $74|71$ & $119|116$\\
\hline
\end{tabular}
\vspace{0,5cm}

This table makes clearer the terminology Fibonacci. We will provide explicit formulas for the three Frobenius algebras arising at level $7$. At prime level $\ell$, these algebras seem to be particularly interesting number fields and deserve a further study. They also provide a conceptual explanation to a phenomenon observed by Funar, Pitsch and Costantino see \cite{FP}.

The general theorem opens the door to a computation of the cohomological invariants using the Givental-Teleman classification theorem, see \cite{Pandharipande}. In particular, we give in this article an algorithm for computing the $R$-matrix at first order that we implemented with Sage. The computation in the Fibonacci case $(\ell=5)$ can be performed by hand and is already much more complicated when $\ell=7$. 

The computation of this $R$-matrix reduces to the computation of Toledo invariants of representations of triangle groups in $\PU(p,q)$. We provide in Appendix \ref{Meyer} a formula generalizing Meyer's formula for the signature of $4$-manifolds which reduces the computation to the signature of some explicit Hermitian matrices. 

These CohFTs look particularly interesting: it seems difficult to compute the $R$-matrices at higher order or to find the spectral curve encoding it through Topological Recursion. Notice that CohFTs already appeared in the context of modular functors in \cite{MOPPZ,ABO} where the authors computed the Chern character of the vector bundle of conformal blocks over $\Mb_{g,n}$. 
We stress that our construction is indeed different as it highly relies on the Hermitian structure, which plays no role in the aforementioned articles.

\subsection{Interlude: a quick construction of Fibonacci representations}

We sketch here a construction which is detailed in \cite{Bordeaux} for the case of $\SU_2$-modular functors. The Fibonacci case which is treated here is indeed different, but the proofs are similar. We include it so that the unfamiliar reader get a flavour of it: we refer to \cite{BHMV} for a full account of these constructions. 

Let $S$ be a surface of genus $g$ and $P\subset S$ be a finite subset of punctures. We set $k=\Q(q)$ to be the cyclotomic field of order $5$ where $q^5=1$ and define the elements $\phi=-q-q^{-1}$ and $A=-q^3$. 

We define $\A(S,P)$ as the $k$-vector space generated by isotopy classes of finite graphs $G$ embedded in $S\setminus P$ (or equivalently $1$-dimensional sub-cell-complexes) up to the following five local moves. 

\begin{enumerate}
    \item Contraction-Deletion relation: $[G/e]=[G]+[G\setminus e]$.
    \begin{figure}[htbp]
    \centering
    \def\svgwidth{7cm}
\begingroup%
  \makeatletter%
  \providecommand\color[2][]{%
    \errmessage{(Inkscape) Color is used for the text in Inkscape, but the package 'color.sty' is not loaded}%
    \renewcommand\color[2][]{}%
  }%
  \providecommand\transparent[1]{%
    \errmessage{(Inkscape) Transparency is used (non-zero) for the text in Inkscape, but the package 'transparent.sty' is not loaded}%
    \renewcommand\transparent[1]{}%
  }%
  \providecommand\rotatebox[2]{#2}%
  \newcommand*\fsize{\dimexpr\f@size pt\relax}%
  \newcommand*\lineheight[1]{\fontsize{\fsize}{#1\fsize}\selectfont}%
  \ifx\svgwidth\undefined%
    \setlength{\unitlength}{285.63540481bp}%
    \ifx\svgscale\undefined%
      \relax%
    \else%
      \setlength{\unitlength}{\unitlength * \real{\svgscale}}%
    \fi%
  \else%
    \setlength{\unitlength}{\svgwidth}%
  \fi%
  \global\let\svgwidth\undefined%
  \global\let\svgscale\undefined%
  \makeatother%
  \begin{picture}(1,0.44103904)%
    \lineheight{1}%
    \setlength\tabcolsep{0pt}%
    \put(0,0){\includegraphics[width=\unitlength,page=1]{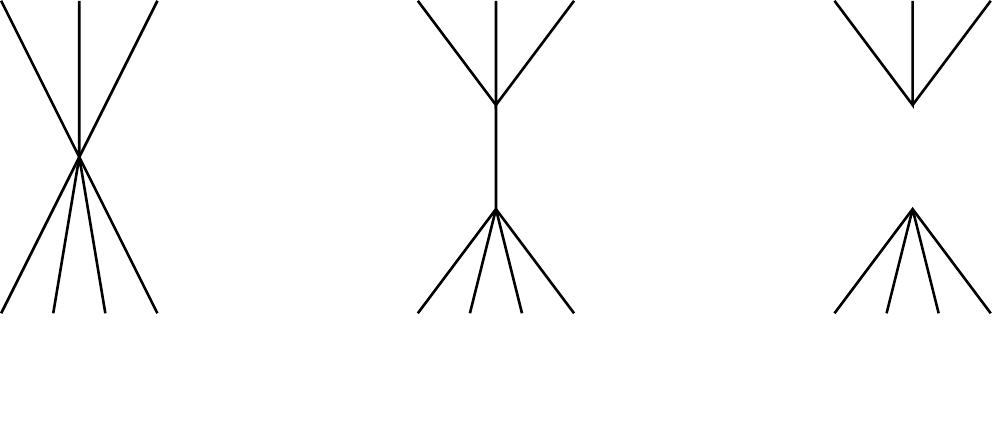}}%
    \put(0.02434753,0.05){\color[rgb]{0,0,0}\makebox(0,0)[lt]{\lineheight{1.25}\smash{\begin{tabular}[t]{l}$G/e$\end{tabular}}}}%
    \put(0.49574163,0.05){\color[rgb]{0,0,0}\makebox(0,0)[lt]{\lineheight{1.25}\smash{\begin{tabular}[t]{l}$G$\end{tabular}}}}%
    \put(0.85809198,0.05){\color[rgb]{0,0,0}\makebox(0,0)[lt]{\lineheight{1.25}\smash{\begin{tabular}[t]{l}$G\setminus e$\end{tabular}}}}%
    \put(0,0){\includegraphics[width=\unitlength,page=2]{contraction-deletion.pdf}}%
    \put(0.46821501,0.27){\color[rgb]{0,0,0}\makebox(0,0)[lt]{\lineheight{1.25}\smash{\begin{tabular}[t]{l}$e$\end{tabular}}}}%
    \put(0,0){\includegraphics[width=\unitlength,page=3]{contraction-deletion.pdf}}%
  \end{picture}%
\endgroup%

\end{figure}
    \item Fusion relation: $[G_\times]=\phi^{-1}[G_{||}]+\phi^{-1}[G_{=}]$.
    \begin{figure}[htbp]
    \centering
    \def\svgwidth{9cm}
\begingroup%
  \makeatletter%
  \providecommand\color[2][]{%
    \errmessage{(Inkscape) Color is used for the text in Inkscape, but the package 'color.sty' is not loaded}%
    \renewcommand\color[2][]{}%
  }%
  \providecommand\transparent[1]{%
    \errmessage{(Inkscape) Transparency is used (non-zero) for the text in Inkscape, but the package 'transparent.sty' is not loaded}%
    \renewcommand\transparent[1]{}%
  }%
  \providecommand\rotatebox[2]{#2}%
  \newcommand*\fsize{\dimexpr\f@size pt\relax}%
  \newcommand*\lineheight[1]{\fontsize{\fsize}{#1\fsize}\selectfont}%
  \ifx\svgwidth\undefined%
    \setlength{\unitlength}{762.81943013bp}%
    \ifx\svgscale\undefined%
      \relax%
    \else%
      \setlength{\unitlength}{\unitlength * \real{\svgscale}}%
    \fi%
  \else%
    \setlength{\unitlength}{\svgwidth}%
  \fi%
  \global\let\svgwidth\undefined%
  \global\let\svgscale\undefined%
  \makeatother%
  \begin{picture}(1,0.19642863)%
    \lineheight{1}%
    \setlength\tabcolsep{0pt}%
    \put(0,0){\includegraphics[width=\unitlength,page=1]{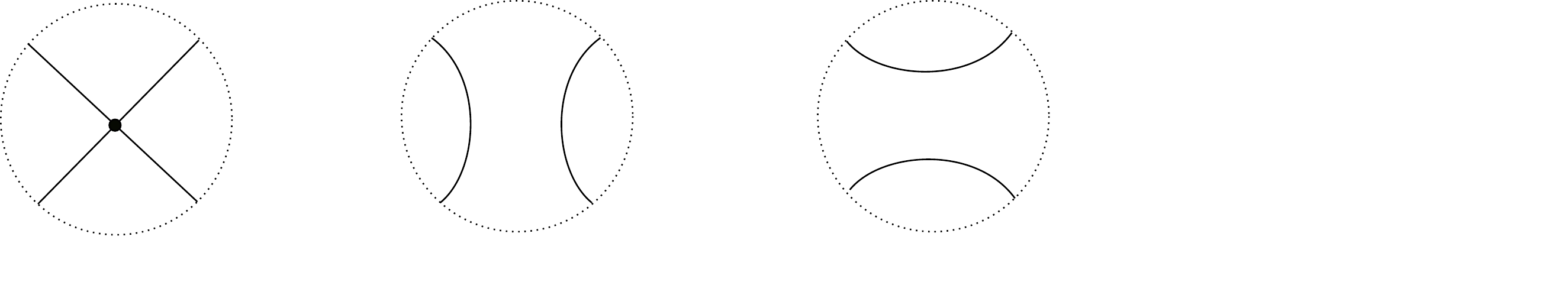}}%
    \put(0.02637126,0.00936214){\color[rgb]{0,0,0}\makebox(0,0)[lt]{\lineheight{1.25}\smash{\begin{tabular}[t]{l}$G_\times$\end{tabular}}}}%
    \put(0.30308046,0.00872734){\color[rgb]{0,0,0}\makebox(0,0)[lt]{\lineheight{1.25}\smash{\begin{tabular}[t]{l}$G_{||}$\end{tabular}}}}%
    \put(0.56455419,0.00809255){\color[rgb]{0,0,0}\makebox(0,0)[lt]{\lineheight{1.25}\smash{\begin{tabular}[t]{l}$G_{=}$\end{tabular}}}}%
    \put(0,0){\includegraphics[width=\unitlength,page=2]{fusion.pdf}}%
    \put(0.82671245,0.00752755){\color[rgb]{0,0,0}\makebox(0,0)[lt]{\lineheight{1.25}\smash{\begin{tabular}[t]{l}$G_{c}$\end{tabular}}}}%
  \end{picture}%
\endgroup%

\end{figure}
    \item Bridge and loop relation: see Figure \ref{fig:pont-boucle}.
    \begin{figure}[htbp]
    \centering
    \def\svgwidth{7cm}
\begingroup%
  \makeatletter%
  \providecommand\color[2][]{%
    \errmessage{(Inkscape) Color is used for the text in Inkscape, but the package 'color.sty' is not loaded}%
    \renewcommand\color[2][]{}%
  }%
  \providecommand\transparent[1]{%
    \errmessage{(Inkscape) Transparency is used (non-zero) for the text in Inkscape, but the package 'transparent.sty' is not loaded}%
    \renewcommand\transparent[1]{}%
  }%
  \providecommand\rotatebox[2]{#2}%
  \newcommand*\fsize{\dimexpr\f@size pt\relax}%
  \newcommand*\lineheight[1]{\fontsize{\fsize}{#1\fsize}\selectfont}%
  \ifx\svgwidth\undefined%
    \setlength{\unitlength}{856.9992522bp}%
    \ifx\svgscale\undefined%
      \relax%
    \else%
      \setlength{\unitlength}{\unitlength * \real{\svgscale}}%
    \fi%
  \else%
    \setlength{\unitlength}{\svgwidth}%
  \fi%
  \global\let\svgwidth\undefined%
  \global\let\svgscale\undefined%
  \makeatother%
  \begin{picture}(1,0.19645565)%
    \lineheight{1}%
    \setlength\tabcolsep{0pt}%
    \put(0,0){\includegraphics[width=\unitlength,page=1]{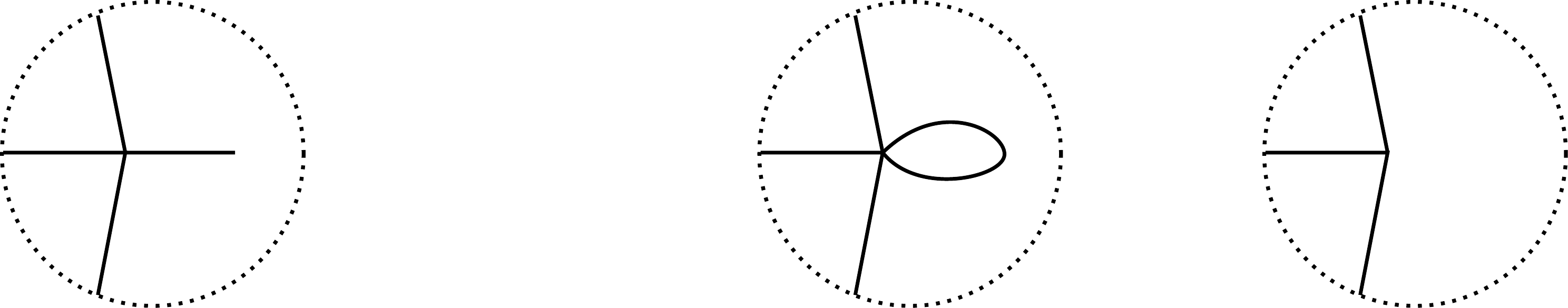}}%
    \put(0.21985781,0.08564706){\color[rgb]{0,0,0}\makebox(0,0)[lt]{\lineheight{1.25}\smash{\begin{tabular}[t]{l}$=0$\end{tabular}}}}%
    \put(0.69247529,0.08564706){\color[rgb]{0,0,0}\makebox(0,0)[lt]{\lineheight{1.25}\smash{\begin{tabular}[t]{l}$=\phi$\end{tabular}}}}%
    \put(0,0){\includegraphics[width=\unitlength,page=2]{pont-boucle.pdf}}%
  \end{picture}%
\endgroup%

    \caption{Bridge and loop relations.}
    \label{fig:pont-boucle}
\end{figure}
    \item Boundary relation: $\gamma=1-\phi$ for all curves $\gamma$ surrounding a puncture $x\in P$, that is $\gamma=\partial D^2$ where $D^2\subset S$ and $D^2\cap P=\{x\}$.
\end{enumerate}
One can easily prove that this vector space is finite dimensional and carries an action of the mapping class group $\Mod(S,P)$ by the formula $[f].[G]=[f(G)]$. 

Even more, $\A(S,P)$ has an algebra structure given by ``stacking'' and which is formally defined in the following way. Let $G_1,G_2$ be two graphs embedded in $S\setminus P$. We can ensure by an isotopy that they intersect tranversally in a finite number of points. For any $\xi:G_1\cap G_2\to \{\pm 1\}$ we define the smoothing $G_1\cup_\xi G_2$ by replacing the neighborhood of each intersection point $p\in G_1\cap G_2$ by a diagram where $G_1$ turns lefts to $G_2$ at $p$ if $\xi(p)=1$, right if $\xi(p)=-1$. 
We set then 
$$[G_1][G_2]=\sum_{\xi:G_1\cap G_2\to\{\pm 1\}} A^{\sum_{p}\xi(p)} [G_1\cup_\xi G_2].$$
One can then prove that this product induces a well-defined structure of algebra on $\A(S,P)$ which is preserved by $\Mod(S,P)$. 

Traditionally, this algebra structure is described with skein modules: we put $G_1$ ``above'' $G_2$ and apply at each crossing the Kauffman relation $[G_c]=A [G_{||}]+A^{-1}[G_{=}]$. The construction of the Fibonacci representation reduces to the following structure theorem:

\begin{Theorem*}
The algebra $\A(S,P)$ is isomorphic to $\End(V)$ for some finite dimensional $k$-vector space $V$. 
\end{Theorem*}
Let $\Phi:\A(S,P)\to \End(V)$ be such an (non canonical) isomorphism. As $\Mod(S,P)$ acts  on $\A(S,P)$ by algebra automorphisms, the Skolem-Noether theorem implies that this action is given through $\Phi$ by a conjugation, hence defining the Fibonacci representation $\rho:\Mod(S,P)\to \PGL(V)$ such that 
$$\Phi(f.x)=\rho(f)\Phi(x)\rho(f)^{-1}\quad\forall f\in \Mod(S,P),\forall x\in \A(S,P).$$
Finally, let $x\mapsto \overline{x}$ be the involution of $k$ satisfying $\overline{q}=q^{-1}$. This extends uniquely to an anti-involution of $\A(S,P)$ given by $\lambda[G]\mapsto \overline{\lambda}[G]$ for any embedded graph $G\subset S\setminus P$. This anti-involution corresponds through $\Phi$ to the adjunction with respect to a Hermitian form $h$ on $V$ preserved by $\Mod(S,P)$. In formulas, $h(\Phi(x)v,w)=h(v,\Phi(\overline{x})w)$ for all $v,w\in V$.

The involution being preserved by $\Mod(S,P)$, the Fibonacci representation is promoted to a representation $\rho:\Mod(S,P)\to \PU(V)$. 

To actually work with this construction, we need to find an explicit model for $V$: it can be constructed in a way similar to $\A(S,P)$ from a handlebody bounding the surface $S$, we refer to \cite{Bordeaux} for details. 

\subsection{Geometrization: cabinet de curiosités}

As the reader can check in the signature table of Fibonacci representations,  the coincidence of dimension \( pq = 3g-3+n\) happens only in few cases that we list here 
\begin{equation} \label{eq: coincidences} (g,n) \in \{ (0,4), \ (0,5), \ (1,2),\ (1,3), \ (2,1) \}  \end{equation} 
We prove that each of these coincidences correspond to a genuine complex hyperbolic structure on some compactification of the corresponding moduli space. 

It turns out that in the Fibonacci case \(\ell = 5\), there exists an orbifold contraction \( \Mb _{g,n}^5\rightarrow \Mb_{g,n} ^{\E}\), which contracts the boundary divisor consisting of stable nodal curves having at least one elliptic tail. This contraction was already considered  in the PhD dissertation  of Livne \cite{Livne} in the case of \( (g,n)=(1,2)\), but we provide a generalization of this result for any \( (g,n)\neq (2,0)\), see  \ref{ss: construction elliptic tail contraction}  (in the case \((g,n)=(2,0)\), the contraction leads to a quadratic singularity that will be studied in a forthcoming paper). This elliptic contraction, which might be interesting in its own, is very well adapted to the study of Fibonacci representations, since \(\pi_1 (\Mb _{g,n}^5)\simeq \pi_1( \Mb_{g,n} ^{\E})\). 
We prove (see the combination of Propositions \ref{p: uniformization_0_5}, \ref{p:M12}, \ref{p:M13} and \ref{p:M21})

\begin{Theorem*}\label{t: curiosities} In each of the cases \eqref{eq: coincidences}, the elliptic tail contraction \(\Mb_{g,n} ^{\E}\) admits a complex hyperbolic structure whose holonomy is the corresponding Fibonacci quantum representation.  \end{Theorem*}

The uniformization of \( \Mb_{0,5}^5\) by the complex hyperbolic plane has been made explicit by Deligne and Mostow in \cite{DM}, with the use of hypergeometric integrals. Together with Theorem \ref{t: curiosities} this gives a proof that the quantum representation \(\rho_{0,5}^5\) is the monodromy of the  hypergeometric function
\[F(x,y) = \int_1 ^{\infty} u^{-2/5} (u-1)^{-2/5} (u-x)^{-2/5} (u-y)^{-2/5} du .  \] 
Hirzebruch gave an alternative more abstract argument, by computing the Chern numbers of a convenient finite abelian smooth covering of the orbifold \(\Mb_{0,5}^5\) and showed that they satisfy the equality \(c_1^2 = 3 c_2 \), leading to the conclusion, thanks to Yau's theorem (solution to the Calabi's conjecture), that \( \Mb_{0,5}^5\) has a complex hyperbolic structure.
This relation of Hirzebruch complex hyperbolic orbifold with that of the orbifold \(\Mb_{0,5}^5\) has been noticed by Eyssidieux and Funar, see \cite[Example 2.7]{Eyssidieux-Funar}.

Analogously, the construction of the elliptic tail contraction together with its complex hyperbolic structure was discovered in the PhD's dissertation of Livne, \cite{Livne}. What theorem \ref{t: curiosities} says in that case is that the holonomy of this structure is in fact the corresponding Fibonacci representation.

A nice unexpected consequence of Theorem \ref{t: curiosities} is that it provides the solution of a problem raised by Siu  in the survey paper  \cite[Problem (a), p. 182]{Siu_problem}. We prove

\begin{Corollary*}\label{c: Siu}
There exists a surjective holomorphic map between connected compact complex hyperbolic manifolds, the domain and target being respectively of dimension \(3\) and \(2\).  
\end{Corollary*}
  
This map is obtained by lifting the forgetful map \(\Mb_{1,3}^{\E}\rightarrow \Mb_{1,2}^{\E}\)  to finite smooth coverings. Using forgetful maps to approach Siu's problem was already investigated in the context of Deligne-Mostow orbifolds in the work of Deraux \cite{Deraux}, although the conclusion was opposite.  We notice that Koziarz and Mok proved that such a surjective map between complex hyperbolic manifolds of different dimensions cannot be a submersion, see \cite{KM}. 

\vspace{0.5cm}

{\bf Acknowledgments:} We are indebted to many persons for numerous conversations and advices around this work, including Martin Deraux, Pascal Dingoyan, Philippe Eyssidieux, Elisha Falbel, Louis Funar, Selim Ghazouani, Alessandro Giacchetto, Vincent Koziarz, Danilo Lewanski, Ron Livne, Gregor Masbaum, Luc Pirio, Adrien Sauvaget, J\'er\'emy Toulisse, Nicolas Tholozan and Dimitri Zvonkine.

\tableofcontents

\section{Twisted orbifold structures on \(\Mb_{g,n}\)}\label{twistedmoduli}

In the context of algebraic geometry, the moduli space $\M_{g,n}$ and its compactification $\Mb_{g,n}$ are Deligne-Mumford stacks, a notion invented specially for them. The twisted version dealed with in this article has already been introduced in the context of TQFT in \cite{Eyssidieux-Funar} and in the context of $r$-spin structures, see \cite{Chiodo} Section 2.1. This notion is not very accessible, at least to the authors of this article, and is not strictly necessary for our purposes. For these reasons we define here these compactifications in an independent way. We present both the orbifold and orbispace viewpoints. Even though we will concretely mostly use the first one, the second will be useful to keep in mind for those having a more topological background/affinities. Most of the material of this section is classical, apart in the last subsection where a new construction of  a particular contraction of the twisted orbifold structure of \(\Mb_{g,n}\) in the level \(\ell =5\) case is described: the elliptic tail contraction.

\subsection{Preliminary remarks on orbifolds}

\subsubsection{Orbifold versus orbispaces}
In this article, we oscillate between two points of view on orbifolds. The first one is the usual concept of orbifold in the realm of differential complex geometry, the other one is the notion of orbispace which belongs to homotopy theory. Both are well-known, we refer to \cite{Henriques} for a nice discussion about their interplay. For the benefit of the reader, let us recall what these structures mean in the case of a developable orbifold, i.e a space of the form $X/G$ where $X$ is a complex variety and $G$ is a discrete group acting properly and holomorphically on $X$. 

An orbifold chart of $X/G$ around $[x]$ is obtained by linearizing the action of $\Stab(x)$ in a neighborhood $U$ of $x\in X$. This open set is projected to a neighborhood of $[x]\in X/G$, providing the orbifold atlas of $X/G$. A map \( f : X/ G\rightarrow Y/ H\) between two developable orbifolds is an orbifold map if it can be lifted to a map \(\tilde{f} : X\rightarrow Y\) which is equivariant with respect to a morphism \(G\rightarrow H\).

The emblematic example that will be considered here is the moduli space \(\M_{g,n}\) of algebraic curves of genus \(g\) with \(n\) marked points, assuming that the stability condition \( 3g-3+n>0\) holds. This space is the quotient of the  Teichm\"uller space \(\T_{g,n}\) by the action of the mapping class group \(\text{Mod}_{g,n}\). We recall (see e.g. \cite{Bers}) that \(\T_{g,n}\) has a structure of smooth complex manifold of dimension \(3g-3+n\) and that the action of \(\text{Mod} _{g,n} \) is properly discontinuous, so \( \M_{g,n}= \T_{g,n}/\Mod_{g,n}\) has a natural structure of developable orbifold. 

Another fundamental example is the Deligne-Mumford compactification \(\Mb _{g,n}\) of \(\M_{g,n}\). We refer to \cite{HubbardKoch}, \cite{ACG} or \cite{Zvonkine} for its definition as an orbifold, and provide a review of its construction in Section \ref{ss: Augmented Teichmuller}. Contrary to \(\M_{g,n}\), \(\Mb_{g,n}\) is not developable. However, we will work with alternative compactifications, twisted versions of \(\Mb_{g,n}\), which are developable, see section \ref{ss: Augmented Teichmuller}. 

This point of view is well-adapted for most geometric constructions involving for instance the integration of differential forms. However, the algebraic topology of $X/G$ is partially lost in the underlying topological space and not so easy to capture from the system of orbifold charts, as for instance the orbifold fundamental group. 

For this reason, what we call the orbispace is, in this case, the homotopical quotient, that is the space $X_G=EG\times X/G$ where $EG$ is a contractible space with a free and proper action of $G$. The action of $G$ is diagonal so that we have a natural projection $p:X_G\to X/G$. In the case when $G$ is a finite group acting trivially on a point $*$, this gives $*_G=EG/G=BG$, the classifying space of $G$. We observe that in general, the preimage $p^{-1}([x])$ is a classifying space for the finite group $\Stab(x)$. We refer to \cite{Henriques} for a general definition of orbispace and for the construction of the orbispace associated to an orbifold. 

The advantage of this second definition is that the orbifold fundamental group of $X/G$ is the usual fundamental group of $X_G$ and more generally, all invariants of $X/G$ coming from algebraic topology will be, by definition, the usual invariants of the homotopical quotient $X_G$.

To sum up, an orbifold is a topological space $M$ endowed with a system of orbifold charts that we denote by $M^o$. It can be converted into a (infinite dimensional) cell-complex $M^h$ by gluing the homotopical quotients of the charts. This latter space comes with a map $p:M^h \to M$ such that $p^{-1}(\{x\})\simeq B_{\Stab(x)}$. When no confusion is possible, the three structures $M,M^o,M^h$ will be denoted simply by $M$. 

In paragraphs \ref{sss: M11orbifold} and \ref{sss: M11orbispace}, we describe the twisted orbifold and orbispace structure \(\Mb_{1,1}^\ell\) in details, for the benefit of the reader who is not familiar with these notions. 

\subsubsection{Euler characteristic}
If a topological space $A$ is homeomorphic to the complement of a closed subcomplex $F$ of a finite complex $X$, we set $\chi(A)=\chi(X)-\chi(F)$. This quantity satisfies the identity $\chi(A)+\chi(B)=\chi(A\cup B)+\chi(A\cap B)$ when it makes sense and gives rise to an Eulerian integral so that $\chi(X)=\int_X d\chi$, see for instance \cite{CGR} for a full account.

If $\tilde{X}\to X$ is a finite covering of degree $d$ of finite CW-complexes, one has $\chi(\tilde{X})=d\chi(X)$. As $EG\to BG$ is a covering of degree $|G|$ and $EG$ is contractible, it is natural to set $\chi(BG)=1/|G|$. 

Finally, by integrating the Euler characteristic along the fibers of the map $p:M^h\to M$, we are led to define $$\chi(M^o)=\int_M \frac{d\chi(x)}{| \Stab(x)|}.$$

\subsubsection{The orbifold structure of $\Mb_{1,1}^\ell$}\label{sss: M11orbifold}
Consider first $\M_{1,1}$, the moduli space of elliptic curves with one marked point. As any pointed elliptic curve has the form $E_\tau=(\C / \Z\oplus \tau\Z, 0)$ for some $\tau\in \H$, two such curve \(E_\tau\) and \(E_{\tau'}\) being biholomorphic iff \(\tau'= \frac{a\tau+b}{c\tau+d}\) where 
\( \begin{pmatrix}
a & b  \\
c & d 
\end{pmatrix}
\in \text{SL}(2,\mathbb Z)\), we have $\M_{1,1}=\H/\SL_2(\Z)$ where the quotient is understood in the orbifold sense. The underlying topological space is a complex plane where the generic point has a stabilizer of order 2 and two special points have order 4 and 6. This gives $\chi(\M_{1,1})=-\frac{1}{12}$. 

In order to compactify $\M_{1,1}$, we add all rational points to the boundary of $\H$ to define $\overline{\H}=\H\cup\P^1(\Q)$ and set $\Mb_{1,1}=\overline{\H}/\SL_2(\Z)$. As these rational points form a single orbit, we have set theoretically $\Mb_{1,1}=\M_{1,1}\cup\{\infty\}$. 

As $\overline{\H}$ is no longer a complex variety, we need to explain how are defined the orbifold charts around the point $\infty$. The stabilizer of $\infty\in \overline{\H}$ is the group $\Z$ of translations by $1$: for $r>1$, the open set $U_r=\{z\in \H, \operatorname{Im} z>r\}\cup\{\infty\}$ induces a homeomorphism $U_r/\Z\to \Mb_{1,1}$ onto a neighborhood $V$ of the point at infinity. 
We fix now $\ell$ a positive integer and identify $U_r/\ell\Z$ with the disc $D$ of radius $e^{-2\pi r/\ell}$ by mapping $z$ to $e^{2i\pi z/\ell}$. By construction, there is a map $D\to V\subset \Mb_{1,1}$ which induces a homeomorphism $D/\mu_\ell\simeq V$ where $\mu_\ell$ is the group of $\ell$-th roots of unity.

\begin{Definition}
The orbifold structure on $\Mb_{1,1}^\ell$ is the unique orbifold structure on $\Mb_{1,1}$ which extends the orbifold structure on $\M_{1,1}$ and such that the orbifold chart around the point at infinity is given by the action of the group $\mu_\ell\times \Z/2\Z$ on $D$ where the second factor acts trivially.
\end{Definition}
By construction, the isotropy group at infinity is $\Z/\ell\Z\times \Z/2\Z$ which gives in particular $\chi(\Mb_{1,1}^\ell)=\frac{1}{2\ell}-\frac{1}{12}$. Its fundamental group is a double cover of the triangular group $\Delta(2,3,\ell)=\langle \alpha,\beta,\gamma| \alpha^2=\beta^3=\gamma^\ell=\alpha\beta\gamma=1\rangle$. It is well-known that this orbifold is developable, its universal covering being the hyperbolic plane $\H$ if $\ell>6$, the complex plane $\C$ if $\ell=6$ and the Riemann sphere $\mathbb P^1$ if $\ell<6$. Notice that in the case \(\ell =5\), the fundamental group of \(\Mb_{1,1}^\ell\) is the binary icosahedral group \(\mathrm{I}2\subset \text{SU} (2)\). This orbifold plays a fundamental role in the article, notably in section \ref{ss: construction elliptic tail contraction}.

We also observe that the topological space underlying $\Mb_{1,1}^\ell$ is homeomorphic to $S^2$. As the map $p:\Mbh_{1,1}^\ell\to \Mb_{1,1}^\ell$ induces an isomorphism in rational (co-)homology and cohomology (as shown by the spectral sequence of equivariant (co-)homology), we get $H^*(\Mb_{1,1}^\ell,\Q)\simeq H^*(S^2,\Q).$

\subsubsection{The orbispace $\Mb_{1,1}^\ell$ as a classifying space} \label{sss: M11orbispace}
Recall that the classifying space of a category is a simplicial set whose vertices are the objects of the category and $n$-simplices are parametrized by chains of maps $C_0\to\cdots\to C_n$. We refer to \cite{segal} for this notion and recall the two following basic facts. Two equivalent categories have homotopically equivalent classifying spaces and the classifying space of the category with one object $*$ with $\Hom(*,*)=G$ is the classifying space $BG$.

Consider a category $SC_\ell$ whose objects consists in pairs $(S,\gamma)$ where $S$ is a closed surface of genus $1$ and $\gamma$ is an essential simple closed curve, possibly empty. 

Denote by $\Gamma_\ell(S,\gamma)$ the group of homeomorphisms of $S$ preserving $\gamma$ and isotopic to some power of $T_\gamma^\ell$, where $T_\gamma$ denotes the Dehn twist along $\gamma$.
A morphism $f:(S,\gamma)\to (S',\gamma')$ is a homeomorphism such that $f(\gamma)\subset \gamma'$ with the relation that $f\sim \phi'\circ f \circ \phi$ for any $\phi\in \Gamma_\ell(S,\gamma)$ and $\phi'\in \Gamma_\ell(S',\gamma')$.

The classsifying space of this category is precisely the space $\Mbh_{1,1}^\ell$. 
The classifying space of the subcategory of pairs of the form $(S,\emptyset)$ is the space $B_{\SL_2(\Z)}$ which is homotopic to $\Mh_{1,1}$. The classifying space of the subcategory of pairs of the form $(S,\gamma)$ with $\gamma\ne\emptyset$ is the space $B_{\Z/2\Z\times\Z/\ell\Z}$. 

This construction is a $\ell$-twisted version of the construction given in \cite{CL}.

\subsection{Construction of the twisted compactification $\Mb_{g,n}^l$}\label{ss: Augmented Teichmuller}


\subsubsection{The orbifold structure \(\Mb _{g,n}^l\)} 
We review an analytic construction of a twisted version of Deligne-Mumford's orbifold, which was considered in the work of Eyssidieux and Funar \cite{Eyssidieux-Funar}. Our point of view is slightly different, and instead of using the stack road, we use the augmented Teichm\"uller space. 

Fix integers \(g,n\geq 0\), \( \ell\geq 1\), such that \( 2g-2+ n >0\), and let  \(S\) be a reference oriented closed surface of genus \(g\) with a subset \(P\subset S\) of cardinality \(n\). We denote by \(\Tb(S, P)\) the augmented Teichm\"uller space, namely the set of equivalence classes of couples \( (C, f)\) where $C$ is a stable nodal curve of genus $g$ and $f:S\to C$ is a pinching map, which means

\begin{enumerate}
    \item $f:S\to C$ is a continuous map such that $f(P)\cap \nodes{C}=\emptyset$ where $\nodes(C)$ is the set of nodes of $C$.
    \item For all $x\in \nodes{C}$, $\alpha_x=f^{-1}(x)$ is a simple curve and, setting $\alpha=f^{-1}\nodes(C)$, $f$ induces a homeomorphism from $S\setminus \alpha$ to $C\setminus\nodes(C)$.
\end{enumerate}
The isotopy class of $\alpha$ will be referred to as the pinched set of $f$. The stability condition is that each component of $C\setminus \nodes{C}\cup f(P)$ has negative Euler characteristic. 
Finally, two pairs $(C,f)$ and $(C',f')$ are equivalent if there exists a biholomorphism $\phi:C\to C'$ such that $\phi\circ f$ and $f'$ are isotopic. 

The augmented Teichm\"uller space is not a manifold, but it carries a natural stratification by sets having a complex manifold structure. Given an isotopy class of one dimensional submanifold \( \alpha\subset S \setminus P\) whose complementary regions have negative Euler characteristic, let \(B_\alpha\) be the stratum corresponding to curves whose pinched set is \(\alpha\). Each stratum \(B_\alpha\) is naturally identified with a product of usual Teichm\"uller spaces, and acquires a structure of complex manifold. For instance, the strata of maximal dimension \(B_{\emptyset}\) is identified with the usual Teichm\"uller space \(\mathcal T(S,P)\). For the topology on \(\overline{\mathcal T}(S,P)\) we refer to \cite{Abikoff, Bers} (see also the more recent treatments \cite{ACG} and \cite{HubbardKoch}). The naive quotient \(\Mb(S,P)=\overline{\mathcal T} (S,P)/\Mod(S,P)\) of the augmented Teichm\"uller space by the modular group is a compact space homeomorphic to the underlying topological space of Deligne-Mumford compactification \( \Mb_{g,n}\) of the moduli space of curves, see \cite{Harvey}. 




We now review the complex orbifold structures on \(\Mb(S,P)\) inherited from Deligne-Mumford, and its twisted versions. Let \(\Gamma_\alpha\subset \Mod (S , P)\) be the subgroup generated by the Dehn twists along the components of \(\alpha\), and by \(\Mod (S , P, \alpha)\subset \Mod (S ,P)\) the subgroup of elements that fix \(\alpha\) (the components might be permuted). Notice that $\Gamma_\alpha$ is a free abelian group of rank $|\alpha|$, the number of components of $\alpha$. We have an exact sequence 
\[0\rightarrow \Gamma_\alpha \rightarrow \Mod (S ,P,\alpha)\rightarrow \Mod (S /\alpha, P)\rightarrow 0\]
where $S/\alpha$ is obtained from $S$ by collapsing each connected component of \(\alpha\) to a point. If $(C,f)$ is an element of $\overline{\T}(S,P)$ whose pinched set is $\alpha$, we define $\Aut(C,f)$ so that it fits in the following exact sequence
\[0 \to \Gamma_\alpha\to \operatorname{Stab}(C,f)\to \Aut(C,f)\to 0\]

Let \(U_\alpha\) be the open subset of \(\overline{\mathcal T}(S, P)\) formed by stable marked curves whose pinched set is contained in \(\alpha\) up to isotopy. We denote by \(V_\alpha\) the quotient of \(U_\alpha\) by \(\Gamma_\alpha\). The following result allows to define an orbifold structure on the quotient \(\Mb(S, P)\) 
which recovers Deligne-Mumford's orbifold \(\Mb_{g,n}\), see \cite{HubbardKoch}:
 
\begin{Theorem}  \label{t: Deligne-Mumford's orbifold}
\begin{enumerate}
    \item \(V_\alpha\) has a unique structure of complex manifold so that the natural map \( \mathcal T(S,P) \rightarrow V_\alpha\) is holomorphic. 
    \item The projections in \(V_\alpha\) of the strata \(B_{\alpha'}\), with \(\alpha'\) describing the components of \(\alpha\), is a family of normal crossing divisors. 
    \item The natural projection \( V_\alpha\rightarrow \Mb(S,P)\) has,  locally around the class of \((C,f)\) in \(V_\alpha  \), fibers given by the orbits of the group \( \Aut (C,f)\).
    \item These charts provide an orbifold structure on \(\Mb(S, P)\) which is  biholomorphic to Deligne-Mumford's orbifold \(\Mb_{g,n}\). 
\end{enumerate}
\end{Theorem} 

We now define, for any \(\ell\geq 1\), a twisted orbifold structure \(\Mb^\ell(S,P)\) which recovers the previous one when $\ell=1$ and share the same underlying topological space. 
We define, for any \((C,f)\in \overline{\mathcal T} (S,P) \) whose pinched set is \(\alpha\), the group \(\Aut^\ell (C,f)=\operatorname{Stab}(C,f)/\ell\Gamma_\alpha\); this group is a central extension 
\begin{equation}\label{auto-l}
0 \rightarrow \Gamma_\alpha/\ell\Gamma_\alpha \rightarrow \Aut^\ell (C,f) \rightarrow \Aut (C,f)\rightarrow 0.
\end{equation}

\begin{Corollary} \label{c: structure of l-twisted compactification}
\begin{enumerate}
    \item There is a unique complex manifold structure on \( V_\alpha^\ell = U_\alpha/\ell\Gamma_\alpha\) such that the ramified covering \( V_\alpha ^\ell \rightarrow V_\alpha\) of group \(\Gamma_\alpha/\ell\Gamma_\alpha\) is holomorphic. 
    \item The projection of the strata \( B_{\alpha'}\) in \(V_{\alpha }^\ell \) is a normal crossing family of divisors. 
    \item For any \( (C,f)\in \overline{\mathcal T} (S,P)\) whose pinched set is \(\alpha\), the natural quotient map \(V_{\alpha}^\ell \rightarrow \Mb(S,P)\) has local fibers around the class of \((C,f) \in V_\alpha ^\ell \) given by the orbits of the group \( \Aut^\ell (C,f)\).  
    \item These charts provide an orbifold structure  \(\Mb^\ell(S, P)\) which is  biholomorphic to the construction given by Eyssidieux and Funar in \cite{Eyssidieux-Funar}. 
\end{enumerate}
\end{Corollary} 

\subsubsection{The orbifold $\Mb^\ell_{g,n}$ is uniformizable for \(\ell\geq 5\) odd}\label{ss: uniformizable}

We recall that an orbifold is uniformizable if it carries a finite orbifold covering which is smooth, in the sense that the isotropy groups are trivial. This is equivalent to saying that the orbifold is the quotient of a smooth manifold by a finite group acting by biholomorphisms. Eyssidieux and Funar proved that the orbifold \( \Mb_{g,n}^\ell \) is uniformizable, at least if \(\ell\geq 5\) is an odd integer, see \cite[Proposition 4.10]{Eyssidieux-Funar}  (In the case \(n=0\), this is a consequence of Pikaart and de Jong's work \cite{PdJ}: the smooth covering is a moduli space of curves with nilpotent level structures.). They notice that the \(\text{SO}(3)\)-quantum representations of level \(\ell\) with all colors equal to \(1\) are injective in restriction to the isotropy groups \(\Aut^\ell (C,f)\) of the orbifold \(\Mb_{g,n}^\ell \), defined in \eqref{auto-l}, hence this is a consequence of Selberg's lemma applied to the image of the relevant quantum representation. 

\subsubsection{Forgetful map} \label{sss: forgetful map twisted}

\begin{Lemma}\label{l: forgetful map} The natural forgetful map \( \Mb _{g,n+1} ^{\ell} \rightarrow \Mb _{g,n} ^{\ell} \) is a holomorphic orbifold map. \end{Lemma} 

\begin{proof} Given \(g,n\) with \( 2g-2+n>0\), we fix a subset $P=\{x_1,\ldots,x_{n+1}\}\subset S_g$ and define a continuous forgetful map (see \cite{ACG}) 
\begin{equation} \label{eq: forgetful AT}  \AT_{g,n+1} \rightarrow \AT_{g,n} \end{equation}
which assigns to a marked stable curve of genus \(g\) with \(n+1\) marked numbered points, \((C, f, Q=f(P) )\), the curve \( (C', f' , Q')\) where \((C',Q')\) is the stabilization of the curve  \((C, \{f(x_1),\ldots,f(x_{n})\}) \), and \( f'\) is 
the composition of $f$ with the stabilisation map \( (C,Q) \rightarrow (C',Q')\). 
The map \eqref{eq: forgetful AT} is equivariant with respect to the morphism 
\begin{equation} \label{eq: forgetful morphism} \text{Mod}_{g,n+1} \rightarrow \text{Mod}_{g,n}\end{equation} 
which sends the subgroup \(\Gamma_{g,n+1}^{\ell} \subset \text{Mod} _{g,n+1} \) generated by the \(\ell\)-powers of  Dehn twists in \(\text{Mod} _{g,n+1}\) to the corresponding subgroup \(\Gamma_{g,n} ^{\ell}  \subset \text{Mod} _{g,n} \).
Hence, denoting \(\AT_{g,n}^{\ell}:= \AT _{g,n}/\Gamma_{g,n} ^{\ell}\),  the map \eqref{eq: forgetful AT} induces a map 
\[ \AT_{g,n+1}^{\ell} \rightarrow \AT_{g,n}^{\ell} \] 
which is holomorphic with respect to the smooth complex structures on \(\AT_{g,n+1}^{\ell}\) and \(\AT_{g,n}^{\ell}\) given by Corollary \ref{c: structure of l-twisted compactification}; indeed, it is continuous and holomorphic in restriction to  \(\mathcal T_{g,n+1}/\Mod_{g,n+1} \), which is Zariski dense in \(\AT_{g,n+1}^{\ell}\), so this is a consequence of Riemann's extension theorem.  We deduce that the forgetful map 
\[ \Mb _{g,n+1} ^{\ell}\simeq  \AT_{g,n+1}^{\ell}/\text{Mod} _{g,n+1} ^{\ell} \rightarrow \Mb _{g,n} ^{\ell}\simeq   \AT_{g,n}^{\ell}/\text{Mod} _{g,n} ^{\ell} \]
is a holomorphic orbifold map as we wanted to prove. 
\end{proof}

\subsubsection{The $\ell$-twisted compactification as a classifying space}
We present here a construction of the homotopical version of $\Mb_{g,n}^\ell$ which is purely topological and makes clear the formal properties of these spaces. It is sufficient for defining the higher Toledo invariants of quantum representations and showing that they satisfy the axioms of a CohFT.

Let $SC_\ell$ be the category whose objects are triples $(S,P,\alpha)$ where $S$ is a closed oriented surface of genus $g$, $P\subset S$ a finite set of cardinality $n$ and $\alpha\subset S\setminus P$ a collection of disjoint simple curves such that each component of $S\setminus (P\cup\alpha)$ has negative Euler characteristic. 
We denote by $\Gamma_{\alpha}^\ell(S,P)$ the group of homeomorphisms of $S$ fixing $P$, preserving $\alpha$ and generated up to isotopy by $\ell$-th powers of Dehn twists along the components of $\alpha$. Finally we define a morphism $(S_1,P_1,\alpha_1)\to (S_2,P_2,\alpha_2)$ as a homeomorphism $f:S_1\to S_2$ mapping $P_1$ to $P_2$ and $\alpha_1$ into $\alpha_2$, up to the relation $f\sim \phi_2\circ f\circ \phi_1$ for $\phi_i\in \Gamma_{\alpha_i}^\ell(S_i,P_i)$.

\begin{Proposition}
The classifying space of the category $SC_\ell$ is the orbispace associated to the orbifold $\Mb^\ell(S,P)$. 
\end{Proposition}

\begin{proof}
This is a direct adaptation of Theorem 6.1.1 in \cite{CL}.
\end{proof}

\subsection{The elliptic tail contraction}\label{ss: elliptic tail contraction}

An alternative compactification of the moduli space will be useful to study the Fibonacci quantum representations, namely those corresponding to the group \(\text{SO}(3)\) and the level \(\ell=5\): it is obtained from \(\Mb_{g,n}^5\) by contracting the elliptic tail divisor \(\delta_{1,\emptyset}^5\). This divisor is made of nodal curves having at least one singular point separating the curve into two components, one of which being an elliptic curve without marked point. 
The main property of this compactification is that it still has a natural orbifold structure. In this section we provide the construction of this compactification, prove that it has some good functoriality properties with respect to forgetful maps, and finally we compute the first Chern class of its canonical bundle. 

\subsubsection{The construction}\label{ss: construction elliptic tail contraction}

\begin{Theorem} \label{t: elliptic contraction} 
For \( (g,n)\neq (2,0)\), there exists an orbifold \(\Mb_{g,n}^{\mathcal E}\) and an orbifold holomorphic (contraction) map \(c:\Mb_{g,n}^5 \rightarrow \Mb_{g,n}^{\mathcal E} \) whose fibers consist of equivalence classes of stable curves that are isomorphic after taking out the elliptic tails while keeping their attaching points. The map \(c\) induces an isomorphism at the fundamental group level. 
\end{Theorem}

The space \( \Mb _{g,n}^{\mathcal E} \), as a set, might be identified with the set of stable nodal curves of genus \( g- k \) with \(n+k\) marked points (\( k\le g\)), \(n\) of which (the marked points) being numbered, not the $k$ remaining ones (the tail points). The map \( \Mb_{g,n} ^5\simeq \Mb _{g,n} \rightarrow \Mb _{g,n} ^{\mathcal E}\) associates to a stable nodal curve of genus \(g\) with \(n\) numbered marked points, having \(k\) elliptic tails, the curve obtained by contracting each of its elliptic tails to a tail point.

\begin{Remark}
In the case \( (g,n)=(2,0)\), the elliptic tail divisor is still contractible, but its contraction leads to a quadratic singularity. This will be investigated in a forthcoming work.
\end{Remark}

It will be convenient to use an appropriate smooth finite orbifold Galois covering of \(\Mb_{g,n}^5\), and to construct the contraction in an equivariant way with respect to the Galois group:

\begin{Lemma}\label{l: Galois covering} There exists a finite Galois orbifold covering  \( X \rightarrow \Mb_{g,n}^5\) having the property that \(X\) is smooth, and that  the preimage in \(X\) of the elliptic tail divisor \(\delta_{1,\emptyset}^5\) is a normal crossing divisor \(\delta_{1,\emptyset}^X\) whose irreducible components are smooth hypersurfaces. 
\end{Lemma}

\begin{proof} Let \( \rho : \pi _1 (\Mb_{g,n}^5) \rightarrow G\) be a morphism to a finite group having the property that it is injective in restriction to the isotropy groups of \( \Mb_{g,n}^5\), see section \ref{ss: uniformizable} for its existence, and \( \rho':  \pi_1 ( \Mb_{g,n}^5) \rightarrow \text{Aut} (H_1(S,\mathbb Z/ 5\mathbb Z)) \) be the morphism induced by the action on the homology of \(S\) modulo \(5\) (it is a priori defined on the mapping class group; the fact that it descends to a morphism defined on \(\pi_1 ( \Mb_{g,n}^5)\) comes from that Dehn twists are mapped to elements of order \(1\) or \(5\) in \(\text{Aut} (H_1(S,\mathbb Z/ 5\mathbb Z))\)). 

Let \(X\) 
be the covering of \( \Mb_{g,n}^5 \) corresponding to the morphism \( \rho \times \rho' : \pi_1 (\Mb_{g,n}^5) \rightarrow G\times \text{Sp} (2g, \mathbb Z/ 5\mathbb Z)\). Since \(\rho\times \rho'\) is injective on isotropy groups of \(\Mb_{g,n}^5\), the covering \(X\) is a smooth orbifold. In particular, denoting by \(\pi: X\rightarrow \Mb_{g,n}^5\) the natural projection map, the second item of Corollary \ref{c: structure of l-twisted compactification} tells us that the preimage \(\delta_{1,\emptyset} ^X := \pi ^{-1} (\delta _{1,\emptyset}^5)\subset X\) is a normal crossing divisor.  

Any point \( x\in X\) corresponds to an equivalence class of pinching maps $f:S\to C_x$. Denote by \(f_x:=f_* : H_1 (S, \mathbb Z/5\mathbb Z) \rightarrow H_1 (C , \mathbb Z/5\mathbb Z)\) the map induced by the pinching map. As $X$ is the quotient of the augmented Teichmüller space 
by the kernel of the morphism $\rho\times \rho'$, the map \(f_x\) is well-defined since the kernel of \( \rho\times \rho'\) contains the kernel of \(\rho'\).

For any symplectic subspace \( E \subset H_1 (S,\mathbb Z/5\mathbb Z)\) of dimension \(2\), denote by \(H_E\subset X\) the set of elements \(x\in X\) so that \(C_x\) is a nodal curve having an elliptic tail whose first homology modulo \(5\) is the subspace \(f_x (E)\); the union of all \( H_E \)'s is the divisor \( \delta_{1,\emptyset}^X\). So it suffices to prove that \(H_E\) is a smooth hypersurface to conclude the proof of the lemma. 

The preimage of \(H_E\) in the augmented Teichm\"uller space is the union of all strata \(B_\alpha\) where some component \(\beta\) of \(\alpha\) is a simple closed curve that separates \(S\) in two components, one of which being homeomorphic to a torus minus a disc \(T_\beta \subset S\)  whose homology modulo \(5\) maps via inclusion onto the subspace \(E\). We denote by \( \mathcal C_E\) the set of all these \(\alpha\)'s. The closure of \(B_{\alpha}\) is the union of strata \(B_{\alpha'}\) where \(\alpha'\) contains an element $\alpha\in \mathcal{C}_E$ up to isotopy. 
Hence it suffices to prove that two non isotopic simple closed curves \(\beta,\beta'\in \mathcal C_E\) cannot  be components of a same \(\alpha \in \mathcal C_E\), or which is equivalent, that \(\beta\) and \(\beta'\) intersect. Suppose by contradiction that this happens.  Then \(\beta'\) being separating and not isotopic to \(\beta\) it cannot be contained in \( T_\beta\). Reversing the role of \(\beta\) and \(\beta'\) shows that \( T_\beta\) and \(T_{\beta'}\) are disjoint; in particular the image of their homology  group in \(H_1 (S,\mathbb Z/5\mathbb Z)\) are orthogonal, which is contradictory to the fact that they are both equal to \(E\).  \end{proof} 

The elliptic tail divisor \(\delta_{1,\emptyset}^5\) is parametrized in a natural way by the moduli space \( \Mb_{1,1}^5 \times \Mb_{g-1,n+1}^5 \) via an attaching map. Although this parametrization is not injective as soon as \(g\geq 2\), and so cannot be inverted at the level of \(\delta_{1,\emptyset}^5\subset\Mb_{g,n}^5\), it can be done at the level of \(H_E\subset X\) in the following sense: one has a natural covering map \(r_E: H_E\rightarrow \Mb_{1,1}^5\times \Mb_{g-1,n+1}^5\), which assigns to an element of \( H_E\) the unique elliptic tail whose homology modulo \(5\) is \(E\) as the first coordinates in \(\Mb_{1,1}^5\), and the contraction of this latter as the second coordinates in \(\Mb_{g-1,n+1}^5\). Since the unique smooth cover of \(  \Mb_{1,1}^5 \) is its universal cover 
which is biholomorphic to the Riemann sphere, this equip each (smooth)  hypersurface \(H_E\) of \(\delta_{1,\emptyset}^X\) with a locally trivial fibration \begin{equation}\label{eq: f_E} \widetilde{\Mb_{1,1}^5 }\simeq \mathbb P^1 \rightarrow H_E\stackrel{f_E}{\rightarrow} K_E\end{equation} 
over a smooth manifold \( K_E\). The fibers of \(f_E\) are the connected components of \( r_E^{-1} ( \Mb_{1,1}^5\times \{*\} )\).

Those fibrations on different \(H_E\)'s are compatible in the following sense:   

\begin{Lemma}\label{l: fibrations}
Given \(E_1,\ldots, E_r\) a collection of symplectic \(2\)-dimensional subspaces of \( H_1 (S, \mathbb Z/5\mathbb Z)\), the intersection \(H_{E_1, \ldots ,E_r}:= H_{E_1}\cap \ldots \cap H_{E_r}\) (it is not empty if and only if the \(E_i\)'s are orthogonal wrt the intersection form), there is a fibration of \( H_{E_1}\cap \ldots \cap H_{E_r}\) by  \((\mathbb P^1)^r\) (with a natural identification of the fibers with \( \widetilde{\Mb_{1,1}^5}^r\) up to the action of \( \pi_1 (\Mb_{1,1}^5)^r\), in particular the monodromy does not permute the \(\mathbb P^1\)-factors) so that the \(i\)-th \(\mathbb P^1\)-subfibration corresponds to the fibration \( f_{E_i}\) of \(H_{E_i}\) restricted to the intersection \( H_{E_1}\cap \ldots \cap H_{E_r}\). 
\end{Lemma}

\begin{proof}
The intersection \( H_{E_1, \ldots ,E_r}\) has a natural map to \( (\Mb_{1,1}^5) ^r \times \Mb_{g-r, n+r}^5\), whose first \(r\) coordinates are given by the elliptic tails corresponding to each \(E_i\)'s, and the last one is the contraction of those elliptic tails. This map is an orbifold covering map, and since \(H_{E_1, \ldots ,E_r}\) is smooth, and the unique smooth covering of \( (\Mb_{1,1}^5) ^r\) is its universal cover biholomorphic to the \(r\)-power of the Riemann sphere, the lift of the fibration of  \( (\Mb_{1,1}^5) ^r \times \Mb_{g-r, n+r}^5\) by  \( (\Mb_{1,1}^5) ^r \) defines a fibration of \(H_{E_1, \ldots ,E_r}\) whose fibers hare naturally universal covers of \( (\Mb_{1,1}^5) ^r \).  The last statement of the lemma is then obvious.
\end{proof}

\begin{Lemma} \label{l: icosahedral} 
For any fiber \(F\) of the fibration \eqref{eq: f_E}, we have \( H_E \cdot F = - 1\). 
\end{Lemma}


\begin{proof} 
Although this lemma is true in general, we explain the proof only when $(g,n)\ne (2,0)$. 
In the case where \((g,n)=(1,2)\), this result can be found in Livne's PhD dissertation \cite{Livne}, but for completeness we recall the proof here. The preimage \(\delta_{1,\emptyset} ^X\) of   \(\delta_{1,\emptyset}^5\simeq \Mb _{1,1}^5\) in \(X\) is a finite union of rational curves in this case, the stabilizer of each of those being a subgroup of \(H\) isomorphic to \(\pi_1(\mathcal V)\), where \(\mathcal V\) is a small tubular neighborhood \(\mathcal V\) of \(\delta_{1,\emptyset}^5\). Notice that the fundamental group of \(\mathcal V\) is a \(\mathbb Z/5\mathbb Z\)-extension of the fundamental group of \(\delta_{1,\emptyset}^5\), which is the binary icosahedral group \(\mathrm{I}2 \subset \SU(2)\), see Subsection \ref{sss: M11orbifold}.  Denoting by \(R\) a component of \( \delta_{1,\emptyset} ^X\), we then have
\[ [R]^2= 5\times |\mathrm{I}2|\times [\delta_{1,\emptyset}^5 ] ^2 . \]
But \( [\delta_{1,0} ^5 ]=\frac{1}{5}\delta_{1,\emptyset}\), so \([\delta_{1,\emptyset}^5]^2 = \frac{1}{5^2} [\delta_{1,\emptyset}]^2 = -\frac{1}{5^2  \times 24}\), see \cite[Part 2.2.2]{Zvonkine}, and finally we find \([R]^2 = -1\) as claimed since \(|\mathrm{I}2|= 120\).

\begin{Remark}
Using Castelnuovo's contraction theorem, one can deduce from this computation that the neighborhood of the divisor \( \delta_{1,\emptyset}^5\) in \(\Mb_{1,2}^5 \) is isomorphic as an orbifold to the neighborhood of the exceptional divisor in the quotient of the blow-up of \(\mathbb C^2\) at the origin by the group generated by multiplication by \( \mu_5 \id\) and by the binary icosahedral group \( \mathrm{I}2 \subset \SU(2)\). We leave the details for the reader. 
\end{Remark}

Let us now consider the general case. First notice that if \(g=0\) there is nothing to prove since \(\delta_{1,\emptyset}^5\) is empty. So in the sequel we suppose that \(g \geq 1\). The strategy is to construct a complex surface \(S\subset X\) which intersects \(H_E\) transversally in a curve whose components are fibers \(F\) of \(f_E\). Then the statement of the lemma is equivalent to saying that \(F\) is a \((-1)\)-rational curve in \(S\).

The construction of \(S\) depends on \((g,n)\). In the sequel we assume that \( (g,n)\notin \{(1,1),(1,2), (2,0)\}\).

If \(n\geq 1\), fix an element \( C_0 \in \Mb _{g-1,n}^5\) (notice that the stability condition \( 2(g-1) - 2 +n>0\) is satisfied by our assumptions), and define \(S\) as being the pull-back in \(X\) of the submanifold \(R\subset \Mb _{g,n}^5\) formed by nodal curves obtained by attaching a curve \(C\in \Mb_{1,2}^5\) to \(C_0\) by identifying fixed marked points. Then \(R\) is isomorphic to \(\Mb _{1,2}^5\), and under this identification, its intersection with \(\delta_{1,\emptyset }^5\subset \Mb_{g,n}^5\) is the elliptic tail divisor of \(\Mb_{1,2}^5\). Hence, the intersection of \( S\) with \(H_E\) are made of fibers \(F\) of \(f_E\), and the previous considerations in the particular case \((g,n)=(1,2)\) show that in \(S\) those \(F\)'s
 are \((-1)\)-curves. Hence we are done in that case.
 
If \(n=0\) and \( g\geq 3\). Fix smooth curves \(C_1\in \Mb_{g-1,1}^5\) and \(C_2\in \Mb_{1,1}^5\), and let \( R\subset \Mb_{g,n}^5\) be the surface formed by nodal curves obtained by attaching \(C_1\) and \(C_2\) to a curve  \(C\in  \Mb_{1,2}^5\).  Let now \(S\subset X\) be the set of curves that project to a curve of \(R\), in such a way that the homology of the component \(C\) modulo \(5\) is equal to \(E\). Then \(S\) is a covering of \(\Mb_{1,2}^5\), and its intersection with \(H_E\) is the pull back of the tail divisor of \(\Mb_{1,2}^5\). Hence the lemma is proved in that case too. 
\end{proof} 

\begin{proof}[Proof of Theorem \ref{t: elliptic contraction}]
Suppose that we have a smooth compact complex manifold  \(Y\), a finite family of normal crossing smooth hypersurfaces \( H_i \subset Y\), and  for any subset of indices \(J\subset I \), fibrations \( (\mathbb P^1)^{J} \rightarrow  H_J:=\bigcap\limits_{j\in J}  H_{j} \stackrel {f_J}{\rightarrow}   K_J \),  in such a way that 
\begin{enumerate}
    \item the monodromy of \(f_J\) does not exchange the factors of the fibers \(\simeq (\mathbb P^1)^{J}\), so given any \(J'\subset J\), the \(J'\)-coordinate part of the fibration \(f_J\) is a well-defined \((\mathbb P^1 )^{J'}\)-fibration , 
    \item given disjoint subsets \(J_1, J_2\subset I\) the intersection \( H_{J_1,J_2}= H_{J_1} \cap H_{J_2}\) is invariant by the fibration \(f_{J_1}\) (resp. \(f_{J_2}\)), and its restriction to \(H_{J_1,J_2}\) is the \(J_1\)-coordinate part (resp. the \(J_2\)-coordinate part) of the fibration \(f_{J_1,J_2}\), 
    \item for any \(i\in I\), and any fiber \(F= f_i ^{-1} (*)\subset H_i\), we have \(H_i \cdot F=-1\).  
\end{enumerate}

Choose such a data and enumerate \(I=\{i_1,\ldots, i_r\}\).  By \cite{Nakano} and its supplement \cite{FN}, we can find a contraction map \( c_1 : Y\rightarrow Y_1 \) to a smooth complex compact manifold, which has the property that it maps \(H_{i_1}\) to a codimension two submanifold of \(Y_1\), contracting each fibers of \(f_{i_1}\) to a point and not more. The \(c_1\)-images of the hypersurfaces \(H_{i_k}\) for \(k\geq 2\) form a smooth family of normal crossing hypersurfaces in \(Y_1\), and their intersections are naturally endowed with fibrations that satisfy all the previous properties 1., 2. and 3.. We can then define inductively smooth compact complex spaces \(Y_k\) and contractions \( c_k : Y_{k-1} \rightarrow Y_k\) that contract the \(\mathbb P^1\)-fibration of the hypersurface \(  c_{k-1}\circ\ldots \circ c_1 (H_{i_k})\). At the end we obtain a space \(Y_r\) which is topologically the quotient of \(Y\) by the equivalence class given by \(x\sim y\) iff for each \(i\) we have that \( f_i(x)=f_i(y)\) as soon as both \(x,y\in H_i\).  This space does not depend on the enumeration of \(I\) that we have chosen; indeed, this is clear at the topological level, and at the analytical one this is a consequence of Riemann's extension theorem and of the fact that the map \(c_{r} \circ \ldots \circ c_{1}\) is injective apart from a codimension one analytic set. 

Applying this to the covering \(X\) of \(\Mb_{g,n}^5\) constructed before, together with the family of smooth hypersurfaces \(H_E\) and the fibrations of their intersections given by Lemma \ref{l: fibrations}, we find a contraction map \( c: X \rightarrow Z \) to a smooth space \(Z\), which by the aforementioned unicity is equivariant with respect to a morphism \( \text{Gal} (X\rightarrow \Mb_{g,n}^5) \rightarrow \text{Aut} (Z)\). The quotient of \(Z\) by the image of the previous morphism is the desired quotient of \( \Mb_{g,n}^5\).    
\end{proof}

\subsubsection{Forgetful map between Fibonacci elliptic tail contractions}

In the Fibonacci case \( \ell = 5\), this enables to construct natural forgetful maps between the elliptic tail contractions. 

\begin{Lemma}\label{l: forgetful map elliptic contraction}
The forgetful map of Lemma \ref{l: forgetful map} induces an orbifold holomorphic map 
\begin{equation} \label{eq: forgetful map elliptic contraction} \Mb _{g,n+1}^{\E}\rightarrow \Mb _{g,n} ^{\E}\end{equation} 
which is onto. 
\end{Lemma}

\begin{proof}
We consider the finite Galois orbifold coverings \( X_{g,n+1}\rightarrow \Mb_{g, n+1}^{5}\) and \( X_{g,n}\rightarrow \Mb_{g, n}^{5}\) constructed in subsection \ref{ss: construction elliptic tail contraction}, with smooth underlying spaces \(X_{g,n+1}\) and \(X_{g,n}\). As in Lemma \ref{l: Galois covering}, we choose the coverings in such a way that that they cover the covering of \(\Mb_{g,n+1}^5\) or \(\Mb_{g,n}^5\) made of pointed curves whose underlying curve is marked by \( H_1 (S, \mathbb Z/ 5\mathbb Z)\). In particular the pull back of the elliptic tail divisor is the union of hypersurfaces \( H_{g,n}=\cup _E H_{E,g,n}\), with \(E\) varying over the set of symplectic dimension two symplectic submodules of \(H_1(S,\mathbb Z/5\mathbb Z)\),  with \(H_{E,g,n}\) being the subset of \(\Mb_{g,n}^5\) having an elliptic tail whose homology group maps onto  \(E\) by inclusion. 

Up to taking a larger covering of \( X_{g,n+1}\) if necessary, we can assume that the forgetful map \( \Mb_{g,n+1}^5\rightarrow \Mb_{g,n}^5\) (constructed in \ref{sss: forgetful map twisted})   lifts to a holomorphic map \(h: X_{g,n+1}\rightarrow X_{g,n} \) equivariant wrt to the actions of the Galois groups of the coverings \(X_{g,n+1}\rightarrow \Mb_{g,n+1}^5\) and  \(X_{g,n}\rightarrow \Mb_{g,n}^5\). By construction, the map \(h \) maps \( H_{E,g,n+1}\) to \(H_{E,g,n}\), sending the \(\mathbb P^1\)-fibration of \(H_{E,g,n+1}\) to the one of \( H_{E,g,n}\). So it induces a continuous (and hence holomorphic by Riemann's extension theorem) map from the contraction \(Z_{g,n+1}^E\) of the \(H_{E,g,n+1}\)'s to the contraction \(Z_{g,n}^E\) of \(H_{E,g,n}\). Reasoning inductively as in the proof of Theorem \ref{t: elliptic contraction} shows that the map \(h\) induces a holomorphic map  from the contraction \(Z_{g,n+1}\) of all \(H_{E,g,n+1}\) to the corresponding space \(Z_{g,n}\). This map is equivariant with respect to the natural actions of the Galois group of \(X_{g,n+1}\rightarrow \Mb_{g,n+1}^5\) (resp. of \(X_{g,n}\rightarrow \Mb_{g,n}^5\)) acting on \(Z_{g,n+1}\) (resp. \(Z_{g,n}\)). This ends the proof of the lemma.\end{proof}

\begin{Remark} The kernel of the forgetful morphism \eqref{eq: forgetful morphism} 
is isomorphic to the fundamental group \(F_{g,n}=\pi_1(S\setminus \{x_1,\ldots, x_n\})\) of the genus \(g\) surface minus \(n\) points, by Birman's exact sequence. So the kernel of the \(\ell\)-twisted forgetful morphism 
\begin{equation} \label{eq: twisted forgetful morphism} \Mod _{g,n+1} ^{\ell} \rightarrow \Mod _{g,n} ^{\ell} \end{equation}
is the quotient of \(F_{g,n}\) by the group generated by \(\ell\)-powers of elements freely homotopic  to simple closed curves. It would be interesting to compute this group for \( (g,n)=(1,2)\) (in this case \(\Mb_{g,n+1}^{\mathcal E} \) and \(\Mb_{g,n}^{\mathcal E}\) are both complex hyperbolic orbifolds of respective dimension  \(3\) and \(2\), and the forgetful map answers in the negative Siu's problem \cite[Problem (a)]{Siu_problem}, see Corollary \ref{c: Siu}). 
\end{Remark}

\subsection{The canonical bundles of  \(\Mb_{g,n}^\ell\) and \(\Mb_{g,n}^{\E} \)}\label{canonical_bundle}

We recall that the second cohomology group of \(\Mb_{g,n}\) with rational coefficients is generated by the \(\psi\)-classes \( \psi_i\), \(i=1,\ldots, n\), the class \(\kappa_1\), and the classes of the boundary divisors: \( \delta_{irr}\) and \(\delta_{a, A}\), where \( 0\leq a\leq g\) and \(A\subset P\) is a subset satisfying the inequalities \(2a-2 + |A|\geq 0\) and \(2(g-a)-2 + (n-|A|) \geq 0\) (see \cite{AC}).   

In the sequel we denote by \(\delta\) the sum of all boundary divisors and by \(\psi=\sum_i \psi_i\) the sum of \(\psi\)-classes. It will be convenient for us to introduce the class 
\[ \tilde{\kappa}_1 = \kappa_1 - \psi\] 
since the Toledo invariants of quantum representations are better expressed in the basis formed by \(\psi\)-classes, boundary classes, and \(\tilde{\kappa}_1\).

\begin{Lemma}\label{l: canonical bundle}
\(K_{\Mb_{g,n}^\ell} = \frac{13}{12} \tilde{\kappa}_1 + \left( \frac{1}{12}- \frac{1}{\ell}\right)  \delta + \psi \)
\end{Lemma}

\begin{proof}
Harris and Mumford proved that  (see \cite[Theorem 7.15]{ACG})
\[ K_{\Mb_{g,n}} = 13\lambda_1 + \psi -2\delta\]
where \(\lambda_1\) is the first Chern class of the Hodge bundle, and this latter is expressed in our basis by the formula (see \cite[Theorem 7.6]{ACG}).
\[ \tilde{\kappa}_1 = 12 \lambda_1 - \delta. \]
So we have: 
\[K_{\Mb_{g,n}} = \frac{13}{12} \tilde{\kappa}_1 -\frac{11}{12} \delta +\psi . \]

Now, the natural map \(\Mb_{g,n}^\ell \rightarrow \Mb_{g,n}\) is holomorphic and ramifies on the boundary divisors at the order \(\ell\), so we get  
\[  K_{\Mb_{g,n}^\ell} = K _{\Mb_{g,n}} + (1-\frac{1}{\ell}) \delta\]
and the result follows.
\end{proof}

\begin{Lemma}\label{l: canonical bundle elliptic contraction}
For \((g,n)\neq (2,0)\), denote by \(c : \Mb_{g,n}^5 \rightarrow \Mb_{g,n}^{\E}\)  the blow-down (see Theorem \ref{t: elliptic contraction}). We have 
\[ c^* K_{\Mb_{g,n}^{\E}} = K_{\Mb_{g,n}^5 } - \delta_{1,\emptyset}^5 \]
so its first Chern class in \( H^2 (\Mb_{g,n}^5,\mathbb Q) \simeq H^2 (\Mb_{g,n},\mathbb Q)\) 
satisfies 
\[ c_1(c^* K_{\Mb_{g,n}^{\E}}) = c_1 (K_{\Mb_{g,n}^5 } ) - \frac{1}{5} \delta_{1,\emptyset} \]
\end{Lemma}

\begin{proof}
If \(Y\) is a smooth complex analytic space, with \(\mathbb P^1\)-fibered hypersurfaces \(H_i\) as in the proof of Theorem \ref{t: elliptic contraction}, we have \( c_1^* (K_{Y_1}) = K_Y - H_1\). By induction, if we denote \(Z= Y_r\), \(H=\cup_i H_i\), and \(c= c_r \circ \cdots \circ c_1: Y\rightarrow Z\)  the composition of all contractions, we get \( c^* K_{Z} = K_Y - H\). If \(Y\), \(H\) and the \(\mathbb P^1 \)-fibrations on the \(H_i\)'s are invariant by a finite group \(G\subset \text{Aut}(Y) \), denoting by \(G'\subset \text{Aut} (Z)\) the image of the action of \(G\) on the quotient \( Z\) (which is unique), we then have 
\[ K_{Y/G} = c^* K_{Z/G'}  - H/G ,\]
which implies the lemma by construction of the elliptic contraction.
\end{proof}

\subsection{The Euler characteristic of $\Mb^\ell_{g,n}$} 

Let $x=(C,f)$ be a stable curve of genus $g$ with $n$ marked points that we think as a point in $\Mb_{g,n}$. 
We define $|\nodes(x)|$ to be the number of nodal points of $x$ and $|\Stab^\ell(x)|$ to be the size of the isotropy group of $x$ in $\Mb^\ell_{g,n}$ (set $\Stab^1=\Stab$). From the exact sequence \eqref{auto-l}, we get $|\Stab^\ell(x)|=\ell^{|\nodes(x)|}|\Stab(x)|$.
This suggest to define the polynomial 
$$\overline{\chi}_{g,n}(\kappa)=\int_{\Mb_{g,n}} \frac{\kappa^{|\nodes(x)|}}{|\Stab(x)|}d\chi(x)\in \Q[\kappa]$$
so that $\chi(\Mb_{g,n}^\ell)=\overline{\chi}_{g,n}(\frac 1 \ell)$. This polynomial satisfies the following quadratic recursion relation which allows to compute it effectively, see \cite{WZ} Theorem 3.6 and the formulas following the theorem (notice that they use instead $\tilde{\chi}_{g,n}=\overline{\chi}_{g,n}/n!$):

$$\overline{\chi}_{g,n}=\int_0^\kappa\Big(\overline{\chi}_{g-1,n+2}+\!\!\sum_{\underset{n+2=n_1+n_2}{g=g_1+g_2}}\!\!n_1n_2\binom{n}{n_1}\overline{\chi}_{g_1,n_1}\overline{\chi}_{g_2,n_2}\Big)d\kappa+\chi(\M_{g,n}).$$
This formula, together with Harer-Zagier formula $$\chi(\M_{g,n})=(-1)^n\frac{(2g-1)B_{2g}}{(2g)!}(2g+n-3)!$$
allows to compute the Euler characteristic of $\Mb_{g,n}^\ell$. 
We find for instance $\overline{\chi}_{0,5}(\kappa)=2-10\kappa+15\kappa^2$ hence $\chi(\Mb_{0,5}^5)=\frac{3}{5}$. 
\section{Hermitian cohomological invariants}\label{toledo}

Let $V$ be a finite dimensional complex vector space endowed with a non-degenerate Hermitian form $h$ of signature $(p,q)$. This means that there are coordinates $(x_1,\ldots,x_p,y_1,\ldots,y_q)$ such that 
\begin{equation}\label{eq: standard hermitian form}  
h=\sum_{i=1}^p|x_i|^2-\sum_{j=1}^q|y_j|^2.
\end{equation} 
For simplicity, we will often remove $h$ from the notation. Depending on the purpose, we will denote by $\PU(V)$ or $\PU(p,q)$ the group of projective unitary transformations of $V$. We will also write $d(V)=p+q$ and $\sigma(V)=p-q$. 

\subsection{Definition of the invariants}

Let $X$ be a connected topological space endowed with a representation $\rho:\pi_1(X)\to \PU(p,q)$.
The purpose of this section is to define a family $$\sch(\rho)=\sum_{k\ge 0}\sch_k(\rho)\in \prod_{k\ge 0}H^{2k}(X,\Q)$$
which is natural in the sense that whenever there is $f:X\to Y$ and $\rho:\pi_1(Y)\to \PU(p,q)$ then $\sch(\rho\circ f_*)=f^*\sch(\rho)$. 

One may think of $\rho$ as the holonomy of a flat $\PU(p,q)$-principal bundle over $X$. By forgetting the flat structure, this bundle is obtained by pulling back a universal $\PU(p,q)$-principal bundle $\mathrm{EPU}(p,q)\to \BPU(p,q)$ by a map $f:X\to \BPU(p,q)$, well-defined up to homotopy. We are then reduced to defining a class $\sch\in H^*(\BPU(p,q),\Q)$ and set $\sch(\rho)=f^*\sch$.

Recall that the natural map $\pi:\SU(p,q)\to \PU(p,q)$ fits into the following central exact sequence
$$0\to \mu_{p+q}\to \SU(p,q)\overset{\pi}{\to} \PU(p,q)\to 0$$
where $\mu_{p+q}\subset \C^*$ is the group of roots of unity of order $p+q$. This exact sequence gives rise to a fibration $\mathrm{B}\mu_{p+q}\to \mathrm{BSU}(p,q)\to \BPU(p,q)$. 
As $H^{k}(\mathrm{B}\mu_{p+q},\Q)=0$ for $k>0$, a usual argument involving the Leray-Serre spectral sequence gives that the map $$\pi^*:H^*(\BPU(p,q),\Q)\to H^*(\mathrm{BSU}(p,q),\Q)$$
is an isomorphism. Hence it is sufficient to define $\sch\in H^*(\mathrm{BSU}(p,q),\Q)$. 

Consider $P\to X$ a principal $\SU(p,q)$-bundle and form the Hermitian bundle $\mathcal{E}\to X$ associated to the tautological action of $\SU(p,q)$ on $V$. We can find an orthogonal decomposition $\mathcal{E}=\mathcal{E}^+\oplus\mathcal{E}^-$ such that the restriction of the Hermitian structure to $\mathcal{E}^{+}$ (resp. $\mathcal{E}^-$) is positive (resp. negative). 
Then we set $$\sch(\rho)=\ch(\mathcal{E}^+)-\ch(\mathcal{E}^-)$$
where $\ch$ denotes the Chern character. In particular, we have $\sch_0(\rho)=\sch_0(P)=p-q=\sigma(V)$. We also observe that if $q=0$ then $\mathcal{E}=\mathcal{E}^{+}$. If this bundle is constructed from a representation, it has a flat connection, giving $\sch(\rho)=d(V)$. The same argument works if $p=0$, giving $\sch(\rho)=-d(V)$. Hence in the sequel, we suppose that $pq>0$. 

Denote by $\H(V)=\H^{p,q}$ the space of orthogonal decompositions $V=V^+\oplus V^-$. It is the symmetric space associated to $\PU(p,q)$, in particular it is contractible. Given $P$, one can form the bundle $\mathcal H\to X$ associated to the action of $\SU(p,q)$ on $\H^{p,q}$. The decompositions $\mathcal{E}=\mathcal{E}^{+}\oplus \mathcal{E}^-$ are in bijection with the sections of $\mathcal{H}$, hence are unique up to homotopy. This shows that the class $\sch(\rho)$ is well-defined.

Notice that in order to compute the higher Toledo invariants, we need to linearize the representation $\rho:\pi_1(X)\to \PU(p,q)$, i.e. to lift it to $\SU(p,q)$. In the cases we will encounter, this is not possible unless we modify the space $X$. However, in some cases handled in the following proposition, we can find a short-cut. 

\begin{Proposition}\label{projplat}
Let $\mathcal{E}\to X$ be a Hermitian bundle endowed with a projectively flat connexion with holonomy $\rho:\pi_1(X)\to \PU(p,q)$. Then, given a decomposition $\mathcal{E}=\mathcal{E}^+\oplus \mathcal{E}^-$ as before, we have:
$$\sch(\rho)=(\ch(\mathcal{E}^+)-\ch(\mathcal{E}^-))e^{-\frac{c_1(\mathcal{E})}{p+q}}.$$
\end{Proposition}

We can check that this quantity does not change if we replace $\mathcal{E}$ by $\mathcal{E}\otimes L$ for some Hermitian line bundle $L$. In particular, if the representation lifts to $\U(p,q)$, and we take $\mathcal{E}$ to be the associated flat bundle, then $c_1(\mathcal{E})=0$ and the two formulas coincide. We insist on a crucial property: $\sch(\rho)$ is independent of the lift.

\begin{proof}[Proof of Proposition \ref{projplat}]
Let $\mathcal{E}\to X$ be as in the proposition. In order to compute $\sch(\rho)$, we need to linearize $\rho$, which is generally impossible. However we can look for a map $f:\tilde{X}\to X$ such that $f^*:H^*(X,\Q)\to H^*(\tilde{X},\Q)$ is an isomorphism and such that there is a diagram
$$\xymatrix{
\pi_1(\tilde{X})\ar[r]^{\tilde{\rho}}\ar[d]& \SU(p,q)\ar[d]\\ 
\pi_1(X)\ar[r]^\rho& \PU(p,q)}$$
Suppose first that we have solved this problem. The naturality of the construction gives $\sch(\tilde{\rho})=f^*\sch(\rho)$. Hence we are reduced to the case when $\rho$ takes its values in $\SU(p,q)$. 
In that case, we may compare the projectively flat bundle $\mathcal{E}$ with the flat Hermitian bundle $\mathcal{E}_\rho$ associated to $\rho$. 
The corresponding projective bundles are isomorphic: this implies that there is a Hermitian line bundle $L$ such that $\mathcal{E}_\rho=\mathcal{E}\otimes L$. As $c_1(\mathcal{E}_\rho)=0$, we get $c_1(L)=-\frac{c_1(\mathcal{E})}{p+q}$.

We check that $\sch(\mathcal{E}_\rho)=\ch(\mathcal{E}_\rho^+)-\ch(\mathcal{E}_\rho^-)=\ch(\mathcal{E}^+\otimes L)-\ch(\mathcal{E}^-\otimes L)=(\ch(\mathcal{E}^+)-\ch(\mathcal{E}^-))\ch(L)$. This coincides with the formula of Proposition \ref{projplat}.  

We now prove the existence of $f:\tilde{X}\to X$ with a twist: we will replace $X$ by a space homotopically equivalent to it. Recall that the obstruction of lifting $\rho:\pi_1(X)\to \PU(p,q)$ is a class $o(\rho)\in H^2(X,\mu_{p+q})$, represented by a map $f:X\to K(2,\mu_{p+q})$.
Replacing $X$ by the (homotopically equivalent) mapping path space 

$$E_f=\{(x,\gamma), x\in X, \gamma:[0,1]\to K(2,\mu_{p+q}), \gamma(0)=f(x)\},$$
\noindent
the map $g:E_f\to K(2,\mu_{p+q})$ given by $g(x,\gamma)=\gamma(1)$ is a fibration homotopic to $f$. Its fiber $F$ solves the problem. Indeed, the composition $F\to E_f\to K(2,\mu_{p+q})$ is constant, meaning that the obstruction $o(\rho)$ vanishes on $F$. Moreover, as the rational cohomology of $K(2,\mu_{p+q})$ is trivial, the inclusion $F\subset E_f\simeq X$ induces an isomorphism in rational cohomology (from the Leray-Serre spectral sequence). 
\end{proof}

For the sake of completeness, we study the problem of realizing a projective representation as the holonomy of a projectively flat bundle. 
\begin{Lemma}\label{projplatobs}
Given any representation $\rho:\pi_1(X)\to \PU(p,q)$, there exists a Hermitian complex bundle $\mathcal{E}\to X$ endowed with a projectively flat connection whose monodromy is conjugate to $\rho$ if and only if some obstruction class in $H^3(X,\Z)$ vanishes. 
\end{Lemma}

\begin{proof}
Recall that the representation $\rho:\pi_1(X)\to \PU(p,q)$ gives rise to a flat $\PU(p,q)$-bundle $P\to X$. 
Consider a good open covering $(U_i)_{i\in I}$ of $X$ with trivializations of $P|_{U_i}$. The transition functions are constant maps $g_{ij}:U_i\cap U_j\to \PU(p,q)$ satisfying a cocycle condition. A Hermitian bundle $\mathcal{E}$ may be constructed by taking continuous maps $h_{ij}:U_i\cap U_j\to U(p,q)$ satisfing the same cocycle condition. The condition that $\mathcal{E}$ has a projectively flat connection with monodromy $\rho$ means that $\pi\circ h_{ij}=g_{ij}$ where $\pi:U(p,q)\to \PU(p,q)$ is the obvious projection. 

As $U_i\cap U_j$ is contractible, one can find such a map $h_{ij}$ independently for all $i\ne j$. The cocycle condition gives a map $U_i\cap U_j\cap U_k\to \U(1)=\ker \pi$ which has to vanish in order to prove the lemma. This defines a class in $H^2(X,C_X(\U(1)))$ where for any topological abelian group $G$, $C_X(G)$ denotes the sheaf of continuous $G$-valued functions. 
From the exact sequence of sheaves $0\to C_X(\Z)\to C_X(\R)\to C_X(\U(1))\to 0$ and the vanishing of $H^*(X,C_X(\R))$, we find an obstruction in $H^3(X,C_X(\Z))=H^3(X,\Z)$.
\end{proof}
\subsection{Compatibility with operations}
Let $V$ and $W$ be two finite dimensional Hermitian spaces. If we have two representations $\rho_V:\pi_1(X)\to \PU(V)$, $\rho_W:\pi_1(X)\to \PU(W)$, we cannot make sense of their sum but we can make sense of their tensor product $\rho_V\otimes \rho_W:\pi_1(X)\to \PU(V\otimes W)$. 

\begin{Proposition}\label{multiplicativity}
Given two projective representations as above we have 
$$\sch(\rho_V\otimes \rho_W)=\sch(\rho_V)\smile\sch(\rho_W)\in H^*(X,\Q).$$
\end{Proposition}
\begin{proof}
As explained in the previous section, one can suppose that the representations are linearized in the sense that $\rho_V:\pi_1(X)\to \SU(V)$, $\rho_W:\pi_1(X)\to \SU(W)$. One may form the associated bundle $\mathcal{E}_{V\otimes W}$ of $\rho_V\otimes \rho_W$ by taking the tensor product of $\mathcal{E}_V$ and $\mathcal{E}_W$, the Hermitian bundles associated respectively to $\rho_V$ and $\rho_W$. Taking a decomposition $\mathcal{E}_V=\mathcal{E}_V^+\oplus \mathcal{E}_V^-$ and $\mathcal{E}_W=\mathcal{E}_W^+\oplus \mathcal{E}_W^-$, we get a decomposition
$$\mathcal{E}_{V\otimes W}=\Big(\mathcal{E}_V^+\otimes \mathcal{E}_W^+\oplus \mathcal{E}_V^-\otimes \mathcal{E}_W^-\Big)\oplus\Big(\mathcal{E}_V^+\otimes \mathcal{E}_W^-\oplus \mathcal{E}_V^-\otimes \mathcal{E}_W^+\Big).$$
From the properties of the Chern character, we readily get
$\ch(\mathcal{E}_V^+\otimes \mathcal{E}_W^+\oplus \mathcal{E}_V^-\otimes \mathcal{E}_W^-)-\ch(\mathcal{E}_V^+\otimes \mathcal{E}_W^-\oplus \mathcal{E}_V^-\otimes \mathcal{E}_W^+)=(\ch(\mathcal{E}_V^+)-\ch(\mathcal{E}_V^-))(\ch(\mathcal{E}_W^+)-\ch(\mathcal{E}_W^-))$ from which the result follows. 
\end{proof}

Let us now deal with the more subtle sum of two Hermitian spaces $V$ and $W$. We set $\PU(V,W)=\mathrm{P}( \U(V)\times \U(W))$. There are two natural projections $p_V:\PU(V,W)\to \PU(V)$, $p_W:\PU(V,W)\to \PU(W)$ and an inclusion $i:\PU(V,W)\to \PU(V\oplus W)$. 

\begin{Proposition}\label{additivity}
Given a representation $\rho:\pi_1(X)\to \PU(V,W)$, we have
$$\sch(p_V\circ \rho)+\sch(p_W\circ\rho)=\sch(i\circ \rho).$$
\end{Proposition}
\begin{proof}
Again we can suppose that $\rho$ takes its values in $\SU(V,W)$. Its associated bundle is the sum $\mathcal{E}_V\oplus \mathcal{E}_W$ of the bundle associated to the projections $p_V:\SU(V,W)\to U(V)$, $p_W:\SU(V,W)\to U(W)$. We may decompose the bundles $\mathcal{E}_V$ and $\mathcal{E}_W$ as usual: this gives $\sch(i\circ \rho)=\ch(\mathcal{E}_V^+)+\ch(\mathcal{E}_W^+)-\ch(\mathcal{E}_V^-)-\ch(\mathcal{E}_W^-)$. From Proposition \ref{projplat} and the fact that $\mathcal{E}_V$ and $\mathcal{E}_W$ are flat, we get $\sch(p_V\circ \rho)=\ch(\mathcal{E}_V^+)-\ch(\mathcal{E}_V^-)$ and the same for $W$, showing the result.  
\end{proof}
\subsection{A differential definition of the super Chern character}

In this subsection, we give an alternative definition of the super Chern character having a differential flavour, in the case of a representation defined on the fundamental group of a smooth orbifold.

Recall that for a Hermitian vector space $V$ of signature $(p,q)$ we denoted by $\H^{p,q}$ the space of either, positive $p$-dimensional subspaces $V^+$, negative $q$-dimensional subspaces $V^-$, or orthogonal decompositions $V=V^+\oplus V^-$. 
The group $\PU(p,q)$ acts transitively on these decompositions and the stabilizer of $V^+\oplus V^-$ is the maximal compact subgroup $\PU(V^+,V^-)=\mathrm{P}(\U(p)\times \U(q))$, showing that $\H^{p,q}$ is the symetric space of $\PU(p,q)$.  

Let us define a family of \(\PU (p,q)\)-invariant differential forms \(\omega_k\) of degree \(2k\) on \(\H^{p,q}\).  The tangent space of \( \H^{p,q}\) at a point \(V^+\) is naturally identified with the space \( \Hom (V^+, V^-)\) and the complex structure on this space induces a complex structure on $\H^{p,q}$: together with $\omega_1$, this gives the Kähler structure on $\H^{p,q}$. The adjonction map $\alpha\mapsto \alpha^*$ gives an anti-linear isomorphism $\Hom(V^+,V^-)\simeq \Hom(V^-,V^+)$. 

\begin{Remark}
It is also possible to identify $T_{(V^+,V^-)}\H^{p,q}$ with $\Hom(V^-,V^+)$, but it gives the opposite complex structure on $\H^{p,q}$. This also corresponds to changing the Hermitian form $h$ to its opposite, or said informally, exchanging $p$ and $q$. We have to take great care of this subtlety which occurs everywhere in the article.
\end{Remark}

For any family \(\xi_1, \ldots, \xi_{2k}\in \Hom (V^+, V^-)\), we set: 
\begin{small}
\[ \omega_k ( \xi_1, \ldots, \xi_{2k})=
\frac{2^{1-k}}{(2i\pi)^k k !} \text{Tr} \sum _{\sigma\in S_{2k} } \varepsilon(\sigma)  \prod_{i=1}^k \Big(\xi_{\sigma(2i-1)}^* \xi_{\sigma(2i)} - \xi_{\sigma(2i)}^* \xi_{\sigma(2i-1)}\Big).\]
\end{small}
In this formula \(S_{2k}\)  is the group of permutations of \(\{1,\ldots, 2k\}\) and \(\varepsilon (\sigma)\) is the signature of \(\sigma\in S_{2k}\).

\begin{Lemma}
Assume that \(X\) is a developable orbifold and \(\rho : \pi _1 (X) \rightarrow \SU (p,q)\) is a morphism. Then, for any smooth \(\rho \)-equivariant map \( f: \widetilde{X} \rightarrow \H^{p,q}\), 
the form \( f^* \omega_k \), which is invariant by \(\pi _1(X)\) and thus descends to a differential form of degree \(2k\) on \(X\), is a De Rham representative of \( \sch_k (\rho)\) in \( H^{2k}(X,\mathbb R) \).
\end{Lemma}

\begin{proof}
Recall that by assumption, the orbifold universal cover $\tilde{X}$ of $X$ is smooth. By \cite[Theorem 2.4]{Yamasaki}, there exists smooth $\rho$-equivariant maps $f:\tilde{X}\to \H^{p,q}$ and those are unique up to homotopy.

Let \(\mathcal{E}^+\rightarrow \H^{p,q}\) be the rank $p$ positive tautological vector bundle, whose fiber over the point $(V^+,V^-)$ is the subspace \(V^+\). We observe that \(\mathcal E^+\) is naturally a sub-bundle of the trivial bundle \(\V=V\times \H^{p,q}\) and denote by \(\pi : \V \rightarrow \mathcal{E}^+\) the orthogonal projection with respect to the Hermitian form. We use the trivial connection \(D\) on \(\V\) to define a connection \(\nabla\) on \(\mathcal E^+\)  by 
\[\nabla_\xi s = \pi  D_\xi s ,\]
where \(s\) is any smooth section of \(\mathcal E^+\) and \(\xi\) any vector field on \(\H^{p,q}\). An painful but elementary computation shows that the curvature \(\Omega_{\nabla}(\xi,\eta) =\nabla _\xi\nabla _\eta - \nabla _\eta \nabla _\xi - \nabla _{[\xi,\eta]} \) of this connection is given by the simple formula
\[ \Omega _{\nabla} (\xi,\eta) = \eta^* \xi - \xi^* \eta ,\]
where as before \(\xi,\eta\) are considered as elements of \( \Hom (\mathcal E^+, \mathcal E^-)\). 
Notice that 
\[ \text{Tr} \exp  \left( \frac{-\Omega_\nabla }{2i\pi} \right)  = \frac{1}{2}\sum_{k\ge 0}\omega_k \]
hence by Chern-Weil theory, the forms $\frac{1}{2}\omega_k$ represent the Chern character of $\mathcal{E}^+$ on $\H^{p,q}$. 
Consider a smooth \(\rho\)-equivariant map \( f: \widetilde{X} \rightarrow \H^{p,q} \). Pulling back $\mathcal{E}^{\pm}$ gives rise to orthogonal sub-bundles \( \mathcal E ^{\pm} _\rho\) of the flat hermitian bundle $\mathcal{E}_\rho$ of fiber $V$ and monodromy \(\rho\) over \(X\). We have \( \ch(\mathcal E_{\rho}^+) + \ch (\mathcal E_{\rho}^-)=\ch (\mathcal{E}_\rho) = p+q\) so \(\sch_k (\rho)= \ch_k(\mathcal E_{\rho}^+) - \ch_k (\mathcal E_{\rho}^-) = 2 f^* \ch_k(\mathcal E^+)\) for $k>0$. We then deduce the result from the fact that the pull-back of the connection \(\nabla\) to \(\mathcal E^+_\rho\) defines a connection whose curvature is \( f^* \Omega_{\nabla}\).  
\end{proof}
It would be interesting to give an analogous geometric construction in the projective case, the following section gives one possible way.

\subsection{Relation to the tangent bundle of the symmetric space}
Let $X$ be a connected topological space and $\rho:\pi_1(X)\to \PU(p,q)$ be a representation. We choose $f:\tilde{X}\to \H^{p,q}$, a continuous $\rho$-equivariant map. 
We form the complex vector bundle \( \mathcal F \) over \(X\) defined as the quotient of \(f^* T \mathcal \H^{p,q}\) by the action of \(\pi_1(X)\) given by 
\[ \gamma \cdot (x,\xi) = (\gamma(x), D_{f(x) }  \rho (\gamma) \xi)\] for any \(x \in \tilde{X}\), any \(\xi\in T_{f(x)} \H^{p,q}\), and any \(\gamma\in \pi_1(X)\). The map \(f\) is well-defined up to \(\rho\)-equivariant homotopy, so the complex vector bundle \(\mathcal F\) is well-defined. 

\begin{Lemma} \label{l: pull-back tangent bundle}
 For odd \(k\), we have $\sch_k (\rho) = \frac{-2}{p+q} ch_k (\mathcal F).$
\end{Lemma}
 In particular \( \frac{(p+q)k!}{2}\sch_k(\rho)\) is an integral class. For even $k$ and $p\ne q$, we can also express $\sch_k$ as a polynomial in the Chern character of $\mathcal{F}$. For instance 
 \[\sch_2 (\rho) = \frac{-2}{p-q}\left(\ch_2(\mathcal F) - \frac{1}{(p+q)^2}\ch_1(\mathcal F) ^2  \right).\]
 \begin{proof}
 The construction being natural in $X$, we can suppose as in the proof of Proposition \ref{projplat} that the representation $\rho$ lifts to $\SU(p,q)$. We can then define the two associated bundles $\mathcal{E}_\rho^\pm$ so that $\mathcal{F}=\Hom(\mathcal{E}_\rho^+,\mathcal{E}_\rho^-)=(\mathcal E_\rho^+)^*\otimes \mathcal E^-_\rho$. The proof follows by inspection of the following identities; $\sch(\rho)=\ch(\mathcal E_\rho^+)-\ch(\mathcal E_\rho^-)$, $\ch(\mathcal E_\rho^+)+\ch(\mathcal E_\rho)^-=p+q$ and 
 $$\ch(\mathcal F)=\Big(p-\ch_1(\mathcal E_\rho^+)+\ch_2(\mathcal{E}_\rho^+)-\cdots\Big)\Big(q+\ch_1(\mathcal E_\rho^-)+\ch_2(\mathcal E_\rho^-)+\cdots\Big).$$
 \end{proof}
This lemma has the following important consequence:

\begin{Corollary}\label{cor:unif}
Suppose that a complex orbifold $X$ is locally modeled on the symmetric space $\H^{p,q}$ and let $\rho:\pi_1(X)\to \PU(p,q)$ be its monodromy representation. 
Then, denoting by $K_X$ the canonical bundle of $X$, we have 
$$\sch_1(\rho)=\frac{2}{p+q}c_1(K_X).$$
\end{Corollary}

In particular, the holonomy of a \(\H^{1,1}\)-structure on a closed oriented surface \(S\) of genus \(g\geq 2\) satisfies \( \int_S \sch_1(\rho) = 2g-2\). 

\begin{Remark}
Corollary \ref{cor:unif} gives a necessary condition for the uniformization of a representation \(\rho: \pi_1 (X)\rightarrow \text{PU} (p,q)\) by a \(\mathbb H^{p,q}\)-structure which is almost sufficient, thanks to Siu's rigidity theory (see Lemma \ref{l: criterion}). 
\end{Remark}

\subsection{The Toledo class as an obstruction class}\label{obstruction}

The purpose of this section is to identify the Toledo class $\sch_1(\rho)$ with an obstruction class. 

For any connected Lie group $G$, we derive from the homotopy sequence of the fibration $G\to EG\to BG$ that $\pi_n(BG)=\pi_{n-1}(G)$, in particular $BG$ is simply connected and from the Hurewicz theorem, we get $H_2(BG,\Z)=\pi_2(BG)=\pi_1(G)$. 
The universal coefficient theorem gives the isomorphism $H^2(BG,\Q)=\Hom(\pi_1(G),\Q)$. Taking $G=\PU(p,q)$, the class $\sch_1\in H^2(\PU(p,q),\Q)$ corresponds to a map $\phi:\pi_1(\PU(p,q))\to\Q$. 

Recall that the maximal compact subgroup of $\PU(p,q)$ is $\mathrm{P}(\U(p)\times \U(q))$ and have the same fundamental group. The exact sequence of the fibration $\U(1)\to \U(p)\times \U(q)\to \mathrm{P}(\U(p)\times \U(q))$ gives the description $\pi_1(\PU(p,q))=\Z^2/(p,q)\Z$. 
\begin{Lemma}
The map $\phi:\pi_1(\PU(p,q))=\Z^2/(p,q)\Z\to \Q$ associated to $\sch_1$ by the above procedure is $$\phi(x,y)=\frac{2}{p+q}(xq-py).$$
\end{Lemma}
This lemma tells that $\sch_1(\rho)$ can be computed by the following constructive procedure. Consider the central extension 
$$0\to \pi_1(\PU(p,q))\to \tilde{\PU}(p,q)\to \PU(p,q)\to 0.$$
The obstruction of lifting $\rho:\pi_1(X)\to \PU(p,q)$ to  $\tilde{\PU}(p,q)$ is a class $o(\rho)\in H^2(\pi_1(X),\pi_1(\PU(p,q))$ that we can map to $H^2(X,\pi_1(\PU(p,q)))$ (using a map $f:X\to B\pi_1(X)$ inducing the identity on fundamental groups). The lemma claims that one has 
$$\sch_1(\rho)=\phi_*(o(\rho)).$$
\begin{proof}
As $\pi_1(\PU(p,q))\otimes \Q$ has dimension 1, this is just a question of normalization. 
Consider first the case $p=q=1$. Then $\pi_1(\PU(1,1))=\Z$ and $\phi:\Z\to \Q$ is the standard inclusion. We have to take care of the orientation here: the loop $\gamma(\theta)=(e^{i\theta},1)\in \PU(1,1)$ corresponds to the positive generator. 
Recall that $\H^{1,1}=\{[v]\in \P V, h(v)>0\}$ is a hyperbolic disc, naturally oriented by its complex structure: we check that $\gamma(\theta)$ acts by rotation of angle $-\theta$ on $\H^{1,1}$: the two orientations disagree.  

We take a surface $S$ of genus $g$ with a $\H^{1,1}$-structure and holonomy representation $\rho:\pi_1(S)\to \PU(1,1)$. The obstruction of lifting it to the universal cover is the Euler class, which in this case is known to be equal to the Euler characteristic $2-2g$. By Corollary \ref{cor:unif}, we have $\int_S\sch_1(\rho)= 2g-2$. The change of sign observed in the previous paragraph makes this formula agree: we have in this case $\sch_1=-\rm{eu}$. 

In the general case, we consider a decomposition $V\oplus W$ where $V$ has signature $(1,1)$ and a representation $\rho:\pi_1(X)\to\SU(V)$, trivially extended to $\SU(V\oplus W)$. The additivity formula of Proposition \ref{additivity} gives $\sch_1(\rho_{V\oplus W})=\sch_1(\rho_V)$. It suffices to check that the following diagram commutes.
$$\xymatrix{ \pi_1 \SU(1,1) \ar[d]^{\phi_{1,1}} \ar[r]& \pi_1 \SU(p,q)\ar[d]^{\phi_{p,q}}\\
\Q\ar[r]^\id& \Q}$$

To check it, consider an element $(x,-x)\in \pi_1 \SU(1,1)$. We compute $\phi_{1,1}(x,-x)=2x$. When mapping $\SU(1,1)$ to $\SU(p,q)$ as above, the element $(x,-x)$ stays equal. This time we compute $\phi_{p,q}(x,-x)=\frac{2}{p+q}(qx-p(-x))=2x$. This proves the result.
\end{proof}

\section{Hermitian modular functors}
\subsection{Marked surfaces}

A marked surface is a triple $(S,\phi,L)$ where 
\begin{enumerate}
    \item $S$ is a compact oriented surface whose boundary is the disjoint union of the components $\partial_i S$ for $i\in  \pi_0(\partial S)$,
    \item $\phi$ is a collection of homeomorphisms preserving the orientation $\phi_i:S^1\to \partial_i S$ for $i\in \pi_0(\partial S)$. 
    \item $L$ is a split Lagrangian in $H_1(S,\Q)$. This means that $L=\bigoplus_{i\in \pi_0(S)} L_i$, where $L_i$ is a Lagrangian in $H_1(\dot{S}_i,\Q)$ and $\dot{S}_i$ is the $i$-th connected component of $S$ where each boundary curve has been collapsed to a point. 
\end{enumerate}
Frequently, we will denote only by $S$ the marked surface $(S,\phi,L)$. A morphism $(S,\phi,L)\to (S',\phi',L')$ is a pair $(f,s)$ where $f:S\to S'$ is a homeomorphism preserving the orientation and satisfying $\phi=\phi'\circ f$ and $s\in \Z$ is an integer. The composition of $(f,s):S_1\to S_2$ and $(g,t):S_2\to S_3$ is $$(g\circ f,s+t-\operatorname{Maslov}((gf)_*L_1,g_*L_2,L_3)).$$

We can define three operations on marked surfaces: the disjoint union, the change of orientation and the gluing operation. Only the third one deserves an explanation. Pick $\partial_+S$ and $\partial_-S$ two components of $\partial S$. We define $S_{\pm}$ to be the result of identifying 
$\phi_+(z)$ and $\phi_-(\overline{z})$ for any $z\in S^1$. 
Denote by $\dot{S}$ the surface obtained from $S$ by collapsing $\partial S_+$ and $\partial S_-$ to a point. There are natural maps $S\to \dot{S}\leftarrow S_{\pm}$. We define $L_{\pm}$ to be the preimage in $H_1(S_\pm,\Q)$ of the image of $L$ in $H_1(\dot{S},\Q)$.
The triple $(S_\pm,\phi,L_\pm)$ is the gluing of $S$ along $\partial_{\pm}S$. 

\subsection{Hermitian modular functor}\label{modular}

Let $\Lambda$ be a finite set endowed with an involution $\lambda\mapsto \lambda^*$ and a unit $0\in \Lambda$ satisfying $0^*=0$.
A $\Lambda$-coloring of a surface $S$ is a map $\lambda:\pi_0(\partial S)\to \Lambda$. A Hermitian modular functor is a functor $\mathcal V$ from the category of $\Lambda$-colored marked surfaces to the category of Hermitian vector spaces satisfying the following axioms.

{\bf MF1:} Monoidality (simplified). There are compatible isomorphisms
$$\mathcal{V}((S,\lambda)\amalg(S',\lambda'))\simeq \mathcal{V}(S,\lambda)\otimes \mathcal{V}(S',\lambda')$$

{\bf MF2:} Gluing: there is a natural isomorphism
$$\mathcal{V}(S_\pm,\lambda)\simeq\bigoplus_{\mu\in \Lambda}\mathcal{V}(S,\mu,\mu^*,\lambda)\otimes\mathcal{V}(S^2,\mu,\mu^*)$$

{\bf MF3:} Change of orientation. There is a natural perfect pairing
$$\mathcal{V}(S, \lambda)\times \mathcal{V}(- S,\lambda^*)\to \C$$

{\bf MF4:} Sphere with 1 point.
$$\dim \mathcal{V}(S^2,\lambda)=1\text{ if }\lambda=0, 0\text{ otherwise}.$$

{\bf MF5:} Sphere with 2 points.
$$\dim \mathcal{V}(S^2,\lambda,\mu)=1\text{ if }\lambda=\mu^*, 0\text{ otherwise}.$$

In these axioms, we make several shortcuts in the notation to keep it light. When we add $\lambda$ to a marked surface, it means either that we color by $\lambda$ part or all of the boundary components or even that we create a boundary component that we color with $\lambda$, depending on the context. 
Let us name some important constants associated to a Hermitian modular functor. 

\begin{enumerate}
\item For any (Hermitian) modular functor, any morphism of the form $(\id,1)$ acts on $\mathcal{V}(S,\lambda)$ multiplying by $c\in \C^*$. 
One can prove from the axioms that this number is independent on $S$ and $\lambda$ and is called the \emph{central charge} of the modular functor.

\item On $\mathcal{V}(S^2,\lambda,\lambda^*)$ acts the Dehn twist  $T_\delta$ along a simple curve separating $x$ and $y$. As this space is $1$-dimensional, $(T_\delta,0)$ acts multiplying by $r_\lambda=r_{\lambda^*}\in \C^*$. We will call these constants the \emph{multipliers} associated to the colors. 

\item On this latter space which is $1$-dimensional, the Hermitian form is definite. We denote by $\epsilon_\lambda=\epsilon_{\lambda^*}$ its sign. 
\end{enumerate}

It is known that $c$ and $r_\lambda$ are always roots of unity. 

\begin{Definition}
The level of a modular functor $\V$ is an integer $\ell$ such that $r_\lambda^\ell=1$ for all $\lambda\in \Lambda$.
\end{Definition}
Our main example is the Fibonacci TQFT for which we have $\ell=5$. We warn the reader that there is a shift with the level commonly used in Conformal Field Theory.

\subsection{The associated cohomological field theory}

We set $V=\Q[\Lambda]$. We first define for all $g,n\ge 0$, 

$$\omega_{g,n}\in \Hom(V^{\otimes n},H^*(\overline{\M}_{g,n},\Q)).$$


It reduces to define for every genus $g$ and for any $\lambda_1,\ldots,\lambda_n\in \Lambda$ a class $\omega_{g,n}(\lambda_1,\ldots,\lambda_n)\in H^*(\overline{\M}_{g,n},\Q)$.

\bigskip
Let $S$ be a surface with genus $g$ and $n$ boundary components. We recall that $\Mod(S)$ is the group of isotopy classes of homeomorphisms of $S$ fixing the boundary pointwise. Pick a marking $\phi$ of the boundary of $S$ and a coloring $\lambda:\pi_0(\partial S)\to \Lambda$. The automorphism group of $S$ in the category of marked surface is a central extension of $\Mod(S)$ (its class is given by $\lambda_1\in H^2(\Mod(S),\Z)$ but it does not matter here). As the modular functor is Hermitian and sends the central element to $c\id$, we get a representation
$$\rho_\lambda:\Mod(S)\to \PU \big(\mathcal{V}({\bf S},\lambda)\big)$$

Let $\delta$ be a simple curve parallel to a boundary component of $S$ colored by $\mu$. From the axioms, the Dehn twist $T_\delta$ acts by multiplication by $r_{\mu}$, hence trivially in the projective unitary group. This means that the representation $\rho_\lambda$ factors trough the group $\Mod(\dot{S},P)$ where $\dot{S}$ is the surface obtained by collapsing each boundary component of $S$ to a point, and $P$ is the set of resulting marked points. 

Finally, the axiom MF2 shows that every Dehn twist $T_\delta$ is diagonalizable with eigenvalues $r_\mu$ for $\mu\in \Lambda$. In particular $T_\delta^\ell$ acts trivially, hence $\rho_\lambda$ factors through a representation 
$$\rho_\lambda:\Mod^\ell(\dot{S},P)\to \PU \big(\mathcal{V}(S,\lambda)\big)$$
where $\Mod^\ell(\dot{S},P)$ is the quotient of the mapping class group $\Mod(\dot{S},P)$ by the (normal) subgroup generated by $\ell$-th powers of Dehn twists. 

As we showed in Section \ref{twistedmoduli} that $\pi_1(\Mb^\ell(\dot{S},P))=\Mod^\ell(\dot{S},P)$, the construction of Section \ref{toledo} defines a class
$$\omega_{g,n}(\lambda)=\sch(\rho_\lambda)\in H^*(\Mb^\ell(\dot{S},P),\Q)=H^*(\Mb_{g,n},\Q)$$
The last equality is due to the fact that the orbifolds $\Mb^\ell(\dot{S},P)$ and $\Mb_{g,n}$ have the same underlying topological space, hence the same rational cohomology.

\subsection{Proof of the CohFT axioms}

We define on $V$ the bilinear form $\eta(\lambda,\mu)=\sign\mathcal{V}(S^2,\lambda,\mu)$. By MF5, $\eta(\lambda,\mu)=\epsilon_\lambda$ if $\lambda= \mu^*$ and $0$ otherwise. We refer to \cite{Pandharipande} for details on the axioms of a CohFT, here we recall them at the same time that we prove them. The first one is a compatibility of the construction with the action of the symmetric groups permuting the colors and the marked point. It is satisfied by construction.

We recall that the product on $V$ is defined by the formula $\omega_{0,3}(\lambda,\mu,\nu)=\eta(\lambda\cdot \mu,\nu)$. It follows from the axioms of a Hermitian modular functor that the unit $1\in V$ corresponds to the color $0\in\Lambda$. We will try not to confuse the reader using both notations.

\subsubsection{Forgetting a point}
Let $\pi:\overline{\M}_{g,n+1}\to\overline{\M}_{g,n}$ the map which forgets the last marked point. One needs to check 
\begin{equation}\label{oubli}
\pi^* \omega_{g,n}(\lambda_1,\ldots,\lambda_n)=\omega_{g,n+1}(\lambda_1,\ldots,\lambda_n,0).
\end{equation}

Let $S$ be a marked surface and $S^\circ$ be the result of removing a disc in the interior of $S$. Corresponding to $\pi$, there is a morphism $p:\Mod(S^\circ)\to \Mod(S)$ obtained by gluing back the disc. 
This morphism $p$ induces a map $\Mod^\ell(S_{g,n+1})\to \Mod^\ell(S_{g,n})$ which is the morphism induced on the fundamental groups by the map $\pi:\Mb^\ell_{g,n+1}\to \Mb^\ell_{g,n}$. 

From the axioms of the modular functor, there is a $p$-equivariant isomorphism $\mathcal{V}( S^\circ,\lambda,0)=\mathcal{V}( S, \lambda)$ which fits in the following commutative diagram:
$$\xymatrix{
\pi_1 \Mb^\ell_{g,n+1}\ar[d]^{\pi_*}\ar[r]^\sim & \Mod^\ell(S_{g,n+1})\ar[r]^{\rho_{\lambda,0}}\ar[d]^\ell & \PU\big(\mathcal{V}( S,\lambda,0)\big)\ar[d]^\sim \\
\pi_1 \Mb^\ell_{g,n}\ar[r]^\sim &\Mod^\ell(S_{g,n})\ar[r]^{\rho_{\lambda}} & \PU\big(\mathcal{V}( S,\lambda)\big) 
}$$

Equation \eqref{oubli} hence follows from the naturality of the class $\sch$.

\subsubsection{Non-separating gluing}

Let $\pi:\overline{\M}_{g,n+2}\to \overline{\M}_{g+1,n}$ the map which glue the two last points. The second axiom of a CohFT to be checked is 

\begin{equation}\label{formuleadd}
\pi^*\omega_{g+1,n}(\lambda_1,\ldots,\lambda_n)=\sum_{\mu\in \Lambda}\epsilon_\mu\omega_{g,n+2}(\lambda_1,\ldots,\lambda_n,\mu,\mu^*).
\end{equation}

We consider this time a marked surface $S=(S,\phi,L)$ with two special boundary components $\partial_+S$ and $\partial_-S$. As in Section \ref{modular}, we denote by $S_{\pm}$ the result of gluing these components using their parametrization. 
Again, there is a natural morphism $p:\Mod(S)\to \Mod(S_{\pm})$ which induces a morphism $p:\Mod^\ell(S_{g,n+2})\to \Mod^\ell(S_{g+1,n})$. This morphism is the one induced by $\pi:\Mb^\ell_{g,n+2}\to \Mb^\ell_{g+1,n}$ on fundamental groups. We get hence a picture very similar to the preceding section. 

The main difference is that axiom MF2 gives a decomposition of $\mathcal{V}(S_{\pm},\lambda)$ which is preserved by the action of $\Mod(S_\pm,\delta)=p(\Mod(S_{g,n+2}))$ where $\delta$ is the common image of the glued boundaries. This decomposition corresponds to the eigenspace decomposition of the Dehn twist $T_\delta$. The situation is better visualized in the following diagram:

$$\xymatrix{
\pi_1 \Mb^p_{g,n+2}\ar[d]^{\pi_*}\ar[r]^\sim & \Mod^p(S_{g,n+2})\ar[r]_{\sum\rho_{\lambda,\mu,\mu^*}}\ar[d]^p & \PU\big(\mathcal{V}({\bf S},\lambda,\mu,\mu^*)_{\mu\in \Lambda}\big)\ar[d]^\sim \\
\pi_1 \Mb^p_{g+1,n}\ar[r]^\sim &\Mod^p(S_{g+1,n})\ar[r]^{\rho_{\lambda}} & \PU\big(\mathcal{V}({\bf S}_{\pm},\lambda)\big) 
}$$
The formula \eqref{formuleadd} follows then from the additivity formula of Proposition \ref{additivity}, taking into account that each factor $\mathcal{V}(S,\lambda,\mu,\mu^*)$ appears tensored by the one-dimensional space $\mathcal{V}(S^2,\mu,\mu^*)$ which has sign $\epsilon_\mu$. This factor acts on the class $\sch$ by the global sign $\epsilon_\mu$ (as a consequence of the multiplicativity property).  

\subsection{Separating gluing}\label{ss: separated gluing}
Let $\pi:\Mb_{g_1,n_1+1}\times \overline{\M}_{g_2,n_2+1}\to \overline{\M}_{g_1+g_2,n_1+n_2}$ the map obtained by gluing the last points. This time we must check that 
\begin{equation}\label{formuleprod}
\pi^* \omega_{g,n}(\lambda_1,\lambda_2)=\sum_{\mu\in \Lambda}\epsilon_\mu\omega_{g_1,n_1+1}(\lambda_1,\mu)\otimes \omega_{g_2,n_2+1}(\lambda_2,\mu^*).
\end{equation}

Here the tensor product makes sense using the Künneth formula. 

Again, consider two marked surfaces $S_1,S_2$ with respective genus $g_1,g_2$ and respectively $n_1+1$ and $n_2+1$ boundary components. The operation of gluing the last components produces a surface $S$ of genus $g=g_1+g_2$ and $n=n_1+n_2$ boundary components together with a map $p:\Mod(S_1)\times \Mod(S_2)\to \Mod(S)$ inducing a map $\Mod^\ell(S_{g_1,n_1+1})\times \Mod^\ell(S_{g_2,n_2+1})\to \Mod^\ell(S_{g,n})$. 
The situation is very similar to the one of the previous section: this time the group $\Mod^\ell(S_{g,n},\delta)$ which is the image of $p$ preserves the decomposition

$$\mathcal{V}(S,\lambda_1,\lambda_2)=\bigoplus_{\mu\in \Lambda} \mathcal{V}(S_1,\lambda_1,\mu)\otimes \mathcal{V}( S_2,\lambda_2,\mu^*)\otimes \mathcal{V}(S^2,\mu,\mu^*).$$

As in the previous section, and using this time the multiplicative property of $\sch$ given in Proposition \ref{multiplicativity}, we obtain a proof of Equation \eqref{formuleprod}, which ends the proof that $\omega_{g,n}$ satisfies the axioms of a CohFT. 

\section{Computation of the CohFT associated to the $\SU_2/\SO_3$-modular functors}\label{examples}

The purpose of this section is to give some detail on two interesting families for which the construction of the preceding section applies. The degree 0 part of those CohFTs (usually called Topological Field Theories or Frobenius algebras) are already interesting and new as they provide formulas for the signatures of TQFT as investigated in \cite{FP}.

\subsection{Semi-simplicity of the Frobenius algebras}

\subsubsection{Generalities on Frobenius algebras}
Let us start with generalities about Frobenius $\Q$-algebras. They are by definition finite dimensional $\Q$-algebras $V$ endowed with a linear form $\epsilon:V\to \Q$ such that the bilinear form $\eta(x,y)=\epsilon(xy)$ is non-degenerate. 

Consider its inverse $\eta^{-1}\in V\otimes V$ : composing  with the multiplication $m:V\otimes V\to V$, we get an element $\Omega=m(\eta^{-1})\in V$. It is well-known and easy to check that any TFT $\omega_{g,n}$ with underlying Frobenius algebra $V$ satisfies 
$$\omega_{g,n}(v_1,\ldots,v_n)=\epsilon(v_1\cdots v_n\Omega^g)$$

In particular, the signature of the Hermitian vector space associated to a genus $g$ surface by a modular functor of Frobenius algebra $V$ is $\omega_{g,0}=\epsilon(\Omega^g)$. This is a generalization of the Verlinde formula. 

A crucial property of a Frobenius algebra is its semi-simplicity, holding if and only if it is isomorphic to a product of number fields. Denote by $M_x\in \End(V)$ the operator of multiplication by $x\in V$ and by $\tr_V:V\to \Q$ the trace form given by $\tr_V(x)=\tr (M_x)$. A property equivalent to semi-simplicity is that the bilinear pairing $(x,y)\mapsto \tr_V(xy)$ is non-degenerate.

Hence a Frobenius structure on a semi-simple algebra is given by an invertible element $\alpha\in V^\times$ satisfying $\epsilon(x)=\tr_{V}(\alpha x)$. It looks like in the most interesting cases of $\SO_3$-modular functors of prime level, the algebra $V$ is a number field. 

\begin{Lemma}
If $V$ is a semi-simple Frobenius algebra associated to $\alpha\in V^\times$ then $\Omega=\alpha^{-1}$. In particular, $$\omega_{g,n}(v_1,\ldots,v_n)=\tr_{V}(v_1\cdots v_n\alpha^{1-g}).$$
\end{Lemma}
\begin{proof}
By Artin-Wedderburn theorem, we can reduce to the case when $V$ is a number field. 
The computation of $\Omega$ can be done in $V\otimes \C$ which is isomorphic to $\C^n$ via the map $x\mapsto(\phi_1(x),\ldots,\phi_n(x))$ where $\phi_1,\ldots,\phi_n$ denote the embeddings $V\hookrightarrow \C$. The linear form $\epsilon$ on the $i$-th factor is the multiplication by $\phi_i(\alpha)$, hence the element $\Omega$ on the $i$-th factor is $\phi_i(\alpha^{-1})$, proving the lemma. 
\end{proof}

A nice example is given by the celebrated Verlinde formula which compute the  dimension of the modular functors. In the next sections, considering the $\SO(3)$-modular functor associated to a specific root of unity $\zeta$ of prime order, we will find a unitary modular functor whose CohFT reduces to its degree $0$ part. In that case, $V$ is the subfield of $\Q(\zeta)^+$ fixed by the involution $\zeta\mapsto \zeta^{-1}$ and $\alpha=-\frac{(\zeta-\zeta^{-1})^2}{\ell}$.

Using the formula $\tr_{V}(x)=\sum\limits_{i=1}^n\phi_i(x)$, we get the Verlinde formula:
$$\dim \mathcal{V}_\ell(S_g)=\Big(\frac{\ell}{4}\Big)^{g-1}\sum_{m=1}^{\frac{\ell-1}{2}}\sin\Big(\frac{2m\pi}{\ell}\Big)^{2-2g}.$$
We will get similar formulas for the signatures with the twist that the conjugates $\phi_i(\alpha)$ will no longer have an explicit expression.


\subsubsection{The $\SU_2$-modular functor}
We set $r\ge 2$ and choose $A$ to be a primitive $4r$-th root of unity. The construction in \cite{BHMV} produces a Hermitian modular functor $\V_A$ from this data.

The set of colors is $\Lambda=\{0,1,\ldots,r-2\}$ with the trivial involution. The multiplicators are $r_i=(-1)^i A^{i(i+2)}$ and the signs are $\epsilon_i=\sign((-1)^i [i+1])$ where $[n]=\frac{A^{2n}-A^{-2n}}{A^2-A^{-2}}$. The level of this theory in our sense is $\ell=4r$.

In the sequel, we denote by $V_A=\Q e_0\oplus\cdots\oplus\Q e_{r-2}$ the Frobenius algebra underlying the CohFT associated to $\V_A$. 

Recall that the bilinear form $\eta$ is diagonal in this basis and satisfies $\eta(e_i,e_i)=\epsilon_i$. As in any Frobenius algebra the product is given by 
$$\eta(e_ie_j,e_k)=\omega_{0,3}(e_i,e_j,e_k)=\sign \V_A(S^2,i,j,k).$$

The space $\V(S^2,i,j,k)$ is one dimensional if one can write $i=b+c,j=a+c,k=a+b$ for some integers $a,b,c\in \N$ and $0$ if we cannot.  In the first case, we find in \cite[Lemma 4.2]{BHMV} the formula $\omega_{0,3}(i,j,k)=\sign \langle i,j,k\rangle$ where 
$$\langle i,j,k\rangle=(-1)^{a+b+c}\frac{[a+b+c+1]![a]![b]![c]!}{[i]![j]![k]!}.$$
Here we used the quantum factorial $[n]!=[1][2]\cdots [n]$.
\begin{Proposition}[$\SU_2$ case]\label{semisimplepair}
For any $4r$-th root of unity $A$, $V_A$ is semi-simple. 
\end{Proposition}
\begin{proof}
 We check from the above formulas that $e_1e_i=\frac{\epsilon_i}{\epsilon_{i-1}}e_{i-1}+e_{i+1}$ if we set $e_{-1}=e_{r-1}=0$. This proves that $e_1$ generates $V_A$ as an algebra and hence the natural surjection  $\Q[t]/P(t)\to V_A$ is an isomorphism where $P(t)=\det(M_{e_1}-t\id)$ and $M_{e_1}$ is the matrix of the multiplication by $e_1$ on $V_A$. This matrix has the simple form 
$$M_{e_1}=\begin{pmatrix} 0 & \epsilon_1/\epsilon_0 & 0 & 0& 0\\
1 & 0 & \epsilon_2/\epsilon_1 & 0 & 0 \\
\vdots &\vdots &\vdots &\vdots &\vdots &\\
0 & 0 & 1 & 0 & \epsilon_{r-2}/\epsilon_{r-3}\\
0 & 0 & 0 & 1 & 0
\end{pmatrix}.
$$
It is an exercise, left to the reader, that these kind of Jacobi matrices have a simple spectrum, which implies that the algebra $V_A$ is semi-simple. 
\end{proof}
When $A=\pm e^{\frac{i\pi}{2r}}$ we get $\epsilon_i=(-1)^i$ and $\omega_{0,3}(e_i,e_j,e_k)=(-1)^{(i+j+k)/2}$ when it is non zero. 
In this case the modular functor is Hermitian in the standard sense (the Hermitian form is definite) and the CohFT constructed above reduces to its degree 0 part. The Frobenius algebra we thus obtained is the Verlinde fusion algebra, described in many places, see \cite{Beauville,BHMV}.

It would be interesting to investigate the properties of these Frobenius algebras. Here, we directly skip to the $\SO_3$-case which gives lower dimensional and often simple Frobenius algebras. Moreover the corresponding representations of the mapping class group are irreducible, have good arithmetic properties if the level is prime, and contain the main example of this article, Fibonacci modular functor.

\subsubsection{The $\SO_3$-modular functor}\label{FrobSO3}

We choose $A$ to be a primitive $2\ell$-th root of unity where, this time, $\ell$ is odd. This corresponds in \cite{BHMV} to a modular functor with group $\SO_3$ and our main example concerns the case when $\ell=5$. 
In this case, the set of colors is $\Lambda=\{0,2,\ldots,\ell-3\}$, the involution is trivial and the multiplicators and the signs are given by the same formulas as above. Precisely, we have $\epsilon_{2i}=[2i+1]$ and $\mu_{2i}=A^{4i(i+1)}$ and we check that $\mu_{2i}^\ell=1$ for all $i$ so that $\ell$ is the level of this theory.  

Set $q=A^2$ and $V_q=\Q e_0\oplus\Q e_1\oplus\cdots\oplus \Q e_{\frac{\ell-3}{2}}$ where $e_i$ represents the color $2i$. This time, the Frobenius algebra depends only on $q$, hence the notation. We have $\omega_{0,3}(e_i,e_j,e_k)=\sign\langle 2i,2j,2k\rangle$ if $i,j,k$ satisfy 
\begin{equation*}\tag{T} i\le j+k,\,j\le i+k,\,k\le i+j\text{ and }i+j+k<l-1
\end{equation*}
and $0$ otherwise. The root giving a Hermitian theory is $q=e^{\pm 2i\pi\frac{\ell-1}{\ell}}$. We claim that Proposition \ref{semisimplepair} also holds when $\ell$ is odd. 

\begin{Proposition}[$\SO_3$ case]\label{semisimpleimpair}
For any $\ell$-th root of unity $q$, $V_q$ is semi-simple. 
\end{Proposition}
\begin{proof}
We compute this time that 
$$e_1 e_{i}=\sign\Big(\frac{[2i+1]}{[2i-1]}\Big)e_{i-1}-\sign\Big(\frac{[2i+2]}{[2i][2]}\Big)e_{i}+e_{i+1}.$$
This shows that the matrix of multiplication by $e_1$ is tridiagonal with non-zero entries. Hence the argument of the preceding proof repeats, showing that $V_q$ is semi-simple for any root $q$ of odd order. 
\end{proof}

We have no proof for the following properties that we checked numerically for $\ell<100$ prime and $q^\ell=1$.
\begin{enumerate}
    \item $V_q$ is a number field. 
    \item The fields associated to $q=\exp(2i\pi k/\ell)$ and $q'=\exp(2i\pi k'/\ell)$ are isomorphic if and only if $kk'=\pm p'$ where $p'=\frac{p+1}{4}$ if $p=-1[4]$ and $p'=\frac{p-1}{4}$ if $p=1[4]$. We say that $q$ and $q'$ are conjugate.
    \item The ring linearly generated by $e_0,\ldots,e_{(\ell-3)/2}$ is equal to the ring of integers of $V_q$ except possibly for one pair of conjugate $\ell$-th roots. 
\end{enumerate}
We will describe all Frobenius algebras of level $5$ and $7$ in Sections \ref{ss: R-matrix level 5} and \ref{ss:level7}. 


\subsection{Degree $2$ of a CohFT and the $R_1$-matrix}\label{calculR}

\subsubsection{Consequences of the Givental-Teleman theorem}
Suppose we have a semi-simple CohFT $\omega_{g,n}:V^{\otimes n}\to H^*(\Mb_{g,n},\Q)$. We denote by $\eta$ its non-degenerate bilinear form and by $1$ its unit. We recall the formula $\omega_{0,3}(v_1,v_2,v_3)=\eta(v_1v_2,v_3)$.


Denote by $\sigma$ and $\tau$ the degree $0$ and $2$ terms of $\omega$. This notation is suggested by our examples where they correspond respectively to the signature and the Toledo invariant of the Hermitian modular functor. 

The celebrated Givental-Teleman classification theorem says that the CohFT $\omega_{g,n}$ can be reconstructed from the degree $0$ part $\sigma$ and a $R$-matrix $R\in \End(V)[[z]]$ that we write $R(z)=\id + zR_1+o(z)$. In this article, we will use it only to express $\tau$ in terms of $\sigma$ and $R_1$ so that we recall only the parts of the theorem necessary for our purposes. We refer to \cite{Pandharipande} for the full statement.

The $R$-matrix satisfies the so-called symplectic condition $R(z)R^*(-z)=\id$ where $A^*$ is the adjoint of $A$ with respect to the bilinear form $\eta$. This condition implies in degree 1 that $R_1$ satisfies $R_1^*=R_1$ or matricially, $\eta^{-1}R_1^T=R_1\eta^{-1}$. We also set $T(z)=z(1-R(z)1)=-z^2R_1(1)+o(z^3)\in V[[z]]$. 
Givental-Teleman's theorem state that $R$ and $T$ act on CohFTs in such a way that one has $$\omega=RT\sigma.$$

Compute first $T\sigma$ at first order, denoting by $p_1=\Mb_{g,n+1}\to\Mb_{g,n}$ the forgetful map and setting $\kappa_1=(p_1)_*\psi_{n+1}^2$, we get from Definition 6 of \cite{Pandharipande}:

$$(T\sigma)_{g,n}(v_1,\ldots,v_n)|_{\deg=2}=
-\omega_{g,n+1}(v_1,\ldots,v_n,R_1(1))\kappa_1.$$

Then from the definition of $R\sigma$ (Equation (2) in \cite{Pandharipande}) we get, writing the symmetric form $R_1\eta^{-1}=\sum r_{\mu\nu}\mu\otimes\nu\in V\otimes V$:

\begin{eqnarray*}
(R\sigma)_{g,n}(v_1,\ldots,v_n)|_{\deg=2}=\sum_{i=1}^n \sigma_{g,n}(v_1,\ldots,R_1(v_i),\ldots,v_n)\psi_i\\
-\sum_{\mu,\nu}r_{\mu,\nu}\sigma_{g-1,n+2}(v_1,\ldots,v_n,\mu,\nu)\delta_{\irr}\\
-\sum_*\sum_{\mu,\nu}r_{\mu,\nu}\sigma_{g_1,n_1+1}(v_{i_1},\ldots,v_{i_{n_1}},\mu)\sigma_{g_2,n_2+1}(v_{j_1},\ldots,v_{j_{n_2}},\nu)\delta_{g_1,I}
\end{eqnarray*}
In this last formula, the sum $\Sigma_*$ is over decompositions $g=g_1+g_2$ and partitions $I\amalg J=\{1,\ldots,n\}$ where $I=\{i_1,\ldots,i_{n_1}\}$ and $J=\{j_1,\ldots,j_{n_2}\}$.

Using the Frobenius algebra structure, we recast this formula in the case when $(g,n)=(0,4)$ or $(1,1)$ in the following proposition.
\begin{Proposition}\label{propr1}
Let $\omega=\sigma+\tau+(\deg>2)$ be a semi-simple CohFT and $R_1$ be its $R$-matrix at first order. We have
\begin{align*}
\tau_{0,4}&(v_1,\ldots,v_4)=\sum_{i=1}^4\eta(R_1v_i,\prod_{j\ne i}v_j)\psi_i-\eta(v_1v_2v_3v_4,R_1(1))\kappa_1\\
&-\eta(R_1(v_1v_2),v_3v_4)\delta_{12}-\eta(R_1(v_1v_3),v_2v_4)\delta_{13}
-\eta(R_1(v_1v_4),v_2v_3)\delta_{14}.\\
\tau_{1,1}&(v)=\eta(\Omega,R_1(v))\psi_1-\eta(\Omega,v R_1(1))\kappa_1-\tr (R_1 M_v)\delta_\irr.
\end{align*}
In this formula, $\Omega\in V$ is the value of the punctured torus: it equals $\Omega=\sum_{i=1}^n v_i^2$ for any orthonormal basis $v_1,\ldots,v_n$ of $V\otimes \C$. 
\end{Proposition}

We simplify further these formulas by observing that $H^2(\Mb_{0,4},\Q)$ and $H^2(\Mb_{1,1},\Q)$ are 1-dimensional. Hence, we can replace the classes with their integrals, using 
$$\int_{\Mb_{0,4}}\psi_i=\int_{\Mb_{0,4}}\kappa_1=\int_{\Mb_{0,4}}\delta_{ij}=1$$
and 
$$\int_{\Mb_{1,1}}\psi_1=\int_{\Mb_{1,1}}\kappa_1=\frac{1}{24}\text{ and }\int_{\Mb_{1,1}}\delta_{\irr}=\frac{1}{2}.$$

\subsubsection{A decomposition of the $R_1$-matrix}\label{decomposition}

Let $\S(V)=\{A\in \End(V), A^*=A\}$ be the space of rational endomorphisms of $V$, symmetric with respect to $\eta$. From the axioms of Frobenius algebras, the map $v\mapsto M_v$ embeds $V$ into $\S(V)$. 

We endow $\S(V)$ with the bilinear form $\langle A,B\rangle=\tr(AB)$: by semi-simplicity, its restriction to $V$ is non-degenerate, hence we have a decomposition $\S(V)=V\oplus V^\perp$ which allows to decompose any $R_1$-matrix in the form $$R_1=M_{r_1}+R_1',\quad r_1\in V,\quad R_1'\in V^\perp.$$

Plugging this decomposition into the formula of Proposition \ref{propr1}, we observe that the contribution of $r_1$ in $\tau_{0,4}$ cancels: knowing $\tau_{0,4}$ is equivalent to knowing $R_1'$. A standard way to do so is to decompose the matrix into the idempotent basis but it seems to be more efficient to use a fixed element $w$, that we will call the \emph{pivot}, and try to extract $R_1'$ from the endomorphism $A_w$ defined for all $u,v\in V$ by $$\tau_{0,4}(w,w,u,v)=\eta(A_w(u),v).$$

A computation using Proposition \ref{propr1} gives 
\begin{eqnarray}\label{double-commutateur}
\notag A_w(v)&=&2R'_1(w)wv+R'_1(w^2)v+w^2R'_1(v)\\
\notag &&-R'_1(1)w^2v-R'_1(w^2)v-2wR'_1(wv)\\
&=&[[R'_1,M_w],M_w](v)-[[R'_1,M_w],M_w](1)v
\end{eqnarray}
This shows that if $M_w$ is semi-simple, we can indeed extract $R_1'$ from $A_w$.

If we decompose $R_1$ in the formula expressing $\tau_{1,1}(v)$ we get from the equality $\tr(R_1M_v)=\tr(M_{r_1}M_v)=\tr_V(r_1v)$ the expression:
\begin{eqnarray*}
\tau_{1,1}(v)&=&\frac{1}{24}\tr_V(R_1(v)-R_1(1)v)-\frac{1}{2}\tr_V(r_1 v)\\
&=&\frac{1}{24}\tr_V(R_1'(v)-R_1'(1)v)-\frac{1}{2}\tr_V(r_1v).
\end{eqnarray*}
This last equation shows how to compute $r_1$ from $R_1'$ and $\tau_{1,1}$. 



\subsubsection{The computation of $R_1'$ for $\SO_3$-modular functors}

Let $q$ be a primitive root of unity of order $\ell=2r+1$ and $\V_q$ be the associated modular functor. We recall that its Frobenius algebra has basis $1=e_0,e_1,\ldots,e_{r-1}=w$. 
From the formulas of Section \ref{FrobSO3}, the pivot $w$ acts by
$$we_i= \sign\Big(\frac{[2i+2]}{[2]}\Big)e_{r-1-i} -\sign\Big(\frac{[2i+1]}{[2]}\Big) e_{r-i}$$
from which it follows that $e_1=w^2+1$. As $e_1$ has a simple spectrum, the same is true for $w$ and the strategy of the preceding section works for $w$. 

\begin{Remark}
Specialists in TQFT may notice that $w$ corresponds to the color $1$ in the basis of ``small colors", see \cite{BHMV}. It is then quite expected that it plays a prominent role. 
\end{Remark}
We can compute the dimension of the vector space $\V_q(S^2,2i,2j,2r-2,2r-2)$ by applying the axiom MF2 along a curve $\gamma$ which separates the colors $2i,2j$ from the colors $2r-2,2r-2$. Due to the constraints $(T)$, the color $2k$ of $\gamma$ can take only the values $0,2$, and cannot take the value $2$ if $i\ne j$. This gives 
$$\tau_{0,4}(w,w,e_{i},e_{j})=0\text{ if }i\ne j.$$

If $i=j$, denote by $f_0,f_1$ the basis of $\V_q(S^2,2i,2j,2r-2,2r-2)$ obtained by assigning the colors $0,2$ to $\gamma$. We compute: 
\begin{enumerate}
\item $||f_0||^2=[2r-1][2i+1]=-[2][2i+1]$
\item $||f_1||^2=[3]^{-1}\langle 2r-2,2r-2,2\rangle\langle 2i,2i,2\rangle=-\frac{[2i+2][2i+1]}{[2i][2][3]^2}$
\item $T_\gamma f_0=f_0,\quad  T_\gamma f_1=q^4 f_1$
\end{enumerate}
Let $\delta$ be a curve separating the colors $2r-2,2i$ from $2r-2,2i$. This time, the possible colors of $\delta$ in the decomposition are $2r-2-2i$ and $2r-2i$. Denoting by $g_0,g_1$ the corresponding vectors, we get 
\begin{enumerate}
\item $||g_0||^2=\langle 2r-2,2i,2r-2-2i\rangle^2[2r-2i-1]^{-1}\overset{\sign}{=}-[2i+2]$
\item $||g_1||^2=\langle 2r-2,2i,2r-2i\rangle^2[2r-2i+1]^{-1}\overset{\sign}{=}-[2i]$
\item $T_\delta g_0=q^{2(r-i-1)(r-i)}g_0,\quad  T_\delta g_1=q^{2(r-i)(r-i+1)} g_1$
\end{enumerate}
These formulas show that $\tau_{0,4}(w,w,e_i,e_i)=0$ if $[2i][2i+2]>0$ because the Hermitian form is definite. 

\begin{Lemma}\label{toledo-triangle}
Let $A,B,C\in \PU(1,1)$ be three elements satisfying for some $a,b,c>1$
$$A^a=B^b=C^c=ABC=1$$ and denote by $\theta_A,\theta_B,\theta_C\in (-\pi,\pi)$ the angles of $A,B,C$ acting on $\H^{1,1}$. Then, the Toledo invariant associated to this representation of the fundamental group of a sphere with three singular points of order $a,b,c$ is 
$$\tau=\epsilon-\frac{\theta_A+\theta_B+\theta_C}{2\pi},\quad \epsilon=\sign(\theta_A)=\sign(\theta_B)=\sign(\theta_C)$$
\end{Lemma}
\begin{proof}
We observe that the centers of $A,B,C$ in $\H^{1,1}$ form a triangle with angles $\frac{1}{2}\theta_A,\frac{1}{2}\theta_B,\frac{1}{2}\theta_C$. Hence these angles have the same sign and their sum satisfy $|\theta_A+\theta_B+\theta_C|\le 2\pi$. The result follows from the Gauss-Bonnet formula and the identification of the Toledo invariant with twice the area of the triangle divided by $2\pi$. 
\end{proof}
We observe that if a matrix $A$ is diagonal in an orthogonal basis $e_0,e_1$, such that $Ae_0=q_0 e_0, Ae_1=q_1e_1, \sign ||e_0||^2=\epsilon_0,\sign ||e_1||^2=\epsilon_1$, we have $$e^{i\theta_A}=\big(\frac{q_1}{q_0}\big)^{\epsilon_0}=\big(\frac{q_0}{q_1}\big)^{\epsilon_1}.$$

This gives in the case when $[2i][2i+2]<0$:
$$e^{i\theta_\alpha}=e^{i\theta_\beta}=q^{4(r-i)\sign([2i])}, e^{i\theta_\gamma}=q^{-4\sign([2][2i+1])}.$$

To sum up, the explicit formulas we have just written can be plugged into Lemma \ref{toledo-triangle} to obtain the Toledo invariants $\tau_{0,4}(w,w,e_i,e_i)$. In particular, they belong to $(-1,1)\cap \frac{1}{\ell}\Z$. 

This gives an explicit formula for the diagonal matrix $A_w$. Inverting Equation \eqref{double-commutateur} gives back $R_1'$. We observe that this equation is easily solved in an idempotent basis $v_1,\ldots,v_r\in V\otimes\C$. Let $\lambda_1,\ldots,\lambda_r\in \C$ be defined by $wv_i=\lambda_i v_i$. In this basis, $R_1'$ has vanishing diagonal: if $r_{ij}$ are the entries of $R_1'$, then the entries of $[[R_1',M_w],M_w]$ are $(\lambda_i-\lambda_j)^2r_{ij}$. It follows that the maximal denominator of $R_1'$ is $\ell \prod_{i\ne j}(\lambda_i-\lambda_j)^2=\ell \Delta_w^2$ where $\Delta_w$ is the discriminant of the minimal polynomial of $w$. This discriminant divides the discriminant $\Delta_V$ of $V$, provided that it is a number field.

\subsubsection{The computation of $r_1$ for $\SO_3$-modular functors}

As explained in the end of Section \ref{decomposition}, one can recover $r_1$ from the data of $R_1'$ and $\tau_{1,1}(e_i)$. Unfortunately, these Toledo invariants are harder to compute for at least two reasons: first the axioms of modular functors are not sufficient to compute it: we need an explicit formula for the image of $T_\gamma$ and $T_\delta$ where $\gamma,\delta$ are two simple curves on a punctured torus intersecting once. Secondly, the dimension of the representation $V_q(S,e_i)$ where $S$ is a punctured torus might be large: it is equal to $r-i$. Although $\Mod^\ell_{1,1}$ is again a triangle group (up to the elliptic involution), there is no simple formula for $\tau_{1,1}(e_i)$ as in Lemma \ref{toledo-triangle}. We need to adapt a formula due to Meyer (see Appendix \ref{Meyer}) to provide an effectively computable formula that we give now.

Suppose that we have already explicit formulas for $T_\gamma,T_\delta\in \U(p,q)$ where $\U(p,q)$ is the unitary group of $\V_q(S,e_i)$ where $S$ is a punctured torus. The following formulas hold in $\PU(p,q)$, yielding a representation of the triangle group $\Delta(2,3,\ell)$: $$T_\gamma^\ell=T_\delta^\ell=(T_\gamma T_\delta)^3=(T_\gamma T_\delta T_\gamma)^2=1.$$
In the Meyer formula of Appendix \ref{Meyer}, we obtain by putting $A=T_\gamma,B=T_\delta T_\gamma, C=(T_\gamma T_\delta T_\gamma)^{-1}$:

\begin{multline*}
\tau_{1,1}(e_i)=\frac{1}{2}\operatorname{Sign}\left[\frac{1}{i}(1-(T_\gamma T_\delta)^{-1})(1-T_\gamma)^{-1}(1-T_\gamma T_\delta T_\gamma)\right]\\
+\frac{1}{2}G(T_\gamma)+\frac{1}{2}G(T_\gamma T_\delta)-\frac{1}{2}G(T_\gamma T_\delta T_\gamma).
\end{multline*}
In this formula, $G(T)$ is a signed sum of arguments of the eigenvalues of $T$ for which we refer to the appendix. We also observe that this formula makes sense only if $T_\gamma$ has no fixed vectors: this will be the case as soon as $i>0$, a harmless assumption since the Toledo invariant vanishes when $i=0$ as the Hermitian structure is then unitary.

It remains to provide an explicit description of $T_\gamma$ and $T_\delta$. For that we will use the curve operators $C_\gamma$: this is a Hermitian operator associated to any simple curve $\gamma$ satisfying Kauffman rules. We refer to \cite{BHMV} or \cite{Bordeaux} for more detail. We will need only two properties for $C_\gamma$: the first one is that it has the same diagonalization basis as $T_\gamma$. 

Applying the axiom MF2 along $\gamma$ yields a decomposition of $\V$ indexed by $2i\in \Lambda$. The eigenvalue of $T_\gamma$ on this subspace is $r_{2i}=q^{2i(i+1)}$ and the eigenvalue of $C_\gamma$ is $c_{2i}=q^{4i+2}+1+q^{4i+2}$. 
We observe that the spectrum of $C_\gamma$ is simple so that for any $2\ell$-root of unity, there exists a polynomial $Q\in \Q(q)[X]$ such that $Q(c_{2i})=r_{2i}$ for all $i\in \{0,\ldots,r-1\}$. 
Hence $T_\gamma$ can be computed from $C_\gamma$ by the formula $T_\gamma=Q(C_\gamma)$.

Consider now a punctured torus $S$ represented in Figure \ref{torep}. Applying the axiom MF2 along $\gamma$ decomposes $\V(S,e_{i})$ into 1-dimensional spaces. Denote by $\psi_j$ the basis vector correponding to the color $2j$. The conditions $(T)$ yield $i\le 2j<2r-i$ giving $\dim \V(S,e_i)=r-i$.

\begin{figure}[htbp]
\centering
\includegraphics[width=3cm]{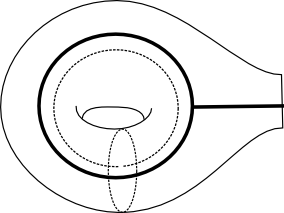}
\caption{A basis for the punctured torus}\label{torep}
\end{figure}

\begin{Proposition}
Setting $u_j=\frac{[i+j+1][j-i]}{[j][j+1]}$, the curve operator $C_\delta$ satisfies 
$$C_\delta \psi_j= \psi_{j+1}+(u_{2j+1}+u_{2j}-1)\psi_j+u_{2j}u_{2j-1}\psi_{j-1}.$$
\end{Proposition}

These complicated formulas yield an explicit algorithm for computing the $R$-matrix which we implemented in Sage. We will give explicit examples in the next section, but we observe (and can indeed prove) that the denominators of any entry of the R-matrix divide $2\cdot 3\cdot \ell \cdot\Delta_V^2$.

\subsection{The example of $q=\exp(2i\pi/5)$}\label{ss: R-matrix level 5} 

In this case, the colors $0,2$ correspond to elements $e_0,e_1$. The element $e_0$ is the unit and $e_1$ satisfies $e_1^2=-e_1-1$. Noting $e_1=t$, this gives $V=\Q[t]/(t^2+t+1)$. One has $\epsilon(1)=1,\epsilon(t)=0$ hence $\eta(1,1)=1$ and $\eta(t,t)=-1$. A simple computation gives $\alpha=\frac{1-t}{3}$ and $\Delta_V=-3$.




As explained in Section \ref{calculR}, to compute the matrix $R_1$, it is sufficient to compute $\tau_{1,1}$ and $\tau_{0,4}$. The first term vanishes because the modular functor is 1-dimensional in that case. It remains to consider the case of $\tau_{0,4}$. The only non trivial term is $\tau_{0,4}=\tau_{0,4}(t,t,t,t)$ which can be computed by Lemma \ref{toledo-triangle}. We find that $\rho_{0,4}:\Mod_{0,4}^5=\Delta(5,5,5)\to \PU(1,1)$ has Toledo invariant $-\frac{2}{5}$. Indeed, each generator acts by a rotation of angle $-\frac{2\pi}{5}$. We recognize here the uniformization of the orbifold $\Mb_{0,4}^5$, justifying the equality

$$\tau_{0,4}=\chi(\Mb_{0,4}^5)=-\frac{2}{5}.$$

Let us compute now the matrix $R_1$. It satisfies $R_1^*=R_1$ and from the fact that $\tau_{1,1}(e_0)=0$ we get from Proposition \ref{propr1} that $\tr R_1=0$. 
Hence we may write $R_1=\begin{pmatrix} a & -b \\ b & -a\end{pmatrix}$. 

Applying Proposition \ref{propr1} we get 
$$ -\frac{2}{5}=\tau_{0,4}=4\eta(R_1(t),t^3)-\eta(t^4,R_1(1))-3\eta(R_1(t^2),t^2)=-6a+3b.$$
Applying it again to compute $\tau_{1,1}(t)=0$ we obtain 
$$0=\frac{1}{24}\eta(\Omega,R_1(t)-tR_1(1))-\frac{1}{2}\tr(R_1 M_t)$$
which yields after computation $10a=23b$. We get finally
$$R_1=\frac{1}{270}\begin{pmatrix}23 & -10 \\ 10 & -23 \end{pmatrix}.$$

For further use, we write explicitly the Toledo invariants for $q=\exp(\frac{2i\pi}{5})$ as follows
$$\tau_{g,n}=a \tilde{\kappa}_1+b\sum_{i=1}^n \psi_i+c\delta_{\irr}+\sum_{g=g_1+g_2,n=n_1+n_2} d_{g_1,n_1}\delta_{g_1,n_1}$$
where

$$ a= - \frac{23} {270} \sigma_{g,n}-\frac{1}{27} \sigma_{g,n+1},\quad   b= - \frac{2}{15} \sigma_{g,n}, \quad c  =\frac{1}{90} \sigma_{g-1,n+1} \,\,(g\ge 1)$$
\begin{multline*}
d_{g_1,n_1}=  - \frac{23} {270} (\sigma_{g_1, n_1} \sigma_{g_2, n_2}+\sigma_{g_1, n_1+1} \sigma_{g_2, n_2+1}  )-\frac{1}{27} (\sigma_{g_1, n_1+1} \sigma_{g_2, n_2}+\\
\sigma_{g_1,n_1} \sigma_{g_2, n_2+1})
\end{multline*}

\subsection{The examples of level 7}\label{ss:level7}

\subsubsection{The Frobenius algebras}

{\bf Case 1: $q_1=\exp(2i\pi/7)$}

Denote by $e_0,e_1,e_2$ the standard basis of $V$ corresponding to the colors $0,1,2$ as explained in Section \ref{FrobSO3}. We find that $e_0$ is the unit and one has $e_1^2=e_1+e_2+e_0, e_1e_2=-e_1-e_2$. Noting $e_1=t$, this gives $e_2=t^2-t-1$ and $V=\Q[t]/(t^3-t-1)$ whose discriminant is $\Delta_V=-23$.

As $\epsilon_0=\epsilon_1=-\epsilon_2=1$, the co-unit $\epsilon:V\to \Q$ satisfies $\epsilon(1)=\epsilon(t^2)=1$ and $\epsilon(t)=0$. This counit can be written $\epsilon(x)=\tr_{V/\Q}(\alpha x)$ for $\alpha=\frac{1}{23}(9+3t-2t^2)$.

\noindent
{\bf Case 2: $q_2=\exp(4i\pi/7)$}

In this case, the same standard basis $e_0,e_1,e_2$ of the previous example behaves differently. The vector $e_0$ is still the unit but this time, $e_1^2=-e_0-e_1+e_2$ and $e_1e_2=-e_1-e_2$. Denoting $e_1=s$ we get $e_2=s^2+s+1$ and $V=\Q[s]/(s^3+2s^2+3s+1)$. We observe that this number field is isomorphic to the preceding one by putting $s(t+1)=-1$. 

This time $\epsilon_0=-\epsilon_1=\epsilon_2=1$, hence the co-unit is given by $\epsilon(1)=1,\epsilon(s)=0,\epsilon(s^2)=-1$. This gives $\alpha=\frac{1}{23}(19+9s+8s^2)$ which differs from the previous one.

\noindent
{\bf Case 3: $q_3=\exp(6i\pi/7)$}

Now $e_0,e_1,e_2$ satisfy $e_1^2=e_0+e_1+e_2$ and $e_1e_2=e_1+e_2$. Writing $t=e_1$ gives $e_2=t^2-t-1$ and $V=\Q[t]/(t^3-2t^2-t+1)$. This field is the subfield of $\Q(\zeta)$ where $\zeta^7=1$ fixed by the involution $\zeta\mapsto \zeta^{-1}$. 
We have $\epsilon_0=\epsilon_1=\epsilon_2=1$, giving $\epsilon(1)=1,\epsilon(t)=0,\epsilon(t^2)=1$. Finally $\alpha=\frac{1}{7}(3-t)$, recovering the example at the beginning of Section \ref{examples}.

\subsubsection{The $R_1$-matrix}

The Toledo invariants $\tau_{0,4}$ and $\tau_{1,1}$ can be computed by the methods of Section \ref{calculR} and are collected in the following table.

\vspace{0,5cm}
\begin{tabular}{|c|c|c|c|c|c|c|}
\hline
&$\sigma^{q_1}$& $\tau^{q_1}$ & $\sigma^{q_2}$ & $\tau^{q_2}$ & $\sigma^{q_3}$ & $\tau^{q_3}$\\ 
\hline
$\omega_{0,4}(e_1,e_1,e_1,e_1)$ & 1 & $\frac{2}{7}$ & 1 & $\frac{2}{7}$ & 3& 0 \\
\hline
$\omega_{0,4}(e_1,e_1,e_1,e_2)$ & 0 &$-\frac{4}{7}$ & -2 & 0 & 2& 0 \\
\hline
$\omega_{0,4}(e_1,e_1,e_2,e_2)$ & 0&$\frac{2}{7}$ & 0 & $-\frac{2}{7}$ & 2& 0 \\
\hline
$\omega_{0,4}(e_1,e_2,e_2,e_2)$ & -1 & 0 & 1 & 0 & 1& 0 \\
\hline
$\omega_{0,4}(e_2,e_2,e_2,e_2)$ & 2 & 0 & 0 & $\frac{4}{7}$ & 2& 0 \\
\hline
\hline
$\omega_{1,1}(e_0)$ & 3 & 0 & 3 & 0 & 3 & 0\\
\hline
$\omega_{1,1}(e_1)$ & 0 & $-\frac{1}{42}$ & -2 & 0 & 2 & 0\\
\hline
$\omega_{1,1}(e_2)$ & -1 & 0 & -1 & 0 & 1 & 0\\
\hline
\end{tabular}
\vspace{0,5cm}

It is known that the representation of $\SL_2(\Z)$ corresponding to the line $\omega_{1,1}(e_0)$ factors through $\mathrm{PSL}_2(\mathbb{F}_7)$. It corresponds to the automorphism group of the Klein quartic. The line corresponding to $\omega_{1,1}(e_1)$ contains the uniformization of the triangle group $(2,3,7)$. Indeed, its Toledo invariant $-\frac{1}{42}$ is equal to the Euler characteristic of a sphere with singularities of order $2,3$ and $7$.


Using these tables, Sage and the formulas of Section \ref{calculR}, we get the following formulas for the $R_1$-matrix: 
$$R^{q_1}_1=\frac{1}{22218}
\begin{pmatrix} 1373 & 1425 & -1635 \\
1425 & 59 & -1722 \\
1635 & 1722 & -1432
\end{pmatrix}$$
$$R^{q_2}_1=\frac{1}{22218}
\begin{pmatrix} -3615 & 1027 & 1973 \\
-1027 & 3719 & -36 \\
1973 & 36 & -104
\end{pmatrix}.$$
Again, we observe that in both cases the common denominator is $22218=2\cdot 3\cdot 7 \cdot 23^2=6\ell\Delta_V^2$ as expected.




\section{Complex hyperbolic structures on moduli spaces associated to Fibonacci TQFT} \label{s: uniformization}

\subsection{A criterion for uniformization}

Recall that $V$ denotes a Hermitian vector space of signature $(p,q)$ and $\H^{p,q}$ is the space of orthogonal decompositions $V=V^+\oplus V^-$ where the restriction of the Hermitian form to $V^+,V^-$ is respectively positive and negative. The aim of this section is to give a converse of Corollary \ref{cor:unif} in the case when $p=1$. It is a rather direct application of Siu's rigidity theorem:

\begin{Lemma}\label{l: criterion}
Let \(X\) be a compact K\"ahler complex orbifold of dimension \(q>1\). Assume that it admits a smooth finite orbifold covering, and that \[c_1(K_X )^q \neq 0.\] Let \( \rho : \pi_1 (X) \rightarrow \PU(1, q) \) be a morphism whose Toledo invariant satisfies 
\[ \tau (\rho ) = \frac{2}{q+1} c_1(K_X) \]
Assume furthermore that there exists a compact complex curve in \(X\) in restriction to which the Toledo invariant is positive. 
Then,  \(X\) admits a \(\mathbb H^{1,q}\)-structure whose holonomy is the representation \(\rho\).
\end{Lemma}

\begin{proof}
Since the Toledo invariant does not vanish, there exists a unique \(\rho\)-equivariant harmonic map \(f : \tilde{X} \rightarrow \mathbb H^{1,q} \). With the assumption that the top power of the Toledo invariant is not zero, there exists a point in \(\tilde{X} \) at which \(f\) is a submersion. Under these circumstances, Siu proved that \(f\) is either a holomorphic or anti-holomorphic map, see \cite{Siu_rigidity}. The assumption on the existence of a compact complex curve in restriction to which the Toledo invariant of \(\rho\) is positive forces \(f\) to be holomorphic. 

Let \( \mathcal F \rightarrow X\) be the vector bundle over \( X\) which is defined as the quotient of \( f^* T \mathbb H^{1,q}  \) by the action of the fundamental group of \(X\) given by \( \gamma (x, \xi) = (\gamma x, D_{x} \gamma (\xi) )\). Lemma \ref{l: pull-back tangent bundle} shows \( \tau (\rho ) = \frac{-2} {q+1} c_1 ( \mathcal F ) \) hence the main assumption of the lemma implies 

$$ c_1(K_X)+c_1(\mathcal{F})=0.$$

The map \(f\) induces a morphism \( f^* : \bigwedge ^q \mathcal{F}^*  \rightarrow K_X \). Denoting by \(D\) its zero divisor, we have \( \bigwedge ^q \mathcal F^* = D + K_{X}\), and so 
\[ -c_1 (\mathcal T ) = [D] + c_1 (K_X). \]  
Our assumptions give \([D]=0\). This implies \(D=0\) since \(D\) is an effective divisor and \(X\) is K\"ahler. 

This says that \(f\) is a \(\rho\)-equivariant local biholomorphism between \(\tilde{X}\) and \(\mathbb H^{1,q}\). The pull-back of the Hermitian metric on \(\tilde{X}\) is complete since \(X\) is compact, so \(f\) is a covering, and indeed a biholomorphism from \(\tilde{X}\) to \(\mathbb H^{1,q}\) since these spaces are connected and \(1\)-connected respectively. The conclusion follows. 

\end{proof}

\subsection{Deligne-Mostow / Hirzebruch's example}

\begin{Proposition}\label{p: uniformization_0_5}
The \(\SO(3)\)-quantum representation of level \(5\) associated to a surface of genus \(0\) with five (non-trivially colored) marked points is the holonomy of a \(\mathbb H^{1,2}\)-structure on \( \Mb_{0,5}^5\).
\end{Proposition}

\begin{proof}
Let \( \rho_{0,5}^5 : \pi _1 (\Mb_{0,5}^5)\rightarrow \text{PU} (1,2)\) be the \(\SO(3)\) quantum representation of level \(5\) with colors \(1\) at the five points and write $\tau^5_{0,5}=\sch_1(\rho_{0,5}^5)$.  The computations of Subsection \ref{ss: R-matrix level 5}  show that
\[\tau_{0,5} ^5 = \frac{13}{270} \tilde{\kappa}_1 + \frac{2}{15} \psi + \frac{13}{270} \delta  \]
where $\psi$ is the sum of the $\psi$ classes and $\delta$ is the sum of the boundary divisors. 
 On \(\Mb_{0,5}\), the classes $\psi_1,\ldots,\psi_5$ form a basis of $H^2(\Mb_{0,5}, \mathbb Q)$, see \cite[Theorem 2.2]{AC}. First, the class \(\kappa_1\) can be expressed as a sum of five boundary divisors,  and since there are ten boundary divisors in total, summing over all the symmetric expressions when permuting the marked points, we get 
 \[\kappa_1=\frac{1}{2} \delta.\]
 We also can express each class \(\psi_i\) as a sum of three boundary divisors, so we deduce similarly  
 \[\delta = \frac{2}{3} \psi .\]  
 The conclusion is that 
 \[\tau^5_{0,5} = \frac{2}{15} \psi.\]
Moreover, Lemma \ref{l: canonical bundle} and the aforementioned relations in \(H^2 (\Mb_{0,5}, \mathbb Q) \) show the equality \(c_1( K_{\Mb_{0,5}^5}) =\frac{1}{5} \psi \). Hence we have the correct proportionality
\[\tau^5_{0,5}= \frac{2}{3} c_1(K_{\overline{M}_{0,5}^5}) .\] 
 
 Let us now verify that \(\rho\) satisfies the other assumptions of Lemma \ref{l: criterion}. Since the product of two distinct \(\psi\) classes is equal to two, and the square of each equal to \(1\), see \cite{Zvonkine}, we deduce that \(\psi^2 = 45\). Hence  \[ c_1  (K_{\overline{M}_{0,5}^5}) ^2 =\frac{45}{25}=\frac{9}{5} > 0 .\]
 
 Moreover, denoting by \(\pi : \Mb_{0,3} \times \Mb_{0,4}\rightarrow \Mb_{0,5} \) the parametrization of a boundary divisor, and using the separated gluing axiom of the CohFT (see Subsection \ref{ss: separated gluing}), we find 
 \[\pi ^* \tau^5_{0,5}= \sigma^5_{0,2}\sigma^5_{0,3} \tau^5 _{0,4} =-\tau^5 _{0,4}.\] 
As the computations of Subsection \ref{ss: R-matrix level 5} give 
\[ \int_{\Mb_{0,4}} \tau^5_{0,4} = -\frac{2}{5} \]
we get that the integral of \(\tau^5 _{0,5}\) on any boundary divisor is positive. 

The result follows from Lemma \ref{l: criterion} and the fact that the orbifold \(\Mb_{0,5}^5\) has a smooth finite orbifold covering, see subsection \ref{ss: uniformizable}. 
\end{proof}


\subsection{Livne's example}

In this section we prove that the elliptic contraction \(\Mb_{1,2}^{\mathcal E}\) of \(\Mb_{1,2}^5\) carries a \(\mathbb H^{1,2}\)-structure whose holonomy is the quantum representation $\rho^5_{1,2}$ (see Proposition \ref{p:M12}). This complex hyperbolic structure has been found by Livne in his PhD dissertation, see \cite{Livne}.


For the next statement, notice that the fundamental group of \(\Mb _{1,2}^{\mathcal E}\)  is isomorphic to the one of \(  \Mb_{1,2}^5\), so we can think of the quantum Fibonacci representation as being defined on \(\pi_1 (\Mb _{1,2}^{\E})\). 

\begin{Proposition}\label{p:M12}
The orbifold \(\Mb_{1,2}^{\E}\) admits a \(\mathbb H^{1,2}\)-structure whose holonomy is the \(\SO(3)\)-quantum representation of level \(5\), genus one and two marked points with non-trivial colors. 
\end{Proposition}

\begin{proof} 
In genus one, there are special relations in the second cohomology group of Deligne-Mumford compactification, see \cite[Theorem 2.2]{AC}. Specifically, in \(\Mb_{1,2}\), we have 
\[\tilde{\kappa}_1= -\delta_{1,\emptyset}, \ \ \psi_i = \frac{\delta_{\rm irr}}{12}+ \delta_{1,\emptyset} .  \]
These relations, together with the computations of Subsection \ref{ss: R-matrix level 5},  show that
\[ \tau_{1,2}^5 = \frac{2}{5}\delta_{1,\emptyset} +\frac{1}{30}\delta_{\rm irr} .\]

Denote by \(c : \Mb_{1,2}^5 \rightarrow \Mb_{1,2}^{\E}\) the blow-down of the elliptic tail divisor. Lemma \ref{l: canonical bundle elliptic contraction}
and the aforementioned relations show the expected proportionality holds:
\[ \tau_{1,2}^5 = \frac{2}{3} \ c^* (c_1 (K_{\Mb_{1,2}^{\E}})).\]

To end the proof, let us compute 
\[ c^* c_1(  K_{\Mb_{1,2}^{\E}} )^2 = \left(\frac{3}{5} \delta_{1,\emptyset} + \frac{1}{20}\delta_{\rm irr}\right)^2 = \frac{3}{200} \]
since \(\delta_{1,\emptyset}^2 = -\frac{1}{24}\), \(\delta_{1,\emptyset}\cdot \delta_{\rm irr} = \frac{1}{2} \) and \(\delta_{\rm irr}^2 = 0\).

It also happens that the restriction of \(\tau_{1,2}^5\) to \( \delta_{\rm irr}\) is positive. Indeed, if \( \pi : \Mb_{0,4} \rightarrow \delta_{\rm irr}\) is the (degree two) parametrization, we have by the non-separating CohFT axiom (only the color \(1\) at the node contributes) 
\[ \pi ^* \tau_{1,2}^5 = -\tau_{0,4}^5 \]
and so
\[ \int_{\delta_{\rm irr} } \tau_{1,2}^5 = \frac{1}{5}>0.\]
The proof follows from the fact that \(c^*\) is injective on the second cohomology group.
\end{proof}

\subsection{Complex hyperbolic structure on \(\Mb_{1,3}^{\E}\)}

The goal of this subsection is to prove that the elliptic contraction \(\Mb_{1,3}^{\E}\) of \(\Mb_{1,3}^5\) has a \(\mathbb H^{1,3}\)-structure whose holonomy is the conjugate of the \(\SO(3)\) quantum representation of level \(5\) and three points non-trivially colored. 

Recall that the orbifold fundamental group of \(\Mb_{1,3}^{\E} \) is isomorphic to the one of \(\Mb_{1,3}^5\) so we can view the representation \(\rho_{1,3}^5\) as a representation \(\rho ^5_{1,3}: \pi _1 (\Mb_{1,3}^{\E}) \rightarrow \text{PU} (3,1)\). 

\begin{Proposition}\label{p:M13} 
The conjugate of the \(\SO(3)\)-quantum representation \(\rho ^5_{1,3}\) of level \(5\) with the three marked points with non-trivial colors, is the holonomy of a \(\mathbb H^{1,3}\)--structure on \(\Mb_{1,3}^{\E}\). 
\end{Proposition}

\begin{proof}
To simplify notation we set \(X= \Mb_{1,3}^{\E}\). We will make use of the relations in the second cohomology group of Deligne-Mumford compactification in genus \(1\), see \cite[Theorem 2.2]{AC}, which take the following form in the case of three marked points
\begin{equation} \label{eq: relations 13} \tilde{\kappa}_1 = -\delta_{1,\emptyset} - \sum_i \delta_{1,\{i\}} \text{ and } \psi = \frac{1}{4}\delta_{irr}+ 3\delta_{1,\emptyset} + 2 \sum_i \delta_{1,\{i\}} .  \end{equation} 
Denote by \(c : \overline{M}_{1,3}^5\rightarrow X\) the blow-down map. 
Using Lemma \ref{l: canonical bundle elliptic contraction} and the relations \eqref{eq: relations 13}, we find after some computations
\[ c^* c_1 (K_X) =  \frac{2}{15} \delta_{irr} + \frac{8}{5}\delta_{1,\emptyset} +\frac{4}{5} \sum_i \delta_{1,\{i\}} .\]
Formulae of subsection \ref{ss: R-matrix level 5} show that we have the right proportionality for the Toledo invariant of the conjugate of \(\rho_{1,3}^5\):
\[ \tau (\overline{\rho_{1,3}^5}) = -\tau(\rho_{1,3}^5) = \frac{2}{3+1} c ^* c_1 (K_X) . \]

At this point, we shall not use Lemma \ref{l: criterion} as such, but rather take a detour which  circumvents the painful computation of \( K_X^3\). First of all, observe that the pull-back  of the conjugate of the representation \(\rho_{1,3}^5\) to the moduli space \(\overline{M}_{0,5}^5\) parametrizing the boundary divisor \(\delta_{irr}\) is the projectivization of the direct sum of a rank one representation with a positive negative hermitian form, and the representation \(\rho_{0,5}^5\). In particular, the \(\overline{\rho_{1,3}^5}\)-equivariant pluriharmonic map \(f:\widetilde{X} \rightarrow \mathbb H^{1,3} _{\mathbb C} \) induces a  biholomorphism between any component of the lift of \(\delta_{irr}\) in \( \widetilde{X}\) and a totally geodesic complex subspace of \(\mathbb H_{\mathbb C}^{1,3}\). Moreover, the image of \(f\) is not globally  contained in such a subspace, since otherwise the representation \(\rho_{1,3}^5\) would be reducible, which is not the case by a result of Roberts \cite{Roberts}. A consequence of this is that the pluri-harmonic map \(f\) has real rank at least \(5\) somewhere. 

The reinforcement of Siu's rigidity theorem obtained by Carlson-Toledo \cite{CT} shows that \(f\) is holomorphic everywhere. Notice that at some point the differential of \(f\) is not zero, and we can follow word by word the last two paragraphs of the proof of Lemma \ref{l: criterion} to deduce that indeed \(f\) is a \(\overline{\rho_{1,3}^5}\)-equivariant biholomorphism between \(\tilde{X}\) and \(\mathbb H_{\mathbb C}^{1,3}\). 
\end{proof}



\subsection{Complex hyperbolic structure on \(\Mb_{2,1}^{\E}\)}

The goal of this subsection is to prove that the elliptic contraction  \(\Mb_{2,1}^{\E}\) has a complex hyperbolic structure whose holonomy is the \(\text{SO}(3)\)-quantum representation of level \(5\) and the color of the marked points equal to \(1\) (Proposition \ref{p:M21}). Recall as before that the orbifold fundamental group of \(\Mb_{2,1}^{\E}\) is isomorphic to the one of \(\Mb_{1,2}^5\).  


\begin{Proposition} \label{p:M21}
The orbifold \(\Mb_{2,1}^{\E}\) has a \(\H^{1,4}\)-structure whose holonomy is the \(\SO(3)\)-quantum representation \(\rho ^5_{2,1}\) of level \(5\) with the marked point colored by \(1\) (in particular, its image is an arithmetic lattice in \(\PU (1,4)\)).
\end{Proposition} 

\begin{proof}
The proof is analogous to the one in the case of \(\Mb_{1,3}\). The only thing which has to be established is the identity 
\begin{equation}\label{eq: relation canonical toledo} \tau_{2,1}^5 = \frac{2}{5}  c_1 (c ^* K_{\Mb_{2,1}^{\E}}) \end{equation}
where \(c : \Mb_{2,1}^5 \rightarrow \Mb_{2,1}^{\E}\) denotes the contraction.  Our formulae for the Toledo invariants of the Fibonacci representations (see section \ref{ss: R-matrix level 5}) show that 
\[ \tau_{2,1}^5 = \frac{2}{5} \psi + \frac{1}{25} \delta_{irr}+ \frac{12}{25} \delta_{1,\emptyset}. \]
We use here the relation in \(H^2 (\Mb_{2,1}^5)\) (see \cite{AC}) \begin{equation}\label{eq: relation21} \tilde{\kappa_1}= \frac{1}{5} \delta_{irr} + \frac{7}{5} \delta_{1,\emptyset} \end{equation}
(since there is only one marked point the divisor \(\delta_0\) in \cite{AC} vanishes).
Lemma \ref{l: canonical bundle elliptic contraction} and relation \eqref{eq: relation21} yield 
\[ c_1 (c^* K_{\Mb_{2,1}^{\E}}) = \psi + \frac{1}{10} \delta_{irr} + \frac{6}{5} \delta_{1,\emptyset}. \]
So \eqref{eq: relation canonical toledo} holds and the proof follows the same route as the one of Proposition \ref{p:M13}. 
\end{proof}

\subsection{Solution to Siu's problem: proof of Corollary \ref{c: Siu}} Lemma \ref{l: forgetful map elliptic contraction} and its proof shows that the forgetful map \( \Mb_{1,3}^{\E}\rightarrow \Mb_{1,2}^{\E}\) is an orbifold map, that can be lifted to a surjective holomorphic map \( X_{1,3} \rightarrow X_{1,2}\) between smooth finite connected orbifold coverings \(X_{1,3}\rightarrow \Mb_{1,3}^\E\) and \(X_{1,2}\rightarrow \Mb_{1,2}^\E\). Both \(X_{1,2}\) and \(X_{1,3}\) are complex hyperbolic compact manifolds as orbifold finite coverings of compact complex hyperbolic orbifolds.  So the result follows.

\appendix

\section{Meyer formula for the Toledo invariant}\label{Meyer}

\subsection{Definition of the Meyer cocycle}
Let $(V,h)$ be a Hermitian vector space of signature $(p,q)$: we denote as usual by $\U(p,q)$ its isometry group. 

Let $S$ be a surface and $\rho:\pi_1(S)\to \U(p,q)$ a homomorphism: the twisted homology group $H_1(S,V)$ is endowed with a skew-Hermitian form, composed of the intersection product together with the Hermitian form. We may write it $ih_S$ and are interested in this appendix in the signature of $h_S$.

Consider the case of a pair of pants $P$ which retracts on a graph $\Theta$. We define a homomorphism $\pi_1(P)\to \U(p,q)$ by sending the three edges of $\Theta$ respectively to $1,A,AB$ where $A,B\in \U(p,q)$. It is equivalent to send them to $A^{-1},1,B$ respectively. Then we set $\mu(A,B)=\sign(h_P)$. Standard arguments show that it is a cocycle, that is an element of $H^2(\U(p,q),\Z)$, see for instance \cite{tu}.

Explicitly, it is supported on the space $K=\{(u,v)\in V^2, (A^{-1}-1)u+(B-1)v=0\}$ with a form given by the following formula, where $(u,v),(u',v')$ are in $K$: 

$$ih_P((u,v),(u',v'))=h(u+v,(1-B)v').$$

Consider the case of $\U(1)=\U(1,0)$. One writes $A=e^{i\alpha},B=e^{i\beta}$. If $A\ne 1$ ou $B\ne 1$, the kernel $K$ is generated by $\kappa=(1-B,A^{-1}-1)$ and we compute
$$ih_P(\kappa,\kappa)=h(A^{-1}-B,(1-B)(A^{-1}-1))=8i\sin(\frac{\alpha+\beta}{2})\sin(\frac{\alpha}{2})\sin(\frac{\beta}{2}).$$

One deduces that in that case, $\mu(e^{i\alpha},e^{i\beta})=\sign(\sin(\frac{\alpha+\beta}{2})\sin(\frac{\alpha}{2})\sin(\frac{\beta}{2}))$.
If $A=B=1$ then $K=V^2$ but $h_P=0$ which gives $\mu(1,1)=0$ and agrees with the preceding formula. If we had chosen $\U(0,1)$ instead of $\U(1,0)$, we would have the opposite result.

Let us observe now what happens when we restrict this cocycle to the center of $\U(p,q)$. The representation $V$ becomes a direct sum and the contributions of the summand add with a sign, giving $\mu(e^{i\alpha}\id,e^{i\beta}\id)=(p-q)\mu(e^{i\alpha},e^{i\beta})$. 

\subsection{A relation with the Toledo class}

We would like to relate this cocycle to two well-known cocycles on $\U(p,q)$: the pull-back by the projection $\U(p,q)\to \PU(p,q)$ of the Toledo invariant $\tau\in H^2(\PU(p,q),\Q)$ and the pull-back by the determinant $\det:\U(p,q)\to \U(1)$ of the fundamental class $c\in H^2(\U(1),\Z)$. As $\mu$ is measurable and $H^2(\U(p,q),\Q)=\Hom(\pi_1(\pi_1(\U(p,q)),\Q)=\Q^2$ (see Section \ref{obstruction}), there exists $x,y\in \Q$ such that $$\mu=x\tau+yc\in H^2(\U(p,q),\Q).$$

As the composition $\U(1)\to \U(p,q)\to \U(1)$ is the map $z\mapsto z^{p+q}$, and the composition $\U(1)\to \U(p,q)\to \PU(p,q)$ is trivial, by pulling back the above equation to the center one finds $y=\frac{p-q}{p+q}$. 
 
Let us give an explicit formula for $\tau$. We recall that $\pi_1(\PU(p,q))\simeq \Z^2/(p,q)\Z$ and we define $\phi:\pi_1(\PU(p,q))\to\Q$ by the formula $\phi(x,y)=\frac{2}{p+q}(qx-py)$. 

Let $\Phi:\widetilde{\PU}(p,q)\to \R$ be the unique homogeneous quasi-morphism (continuous) verifying $\Phi(zg)=\phi(z)+\Phi(g)$ for $z\in \pi_1(\PU(p,q))$ and $\tau(A,B)=\Phi(\tilde{A}\tilde{B})-\Phi(\tilde{A})-\Phi(\tilde{B})$. 

We can give an explicit formula for $\Phi$ as in \cite{FP}. Suppose that there is an orthogonal basis $e_1,\ldots,e_n$ of $V$ with  
$$A e_j=e^{i\alpha_j}e_j,\quad h(e_j,e_j)=\epsilon_j\in\{\pm 1\}.$$
We define $\tilde{A}$ to be the path $A_t(e_j)=e^{it\alpha_j}e_j$ and obtain 
$$\Phi(\tilde{A})=\sum_{j=1}^n\alpha_j\epsilon_j-\frac{p-q}{p+q}\sum_j\alpha_j.$$

Restricted to diagonal matrices, $\Phi$ is a morphism, hence $\tau(A,B)=0$ if $A$ and $B$ are both diagonal. 

The cocycle $c$ is represented by the cyclic ordering. Precisely, $c(\alpha,\beta)=\ord(1,\alpha,\alpha\beta)=\ord(\alpha^{-1},1,\beta)=\sign(\sin(\alpha)+\sin(\beta)-\sin(\alpha+\beta))$. 

As $\sin(\alpha)+\sin(\beta)-\sin(\alpha+\beta)=4\sin(\frac{\alpha+\beta}{2})\sin(\frac{\alpha}{2})\sin(\frac{\beta}{2})$, we deduce $$c(\alpha,\beta)=\mu(e^{i\alpha},e^{i\beta}).$$

One need to establish the following for all $A,B\in \U(p,q)$:
$$\mu(A,B)=x \tau(A,B)+ \frac{p-q}{p+q} c(\det A,\det B)+dF(A,B)$$
for some function $F:\U(p,q)\to \R$.

Let us analyse this equation in restriction to diagonal matrices: 

$$\sum_{j=1}^{p+q} \epsilon_j c(\alpha_j,\beta_j)=\frac{p-q}{p+q}c(\sum \alpha_j,\sum \beta_j)+dF(A,B)$$

We observe that there exists indeed a map $f:\R/2\pi\Z\to \R$ such that
$$f(\alpha+\beta)-f(\alpha)-f(\beta)=c(\alpha,\beta).$$
We simply set $f(\alpha)=1-\frac{\alpha}{\pi}$ where $\alpha\in ]0,2\pi[$ and $f(0)=0$. 

To see it geometrically, one observe that $\pi f(\alpha)$ is the oriented area of a hyperbolic triangle with vertices $0,1,\alpha$. The Gauss-Bonnet formula gives an area equal to $\pi-\alpha$. This function is not continuous but has a nice Fourier expansion given by 
$$f(\alpha)=\frac{2}{\pi}\sum_{n>0}\frac{\sin(n\alpha)}{n}.$$

We are led to define, for $A$ diagonal with coefficients $e^{i\alpha_j}$ and signature $\epsilon_j$:
$$F(A)=\sum_{j=1}^{p+q}\epsilon_j f(\alpha_j) -\frac{p-q}{p+q}f(\sum_{j=1}^{p+q} \alpha_j)=G(A)-\frac{p-q}{p+q}f(\det A).$$

One can give an invariant formula for $G(A)$ in the spirit of the $G$-signature theorem. Suppose that $A$ has finite order or more generally, that its orbit is relatively compact. One can then find a decomposition $V=V^+\oplus V^-$ invariant by $A$ (take the barycenter of the orbit in the symmetric space). One defines then $\str(A)=\tr A|_{V^+}-\tr A_{V^-}$: it is independent of the decomposition. 
If $A$ is diagonal as above, one has $\str(A)=\sum_j \epsilon_j e^{i\alpha_j}$
We have then $$\sum_{n\in \Z^*}\frac{1}{n\pi}\str(A^n)=\sum_j i\epsilon_j f(\alpha_j)$$
hence the nice formula
$$G(A)=\sum_{n\in \Z^*}\frac{1}{in\pi}\str(A^n).$$

It remains to find the coefficient $x$. To this aim, we take a Fuchsian representation $\pi_1(S)\to \SU(1,1)$ that we send to $\SU(p,q)$ in the obvious way. 
We found that its Toledo invariant is $2g-2$. Compare with the Meyer cocyle: we decompose $V=E\oplus F$ where $E$ has signature $(1,1)$ and carries the action of $\pi_1(S)$. We have $H_1(S,E\oplus F)=H_1(S,E)\oplus H_1(S,F)$ and $H_1(S,F)=H_1(S,\C)\otimes F$. As the signature of $H_1(S,\C)$ vanishes, one sees that the factor containing $F$ does not contribute. For what concerns the factor $H_1(S,E)$, one find a positive definite Hermitian space of dimension $4g-4$, which gives $\mu=2\tau$, hence $x=2$. 

We sum up the formula that we obtained:
$$\mu(A,B)=2\tau(A,B)+dG(A,B)$$

\subsection{Application to triangle groups}
For $a,b,c$ three positive integers, we set $\Delta(a,b,c)=\langle A,B,C| ABC=A^a=B^b=C^c=1\rangle$ and consider a representation $\rho:\Delta(a,b,c)\to \PU(p,q)$. As $\Delta(a,b,c)$ is the orbifold fundamental group of a sphere $S$ with three singular points of order $a,b,c$ one can define $\int_S\rho^* \tau\in \Q$ and the aim of this section is to give an explicit formula for this rational number using the Meyer cocycle.

We observe that the virtual fundamental class of $\Delta(a,b,c)$ is given by  $[ABC]-\frac{1}{a}[A^a]-\frac{1}{b}[B^b]-\frac{1}{c}[C^c]\in H_2(\Delta(a,b,c),\Q)$, see Appendix \ref{Japonais}. We lift $A,B,C$ to $\tilde{A},\tilde{B},\tilde{C}\in \widetilde{\PU}(p,q)$ and get
$$\tau(\rho)=\Phi(\tilde A\tilde B\tilde C)-\frac{1}{a}\Phi(\tilde A^a)-\frac{1}{b}\Phi(\tilde B^b)-\frac{1}{c}\Phi(\tilde C^c).$$

Using homogeneity and setting $\tilde C=(\tilde A\tilde B)^{-1}$, we get

$$\tau(\rho)=\Phi(\tilde A\tilde B)-\Phi(\tilde A)-\Phi(\tilde B)=\tau(A,B).$$ 

Hence one can use Meyer formula which gives
$$\tau(A,B)=\frac{1}{2}\mu(A,B)-\frac{1}{2}dG(A,B)=\frac{1}{2}\left(\mu(A,B)+G(A)+G(B)+G(C)\right).$$

To analyse further $\mu(A,B)$ let us suppose that $A$ has no fixed point, so that $A^{-1}-\id$ is invers-tible. It allows to identify $K$ with $V$ by setting $u=(A^{-1}-\id)^{-1}(\id-B)v$. Hence 
\begin{eqnarray*}
ih_P(v,v')&=&h((A^{-1}-\id)^{-1}(A^{-1}-B)v,(\id-B)v')\\
&=&h((B^{-1}-\id)(A^{-1}-1)^{-1}(B-A^{-1})v,v').
\end{eqnarray*}
To sum up, $\mu(A,B)$ is the signature of the Hermitian matrix
\begin{eqnarray*}
H&=&\frac{1}{i}(B^{-1}-\id)(A^{-1}-\id)^{-1}(B-A^{-1})\\
&=&\frac{1}{i}(\id-B^{-1})(\id -A)^{-1}(\id-C^{-1})\\
&=&\frac{1}{i}(B-\id)(CB-\id)^{-1}(C-\id)
\end{eqnarray*}

It is a nice exercise to show that this matrix is indeed Hermitian: one way is to use Cayley parametrization: setting $B=(iY-1)(iY+1)^{-1},C=(iZ-1)(iZ+1)^{-1}$ with $Y$ and $Z$ Hermitian, we find $H=2(Y+Z)^{-1}$ which is again Hermitian.

\section{Toledo invariants from mapping class group presentations}\label{Japonais}

\subsection{Equivariant Hopf formula}

Let $S_{g,n}$ denote a closed oriented surface of genus $g$ with $n$ marked points. We fix a level $\ell\ge 1$ and recall that we have set $\Mod^\ell_{g,n}$ to be the quotient of $\Mod(S_{g,n})$ by the subgroup generated by $\ell$-th powers of all Dehn twists. Given colors $\lambda=(\lambda_1,\ldots,\lambda_n)$, we denote by $\rho_{g,n}^\lambda:\Mod_{g,n}^\ell\to \PU(\V(S,\lambda))$ the quantum representation. 

By construction, the invariant $\tau_{g,n}\in H^2(\Mb_{g,n},\Q)$ can be computed from $\rho_{g,n}^*\sch_1\in H^2(\Mod_{g,n}^\ell,\Q)$. The purpose of this section is to do this computation in the case $\ell=5$ starting from a presentation of the level $\ell$ mapping class group. We did it to double check our formulas for the Toledo invariants: in particular it is independent on Givental-Teleman classification and could shed further light on the properties of these invariants. 



Let $\Gamma$ be the free group generated by Dehn twists along (isotopy classes of) simple curves $\gamma\subset S\setminus P$. We denote by $T_\gamma$ both the formal and the actual Dehn twist. From the presentation $0\to R\to \Gamma \to \Mod^\ell_{g,n}\to 0$ and the Leray-Serre spectral sequence, we get the exact sequence
\begin{equation}\label{stallings}
0\to H_2(\Mod^\ell_{g,n},\Z)\to R/[\Gamma,R]\to H_1(\Gamma,\Z)\to H_1(\Mod^\ell_{g,n},\Z)\to 0.
\end{equation}

As $\Mod^\ell_{g,n}$ is generated by elements of order $\ell$, we get after tensoring by $\Q$ the exact sequence: 
\begin{equation}\label{stallingsrat}
0\to H_2(\Mod^\ell_{g,n},\Q)\to R/[\Gamma,R]\otimes \Q\overset{\ab}{\to} H_1(\Gamma,\Q)\to 0.
\end{equation}

We observe now that there is a natural action of the usual mapping class group $\Mod_{g,n}$ on $\Gamma$ given by $f.T_\gamma=T_{f(\gamma)}$. This also gives an action of $\Mod_{g,n}$ on $R$ and on $R/[\Gamma,R]$. As $f.T_\gamma=fT_\gamma f^{-1}$ in $\Mod_{g,n}$, the action on $H_2(\Mod^\ell_{g,n},\Q)$ is by conjugation, hence trivial. Using the fact that $H_1(\Mod_{g,n},\Q)=0$, we get the same sequence for co-invariants:

\begin{equation}\label{stallingscoinv}
0\to H_2(\Mod^\ell_{g,n},\Q)\to (R/[\Gamma,R]\otimes \Q)_{\Mod_{g,n}}\to H_1(\Gamma,\Q)_{\Mod_{g,n}}\to 0.
\end{equation}

The space $H_1(\Gamma,\Q)_{\Mod_{g,n}}$ is the $\Q$-vector space generated by orbits of simple closed curves. This finite set can be written as the disjoint union of the boundary curves $\{\gamma_1,\ldots,\gamma_n\}$ and a set of $\{\gamma_i\}_{i\in I}$, where $I$ parametrizes boundary divisors in $\Mb_{g,n}$. 

To go further, we need to recall a generating set of the subgroup $R$ of relations. We take it from \cite{Luo}.
\begin{Proposition}
The group $R\subset \Gamma$ of relations defining $\Mod^\ell_{g,n}$ is generated by the following elements: 
\begin{enumerate}
\item Disjointness. If $\alpha\cap \beta=\emptyset$, $D_{\alpha,\beta}=T_\alpha T_\beta (T_\beta T_\alpha)^{-1}$. 
\item Braiding. If $\alpha\cap\beta=\{pt\}$, $B_{\alpha,\beta}=T_\alpha T_\beta T_\alpha (T_\beta T_\alpha T_\beta)^{-1}$. 
\item 2-chain.  If $\alpha\cap\beta=\{pt\}$, $C_{\alpha,\beta}=(T_\alpha T_\beta)^{6}T_\gamma^{-1}$ where $\gamma$ is the boundary of a tubular neighborhood of $\alpha\cup \beta$. 
\item Lantern. If $\Sigma\subset S$ is a sphere with boundary components $\alpha,\beta,\gamma,\delta$, $L_\Sigma=T_\zeta T_\eta T_\theta (T_\alpha T_\beta T_\gamma T_\delta)^{-1}$ where $\zeta,\eta,\theta$ are as in Figure \ref{fig:lantern}. 
\item Boundary. If $\gamma_j$ is the $j$-th boundary curve, $\partial_j=T_{\gamma_j}$. 
\item $\ell$-th powers. If $i\in I$ is represented by $\gamma_i$, we set $R_i=T_{\gamma_i}^\ell$. 
\end{enumerate}
\end{Proposition}
\begin{figure}
    \centering
    \def\svgwidth{5cm}
    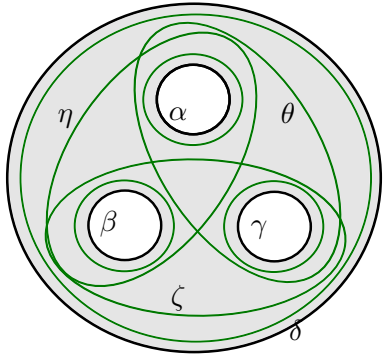
    \caption{Lantern relation}
    \label{fig:lantern}
\end{figure}

It remains to give an expression of the classes $\psi_j,\lambda_1,\delta_i$ as linear forms on $R$. These formulas are spread out in the literature (mainly \cite{FM,Endo_Nagami,MR}) and can be easily guessed, however we did not find it easy to justify them rigorously. As it would take too much space in this appendix and we used them only for double checking, we skip this justification here and give here the result. The table shows the value of each class evaluated on each type of relation. 
$$
\begin{tabular}{|c|c|c|c|c|c|c|}
\hline
& $D$ & $B$ & $C$ & $L$ & $\partial_j$ & $R_i$ \\
\hline
$\psi_k$ & $0$ & $0$ & $0$ & $0$ & $-\delta_{jk}$ & $0$ \\
\hline
$\lambda_1$ & $0$ & $0$ & $-1$ & $0$ & $0$ &$0$\\
\hline
$\delta_k$ & $0$ & $0$ & $0$ & $0$ & $0$ & $\ell \delta_{ik}$\\
\hline
\end{tabular}$$

\subsection{Computation of the Toledo invariant}

Let us consider the case of the $\SO_3$ modular functor given in Section \ref{examples} with $A=e^{i\pi/5}$ that we denote simply by $\V$. As the set of colors is $\Lambda=\{0,2\}$ and because coloring a boundary component with $0$ amounts to fill it with a disc, we can suppose that all marked points are colored with $2$.

We fix now a surface $S=S_{g,n}$ with marked points (and an immaterial Lagrangian). We denote by $(p_{g,n},q_{g,n})$ the signature of $\V(S_{g,n})$ and sometimes abbreviate in $(p,q)$. We also set $d_{g,n}=p_{g,n}+q_{g,n}$ and $\sigma_{g,n}=p_{g,n}-q_{g,n}$. 

Given a simple curve $\gamma\subset S_{g,n}$, we set $\rho(T_\gamma)=\V(T_\gamma,0)\in \U(\V(S))$. By the axiom MF2, there is a decomposition $\V(S)=\V(S)_0\oplus \V(S)_2$. This decomposition is also a decomposition of eigenspaces for $\rho(T_\gamma)$. It acts by multiplication by $1$ and $q^{-1}$ on each factor, where $q=A^2=e^{2i\pi/5}$. 

It is natural to lift $\rho(T_\gamma)$ to the path $\rho^t(T_\gamma)\in \PU (\V(S))$ of transformations acting on each factor by $1$ and $\exp(-2i\pi t/5)$ respectively. This gives an equivariant lift $\rho^t:\Gamma\to \widetilde{\PU}(\V(S))$ which, by restriction to $R$ gives the Toledo class we are looking for. We denote by $\tau:R\to\Q$ the map defined by $\tau(r)=\phi(\rho^t(r))$ and compute in the following sections the image of all generators in $R$. 

\subsubsection{The easy relations: $D,B,\partial_j,R_i$}
It is not hard to show that $\tau(D)=\tau(B)=0$. We skip it to save space. 

Another easy computation is $\tau(\partial_j)$.
Let $\gamma_j$ be a simple curve surrounding the $j$-th boundary point. The matrix $\rho^t(T_\gamma)$ acts on $\V(S)$ by $q^{-t}\id$. This is a trivial path in $\PU(\V(S)$ giving $\rho^t(\partial_x))=1$ and $\tau(\partial_x)=0$.

We choose a curve $\gamma$ in the class $i\in I$ and wish to compute $\rho^t(R_i)$. We compute that $\rho^t(T_\gamma)^5$ acts on $\V(S)_2$ by $\exp(-2i\pi t)$. Denoting by $(p_i,q_i)$ the signature of the Hermitian form on $\V(S\setminus \gamma,2,2)$ and reminding that as $\epsilon_1=-1$  this signature is the opposite of the signature of $\V(S)_2\subset \V(S)$, we find that $$\rho^t(R_i)=(-q_i,-p_i)\in \Z\oplus \Z=\pi_1 U(\V(S,P))$$
which gives
$$\tau(R_i)=\frac{2}{p+q}(p_ip-q_iq).$$

\subsubsection{2-chain}
Take $\alpha,\beta$ two simple cirves intersecting once and denote by $\gamma$ the boundary of a tubular neighborhood $T$ of $\alpha\cup\beta$. Using MF2, we can decompose along $\gamma$ and write
$$\V(S)=\bigoplus_{\epsilon\in \Lambda}\V(T,\epsilon)\otimes \V(S^2,\epsilon,\epsilon)\otimes \V(S\setminus T,\epsilon).$$

The middle factor just serves for adjusting the sign and the right factor is inert as $T_\alpha$ and $T_\beta$ only act on the first factor. 

Consider first the factor $\V(T,2)$ which is $1$-dimensional and has negative sign. We find that $T_\alpha,T_\beta$ and $T_\gamma$ act by $q^{-1}$ hence $\rho^t(C_2)$ acts by $e^{-11\cdot 2i\pi t/5}$. 

We cancel this factor by multiplying that action of $\rho^t(C_2)$ on the factor $\V(T,0)$ by $e^{11\cdot 2i\pi t/5}$.

As $\V(T,0)$ has signature $(2,0)$, it suffices to compute the determinant $\det \rho^t((\alpha\beta)^6\gamma^{-1})=\exp(-12\cdot 2i\pi t/5)$. With the compensation $e^{22\cdot 2i\pi/5}$ we get that $\rho^t(C_2)=2\in \pi_1 U(2)$. Denoting by $(p_{g-1,n},q_{g-1,n})$ the signature of $\V(S\setminus T,0)$, we get $\rho^t(C_2)=(2p_{g-1,n},2q_{g-1,n})\in \pi_1 U(\V(S))$ hence 
$$\tau(C_2)=\frac{4}{p_{g,n}+q_{g,n}}(p_{g-1,n}q_{g,n}-q_{g-1,n}p_{g,n})$$

\subsubsection{Lantern}
Fix a 4 times punctured sphere $\Sigma\subset S\setminus P$: recall that we defined $\rho^t(L_\Sigma)=\rho^t(T_\zeta T_\eta T_\theta(T_\alpha T_\beta T_\gamma T_\delta)^{-1})$. 
We first decompose (forgetting the sign factor)
$$\V(S,P)=\bigoplus_{\epsilon\in \Lambda^4}\V(\Sigma,\epsilon)\otimes \V(S\setminus\Sigma,P,\epsilon)$$
Again one can write $\rho(L_\Sigma)=\sum_{\epsilon} \rho^t(L_\Sigma)_\epsilon\otimes \id$ and we are reduced to compute the terms
$\rho^t(L_\Sigma)_\epsilon$ individually. 
We compute directly 
$$\rho(L_\Sigma)_{0000}=1, \rho(L_\Sigma)_{0002}=0, \rho^t(L_\Sigma)_{0022}=(e^{-2i\pi t/5})^2(e^{-2i\pi t/5})^{-2}=1,$$
$$\rho^t(L_\Sigma)_{0222}=(e^{-2i\pi t/5})^3(e^{-2i\pi t/5})^{-3}=1.$$ 
Hence, the unique non-trivial contribution is 

$$\rho^t (L_\Sigma)_{2^4}=(x,y)\in \Z\oplus\Z=\pi_1 U(1,1).$$ 

We have $\det(\rho^t(L_\Sigma)_{1^4})=(e^{-2i\pi t/5})^3(e^{-2i\pi t/5})^{-8}=e^{2i\pi t}$. From this we get $x+y=1$. To compute $x-y$ we analyze $\rho^t(L_\Sigma)_{1111}$ in $\pi_1 \PU(1,1)$. The boundary factors do not contribute and we can deform the centers of $\rho(T_\zeta),\rho(T_\eta),\rho(T_\theta)$ until they coincide, yielding three time the matrix with diagonal entries $1,e^{-2i\pi t/3}$. This gives $x-y=1$ hence $(x,y)=(1,0)$. 
Denote by $(p_\Sigma,q_\Sigma)$ the signature of the space $\V(S\setminus \Sigma,2^4)$. The inclusion $U(1,1)\to U(p_\Sigma+q_\Sigma,p_\Sigma+q_\Sigma)$ maps $(1,0)$ to $(p_\Sigma,q_\Sigma)$. Pushing it to $\U(p,q)$ gives 
$$\tau(L_\Sigma)=\frac{2}{p+q}(p_\Sigma q-q_\Sigma p)$$

\subsubsection{Solving the linear system}

Let $\mu=(\mu_i)_{i\in \{1,\ldots,n\}\amalg I}$ be a collection of rational numbers indexed by the topological types of simple curves in $S_{g,n}$. We can think of $\mu$ as a $\Mod(S_{g,n})$-invariant linear map $H_1(\Gamma,\Z)\to \Q$, and composing with the natural map $R/[\Gamma,R]\to H_1(\Gamma,\Z)$ as a map $\mu:R/[\Gamma,R]\to \Q$. This map sends a relation $r$ to the sum of the signed values of $\mu$ on the Dehn twists contained in $r$. 
In order to compute the cohomology class $\tau$ in terms of the standard generators of $\Mb_{g,n}$, we need to find $a,b,c,\mu$ such that 

$$\tau = a \lambda_1+\sum_{i\in I} b_i \delta_i+\sum_{j=1}^n c_j \psi_j + \mu.$$
For symmetry reasons, $c_j$ is independent of $j$ so that we write it $c_j=c$.
We first observe that all classes vanish on the relations $D_{\alpha,\beta}$ and $B_{\alpha,\beta}$ so that they give no information. Next, as the classes $\lambda_1,\delta_i,\psi_x$ vanish on the lantern relation, we get $\tau(L_\Sigma)=\mu(L_\Sigma)$. This gives a linear system allowing to compute $\mu$.

\begin{Lemma}\label{lem:crible}
Let $\sigma,d:I\to \Z$ be defined respectively by

$$\sigma_i=\sign \V(S\setminus\gamma_i,2,2)\text{ and }d_i=\dim \V(S\setminus\gamma_i,2,2)$$
where $\gamma_i$ is a simple curve representing the topological type $i\in I$. For $j\in \{1,\ldots,n\}$ we set $\sigma_j=-\sigma_{g,n}$ and $d_j=d_{g,n}$. Then, $$\sigma(L_\Sigma)=3\sign \V(S\setminus\Sigma,P,2^4)\text{ and }d(L_\Sigma)=-5 \dim \V(S\setminus \Sigma,P,2^4).$$
\end{Lemma}
\begin{proof}
We skip the proof of this lemma that we obtained by brute force. It would be interesting to interpret the coefficients $3$ and $-5$ in terms of the Frobenius algebra structure. 
\end{proof}

We write $d_\Sigma=p_\Sigma+q_\Sigma$ and $\sigma_\Sigma=p_\Sigma-q_\Sigma$ so that $\tau(L_\Sigma)= \frac{q-p}{q+p}d_\Sigma+\sigma_\Sigma$. The solution for $\mu$ is then 
$$\mu=\frac{\sigma}{3}+\frac{d}{5}\cdot\frac{p-q}{p+q}.$$

This already gives the coefficient of the $\psi$ classes because $\tau(\partial_x)=0=-c_x+\mu_x$. As $\sigma_x=q-p$ and $d_x=p+q$ this gives 

$$c=-\frac{2}{15}\sigma_{g,n}.$$

Next, we compute the boundary divisors: $\tau(\frac{1}{5}R_i)=\frac{1}{5}\frac{p-q}{q+p}d_i+\frac{\sigma_i}{5}=b_i+\mu_i$ hence 

$$b_i=-\frac{2}{15}\sigma_i$$

Finally we compute the coefficient of $\lambda_1$ by evaluating the relation $C_2$: 
$$\tau(C_2)=2\sigma_{g-1,n}-2d_{g-1,n}\frac{\sigma_{g,n}}{d_{g,n}}=-a+12\mu_{\irr}-\mu_{1,\emptyset}$$
from which we get 
$$a=\big( -2\sigma_{g-1,n}+4\sigma_{\irr}-\frac{1}{3}\sigma_{1,\emptyset}\big)+\frac{\sigma_{g,n}}{d_{g,n}}\big(2d_{g-1,n}+\frac{12}{5}d_{\irr}-\frac{1}{5}d_{1\emptyset}\big)$$

Using the axiom MF2 we get the following relations removing $n$ from the notation:
$$\begin{cases} d_{\irr}+d_{g-1}=d_g\\ d_{1\emptyset}+2d_{g-1}=d_g\end{cases}\quad
\begin{cases} -\sigma_{\irr}+\sigma_{g-1}=\sigma_g\\-\sigma_{1\emptyset}+2\sigma_{g-1}=\sigma_g\end{cases}\quad \sigma_{g,n+1}=\sigma_{g,n}-3\sigma_{g-1,n}$$

This gives finally 
$$a=\frac{92}{15}\sigma_{g,n}+\frac{4}{3}\sigma_{g-1,n}=-\frac{46}{45}\sigma_{g,n}-\frac{4}{9}\sigma_{g,n+1}$$

We invite the reader to check that these formulas are compatible with the ones given in Section \ref{ss: R-matrix level 5} using the formula $\tilde{\kappa}_1=12\lambda_1-\delta$, see Section \ref{canonical_bundle}.

\Addresses
\end{document}